\documentclass[a4paper]{article}
\usepackage{arxiv}
\usepackage{comment}
\usepackage[utf8]{inputenc} 
\usepackage[T1]{fontenc}    
\usepackage{hyperref}       
\usepackage{url}            
\usepackage{booktabs}       
\usepackage{amsmath}
\usepackage{amssymb}
\usepackage{amsthm}
\usepackage{amsfonts}       
\usepackage{nicefrac}       
\usepackage{microtype}      
\usepackage{xcolor}
\usepackage{graphicx}
\usepackage{stmaryrd}
\usepackage{bm}
\usepackage[normalem]{ulem}
\usepackage[tickmarkheight=0.1cm,textsize=tiny]{todonotes}
\usepackage{algorithmicx}
\usepackage{algorithm}
\usepackage{algpseudocode}
\usepackage{caption}
\usepackage{subcaption}
\usepackage{lineno}
\usepackage{orcidlink}
\usepackage[inline,final]{showlabels}

\usepackage{multirow}
\usepackage{comment}
\usepackage{enumitem}

\hypersetup{
    colorlinks=true, 
    linkcolor=blue, 
    urlcolor=red, 
    citecolor=magenta,
    linktoc=all 
}

\newcommand{\ud}{\textrm{d}}

\newcommand{\Uad}{\mathcal{U}_{\textrm{ad}}}

\newtheorem{definition}{Definition}[section]
\newtheorem{lemma}{Lemma}[section]
\newtheorem{remark}{Remark}[section]
\newtheorem{theorem}{Theorem}[section]
\DeclareUnicodeCharacter{0301}{\hspace{-1ex}\'{ }} 

\definecolor{darkgreen}{RGB}{50, 200, 50}

\newcommand{\rev}[1]{{\color{black}#1\color{black}}} 
\newcommand{\reva}[1]{{\color{black}#1\color{black}}} 
\newcommand{\revb}[1]{{\color{black}#1\color{black}}} 
\newcommand{\revc}[1]{{\color{black}#1\color{black}}} 
\newcommand{\norev}[1]{{\color{black}#1}}

\newcommand{\mb}[1]{\boldsymbol{#1}}
\newcommand{\p}{\partial}

\newcommand{\timeavgflux}[1]{{#1}^{\text{LW}}}
\newcommand{\iniguessflux}[1]{{#1}^{\text{IG}}}

\title{Bound Preserving Lax-Wendroff Flux Reconstruction Method for Special Relativistic Hydrodynamics}

\author{
Sujoy Basak \orcidlink{0009-0009-0612-6361}\thanks{Corresponding author}\\
Department of Mathematics\\
Indian Institute of Technology Delhi\\
New Delhi -- 110016, India\\
\texttt{sujoybasak42@gmail.com} \\
\And
Arpit Babbar \orcidlink{0000-0002-9453-370X} \\
Centre for Applicable Mathematics\\
Tata Institute of Fundamental Research\\
Bangalore -- 560065, India\\
\texttt{arpit@tifrbng.res.in} \\
\And
Harish Kumar \orcidlink{0000-0003-4746-2336}\\
Department of Mathematics\\
Indian Institute of Technology Delhi\\
New Delhi -- 110016, India\\
\texttt{hkumar@iitd.ac.in} \\
\And
Praveen Chandrashekar \orcidlink{0000-0003-1903-4107}\\
Centre for Applicable Mathematics\\
Tata Institute of Fundamental Research\\
Bangalore -- 560065, India\\
\texttt{praveen@math.tifrbng.res.in}
}
\begin{document}
\maketitle
\begin{abstract}
Lax-Wendroff flux reconstruction (LWFR) schemes have high order of accuracy in both space and time despite having a single internal time step. Here, we design a Jacobian-free LWFR type scheme to solve the special relativistic hydrodynamics equations \revc{on Cartesian grids}. We then blend the scheme with a first-order finite volume scheme to control the oscillations near discontinuities. We also use a scaling limiter to preserve the physical admissibility of the solution after ensuring the scheme is admissible in means. A particular focus is given to designing a discontinuity indicator model to detect the local non-smoothness in the solution of the highly non-linear relativistic hydrodynamics equations. Finally, we present \rev{numerical results for a wide range of test cases to show the robustness and efficiency of the proposed scheme}.
\end{abstract}
\keywords{Relativistic hydrodynamics \and Lax-Wendroff flux reconstruction \and Admissibility preservation}
\section{Introduction}
The primary concern of this paper is to design a high-order, accurate, stable numerical method for special relativistic hydrodynamics (RHD) equations in one and two dimensions with the ideal equation of state. This system of equations belongs to the class of hyperbolic conservation laws. Often in high energy physics and astrophysics, for the investigation of astrophysical phenomena, one should take into account the behavior of the fluid when it approaches speed close to that of light or the effects of strong gravitational potentials, and hence, the relativistic description is necessary. Some examples are the formation of black holes, X-ray binaries, gamma-ray bursts, super-luminal jets, and active galactic nuclei~\cite{begelman1984theory,bottcher2012relativistic,mirabel1999sources,zensus1997parsec}.

Due to the relativistic influences, the non-linearity of the RHD equations is much stronger in the form of the Lorentz factor compared to that of the non-relativistic case, making it extremely difficult to treat analytically. Hence, we need to rely on numerical simulations for understanding the physical mechanisms of RHD equations in various applications. Some popular numerical schemes for hyperbolic partial differential equations are based on the Godunov approach, where the numerical fluxes on the cell interfaces are approximated by solving a Riemann problem and are then used in finite volume or discontinuous Galerkin (DG) methods. For the time update, suitable Runge-Kutta methods are used.
 
One of the first attempts to solve the RHD equations numerically was done by Wilson using an explicit finite difference method with the artificial viscosity technique for shock capturing in 1972~\cite{wilson1972numerical}, but it is not very accurate for cases having a value of Lorentz factor greater than 2~\cite{centrella1984planar}.
Following that, the numerical study of RHD equations did not attract much attention until the 1990s when various modern methods for capturing the shocks were developed with an approximate or exact Riemann solver~\cite{marti1991numerical,marti1994analytical,BALSARA1994,dai1997iterative,ibanez1999riemann}. Besides those, many higher-order methods are also designed, for example the piece-wise parabolic reconstruction method~\cite{marti1996extension,aloy1999genesis,mignone2005piecewise}, ENO (essentially non-oscillatory) and WENO (weighted ENO) methods~\cite{dolezal1995relativistic,del2002efficient,tchekhovskoy2007wham}, the discontinuous Galerkin method~\cite{radice2011discontinuous}, the RKDG (Runge-Kutta DG) method with WENO limiter~\cite{zhao2013runge}. In~\cite{font2008numerical}, the authors have conducted an examination of numerous numerical methods, analyzing how well they perform across a diverse set of problems.

For the solution of the RHD equations to be physically meaningful, we need the density and pressure to be positive and the absolute value of the velocity to be less than unity (assuming the speed of light to be unity), and this type of solution is said to be admissible. However, most of the above existing methods do not preserve the admissibility of the solution, even though they have been used to solve RHD equations in many cases successfully. Consequently, these methods can fail when used to solve RHD equations with low pressure or \rev{low} density or having fluid velocity near unity or with strong discontinuity, because the moment the solution gets out of the admissibility region, the Jacobian matrix of the flux function \revb{has imaginary} eigenvalues resulting in the ill-posedness of the problem. Often, this type of nonphysical solution is recalculated using very small CFL numbers and more diffusive methods until the solution becomes physically meaningful~\cite{zhang2006ram,hughes2002three}. This approach lacks scientific validity to some extent, thereby showing the importance of advanced high-order accurate numerical methods that produce solutions consistent with the fundamental physical constraints. In the last decade, there has been significant development in the bound-preserving numerical schemes for the RHD equations. Several high-order admissibility preserving schemes have been developed in~\cite{wu2015high,wu2016physical,qin2016bound,wu2017design} by analyzing the properties of the admissible set. Most of these schemes are based on the idea of using a scaling limiter on the solution polynomial evolved in the next time step or the flux-corrected limiting method. The primary technical challenge that arises in these methods is the non-availability of an explicit expression for primitive variables and flux function in terms of conservative variables, requiring us to solve a non-linear equation~\cite{schneider1993new,bhoriya2020entropy}.

Apart from the admissibility criteria, we also need to ensure the efficiency of the numerical schemes. Usually, for time updates when Runge-Kutta methods are used, each internal time step involves a transfer of data between different cores (in parallel implementation), which is the critical bottleneck for the performance, especially in the modern computational architecture involving GPUs. In~\cite{LW1}, the authors have described the Lax-Wendroff method, which requires only a single step to evolve in time for a second-order scheme, unlike the Runge-Kutta methods. Although the original method is unstable for the non-linear systems, in~\cite{LWDG1, LWDG2}, the authors have proposed a new scheme with arbitrary order of accuracy by discretizing the space using DG method and discretizing the time using the Lax-Wendroff method (typically denoted as LWDG). In~\cite{lou2020flux}, the authors have combined the Lax-Wendroff method with the flux reconstruction method, initially proposed in~\cite{huynh2007flux} by creating a continuous estimate of the flux utilizing numerical flux at the boundaries of cells and using a correction function. The quadrature-free computationally efficient flux reconstruction method is ideal for use in modern vector processors~\cite{vincent2016towards,lopez2014verification,vandenhoeck2019implicit}. However, the method described in~\cite{lou2020flux} is computationally expensive because it uses the traditional way of computing the flux derivatives by the chain rule of differentiation~\cite{burger2017approximate}. To overcome this issue, recently in~\cite{BABBAR2022111423}, the authors have given a Jacobian-free modified version of the LWFR method using the approximate Lax-Wendroff procedure~\cite{zorio2017approximate},
which requires significantly less number of computations.

The aim of this paper is to design a high-order LWFR method with an effective discontinuity indicator model for the RHD equations, which preserves the admissibility of the solution while minimizing the loss of accuracy.
In Section~\ref{sec: RHD}, we discuss some properties of the RHD equations that are necessary for the implementation of the scheme. Following the work in~\cite{BABBAR2022111423}, we design a higher-order LWFR method for the RHD equations in Section~\ref{sec:LWFR}. In Section~\ref{sec: oscillation_control}, we discuss about blending the scheme with a first-order finite volume scheme to control the oscillations near discontinuities with the help of a discontinuity indicator model designed in Section~\ref{sec:discontinuity_indicator}. In Section~\ref{sec: admissibility}, we discuss the admissibility region of the RHD equations and present a detailed process of making our scheme admissibility preserving following~\cite{babbar2024admissibility} and changing the admissibility region according to the framework.
We verify the efficiency and robustness of the scheme with several numerical simulations in one and two dimensions in Section~\ref{sec: Numerical_verifications}. Lastly, we summarise our work in Section~\ref{sec: summary}.

\section{Properties of the relativistic hydrodynamics equations} \label{sec: RHD}
In the context of an ideal fluid, one can express the $d$-dimensional (in space) relativistic hydrodynamics equations in the laboratory frame of reference as a set of conservation laws governing mass density, momentum density, and total energy density, which take the following form~\cite{anile2005relativistic,landau1987see,synge1965relativity}:
\begin{align}\label{RHD_equation}
    \frac{\p \mb{u}}{\p t}+\sum_{i=1}^d\frac{\p \mb{f}_i(\mb{u})}{\p x_i}=0,
\end{align}
where $\mb{u}$ is the vector of conserved variables and $\mb{f}_i$ is the flux vector in $x_i$ direction. Specifically, these vectors are defined by,
\begin{align}
\begin{split}
    \mb{u}&=(D, m_1,m_2,\dots,m_d, E)^\top,\\
    \mb{f}_i(\mb{u})&=(Dv_i, m_1v_i+p\delta_{1,i}, m_2v_i+p\delta_{2,i},\dots, m_dv_i+p\delta_{d,i}, m_i)^\top,\ \ \ i=1,\dots, d,
\end{split}
\end{align}
where the laboratory quantities, $D=\rho\Gamma$ is the relativistic density, $\mb{m}=(m_1,m_2,\dots,m_d)^\top$ is the momentum density vector and $E$ is the energy density given by $E={\rho h\Gamma^2}-p$. Here the primitive variables are the rest-mass density $\rho$, kinetic pressure $p$ and the fluid velocity $\mb{v}=(v_1, \dots, v_d)^\top$. The conservative variable $\mb{m}$ is expressed in terms of the primitive variables as $\mb{m}=\rho h \mb{v}\Gamma^2$. \revb{We adopt a scaling of the unknowns in which the speed of light becomes one;} then the Lorentz factor is given by $\Gamma=\frac{1}{\sqrt{1-|\mb{v}|^2}}$. The above equations can be closed using the equation of state for an ideal gas,
\begin{equation}\label{eq: eos}
    h=1+\frac{\gamma}{\gamma-1}\frac{p}{\rho},
\end{equation}
where $h$, $\gamma$ are the specific enthalpy and the specific heat ratio (also called adiabatic index) respectively, and under the compressibility assumptions we take $\gamma \in (1,2]$~\cite{cissoko1992detonation}. \revb{For the solution of the RHD equations to be physically meaningful, we need  the density and pressure to be positive and the magnitude of fluid velocity to be less than one, the speed of light; thus the solution $\mb{u}$ should belong to the admissible set,
\begin{equation}
\Uad =\{\mb{u}=(D,\mb{m},E)^\top: \rho (\mb{u}) >0, p(\mb{u})>0, |\mb{v}(\mb{u})|<1\}.
\end{equation}
Then we have the following result regarding hyperbolicity of the RHD equations.
\begin{lemma}
The system of RHD equations is hyperbolic when $\mb{u}$ belongs to the set $\Uad$. Thus, it has real eigenvalues with a complete set of corresponding eigenvectors~\cite{anile2005relativistic,ryu2006equation}.
\end{lemma}
The expressions of the eigenvalues and eigenvectors are available in~\cite{ryu2006equation,bhoriya2020entropy}.

Given a vector $\mb{u}$, it is expensive to check if it belongs to $\Uad$ since we have to first convert it to primitive variables. Moreover, some numerical methods that enforce admissibility make use of the fact that the functions $\rho(\mb{u})$, $p(\mb{u})$ are concave, but this is not the case for pressure in the RHD model. In~\cite{wu2015high}, an equivalent characterization of the admissible solution set is given in terms of the set,
\begin{equation} \label{eq:uad.defn}
\Uad' = \{\mb{u}=(D,\mb{m},E)^\top: D >0, q(\mb{u}):= E-\sqrt{D^2 + |\mb{m}|^2}>0\}
\end{equation}
and it is shown that,
\begin{lemma} \label{equivalent_admissibility_region}
The sets $\Uad$ and $\Uad'$ are equal. Moreover, both sets are convex.
\end{lemma}
\revc{\begin{proof}
See \cite{wu2015high}.
\end{proof}}

 This new admissible set has constraints that are directly computable from the conservative quantities; it does not require converting to primitive variables which makes it more efficient. Moreover, as we show below, the function $q(\mb{u})$ is concave which makes it easy to enforce the admissibility of the element means of the solution in the numerical scheme. Specifically, it helps in proving the admissibility of the low-order evolutions in Step 4 in Section~\ref{sec: admissibility} (see proof of Theorem~\ref{admissibility_theorem_fc}). Once we have the admissibility of the element means of the solution, concavity of the admissibility constraints again plays an important role in applying the scaling limiter from~\cite{zhang2010maximum}, the key idea for this is that we have a scheme which gives admissible element means of the solution, $\bar{\mb{u}}$; if the solution at some point $\mb{u}_{ij}$ in the element is not admissible, i.e., $q(\mb{u}_{ij}) \le 0$ or $D(\mb{u}_{ij}) \le 0$, then we push it closer to $\bar{\mb{u}}$, i.e., reset $\mb{u}_{ij} \leftarrow (1-\theta) \bar{\mb{u}} + \theta \mb{u}_{ij}$ for some $\theta \in [0,1]$. Since $q$ is concave, $q((1-\theta) \bar{\mb{u}} + \theta \mb{u}_{ij}) \ge (1-\theta) q(\bar{\mb{u}}) + \theta q(\mb{u}_{ij})$ and since $q(\bar{\mb{u}}) > 0$, the right hand side can be made non-negative by choosing $\theta$ small enough. A similar argument can be used for $D$ as well, since $D$ can also be considered as a concave function of the conservative variables.
 }
\begin{lemma} \label{q_concave}
    The function $q(\mb{u})$ is concave.
\end{lemma}
\begin{proof}
    Let, $\lambda\in [0,1]$ and $\mb{u}_a=(D_a, \mb{m}_a, E_a)$, $\mb{u}_b=(D_b, \mb{m}_b, E_b)$ be two points in $\Uad$.

    Now introducing the notation $(\cdot)_\lambda$ to denote $\lambda (\cdot)_a + (1-\lambda) (\cdot)_b$ we have,
    \begin{align*}
    q(\lambda \mb{u}_a + (1-\lambda) \mb{u}_b)
    &= E_\lambda - \sqrt{D_\lambda^2+|\mb{m}_\lambda|^2}\\
    &= E_\lambda - \sqrt{\big(\lambda D_a + (1-\lambda) D_b\big)^2 + \sum\limits_{i=1}^d \big[\lambda (m_i)_a + (1-\lambda) (m_i)_b]^2}\\
    &\geq E_\lambda - \lambda \sqrt{D_a^2+\sum\limits_{i=1}^d (m_i)_a^2}-(1-\lambda)\sqrt{D_b^2+\sum\limits_{i=1}^d (m_i)_b^2}\ , \hspace{1cm} \text{by Minkowski inequality}\\
    &= \lambda \Bigg(E_a- \sqrt{D_a^2+\sum\limits_{i=1}^d (m_i)_a^2}\Bigg) + (1-\lambda) \Bigg(E_b- \sqrt{D_b^2+\sum\limits_{i=1}^d (m_i)_a^2}\Bigg)\\
    &= \lambda q(\mb{u}_a) + (1-\lambda) q(\mb{u}_b).
    \end{align*}
    Hence, $q(\mb{u})$ is a concave function.
\end{proof}

\section{Lax-Wendroff flux-reconstruction scheme}\label{sec:LWFR}
Let us present the numerical scheme for two-dimensional RHD equations, which can easily be reduced to one dimension and generalized to higher dimensions. Without loss of generality, consider a computational domain $\Omega = [x_a,x_b]\times [y_a,y_b]$ that has been discretized into disjoint elements $\Omega_{pq} = [x_{p-\frac{1}{2}}, x_{p+\frac{1}{2}}]\times [y_{q-\frac{1}{2}}, y_{q+\frac{1}{2}}]$ and let the length of the intervals be given by $\Delta x_p=x_{p+\frac{1}{2}}-x_{p-\frac{1}{2}}$ and $\Delta y_q=y_{q+\frac{1}{2}}-y_{q-\frac{1}{2}}$.

We take the reference element to be $[0,1]\times [0,1]$, and the corresponding function to map a physical element to the reference element is given by,
\[
(x,y) \to (\xi, \eta) =\bigg(\frac{x-x_{p-\frac{1}{2}}}{\Delta x_p}, \frac{y-y_{q-\frac{1}{2}}}{\Delta y_q}\bigg). 
\]
To approximate the solution inside each element with a degree-$N$ polynomial in each direction, we take $(N+1)\times (N+1)$ solution points inside the reference element, $(\xi_i,\eta_j)$ where $i,j=0,1,\dots,N$ with $0\leq \xi_0< \xi_1 \cdots < \xi_N\leq 1$ and $0\leq \eta_0< \eta_1 \cdots < \eta_N\leq 1$. We will take these points as the Gauss-Legendre nodes, which give exact results in the quadrature rule for any polynomial of degree at most $2N+1$ with the corresponding Gauss-Legendre weights.

The solution inside an element will be given in reference coordinates as,
\[
\mb{u}_h (\xi,\eta, t) = \sum^N_{i,j=0} \mb{u}^e_{ij}(t) \ell_i(\xi)\ell_j(\eta),
\]
where $\ell_i$'s are the Lagrange polynomials of degree $N$ and $\mb{u}^e_{ij}$'s are the unknowns we need to find, called the degrees of freedom. They are the solution values at the solution points inside an element $\Omega_e=\Omega_{pq}$.

By formally expanding the solution, $\mb{u}^e(t)$ around $t=t_n$ using Taylor's series and using the conservation law (\ref{RHD_equation}) for dimension $d=2$ we get,
\begin{align}\label{eq: Taylor_exp_sol}
    (\mb{u}^e)^{n+1}= (\mb{u}^e)^n - \Delta t \bigg[\frac{\p \mb{F}\big((\mb{u}^e)^n\big)}{\p x} +\frac{\p \mb{G}\big((\mb{u}^e)^n\big)}{\p y}\bigg] + O\big(\Delta t^{N+2}\big),
\end{align}
where
\begin{equation}\label{timeavgflux}
    \mb{F}(\mb{u}) = \sum^{N}_{m=0}\frac{\Delta t^m}{(m+1)!}\frac{\p^m \mb{f}(\mb{u})}{\p t^m},\ \ \ \ \mb{G}(\mb{u}) = \sum^{N}_{m=0}\frac{\Delta t^m}{(m+1)!}\frac{\p^m \mb{g}(\mb{u})}{\p t^m}.
\end{equation}
Here \rev{and in the rest of the paper}, we have represented the fluxes $\mb{f}_1, \mb{f}_2$ as $\mb{f},\mb{g}$ for the sake of notational simplicity. The quantities $\mb{F}(\mb{u})$ and $\mb{G}(\mb{u})$ are called the time average fluxes, as they can be expressed as the time averages of the truncated Taylor's series expansions of the original fluxes in the time interval $[t^n,t^{n+1}]$. Reconstructing the time average fluxes by continuous functions $\mb{F}_h, \mb{G}_h$ in their respective space variables and truncating the equation~\eqref{eq: Taylor_exp_sol}, we get the single stage update of LWFR scheme at the solution points as,
\begin{equation}\label{LWupdate}
    (\mb{u}^e_{ij})^{n+1} = (\mb{u}^e_{ij})^n - \frac{\Delta t}{\Delta x_p}\frac{d \mb{F}_h}{d \xi}(\xi_i, \eta_j) - \frac{\Delta t}{\Delta y_q}\frac{d \mb{G}_h}{d \eta}(\xi_i,\eta_i), \ \ 0\leq i,j\leq N,
\end{equation}
where
\begin{align}\label{LWupdate.fluxes}
\begin{split}
\mb{F}_h(\xi,\eta_j) &= \Big[\mb{F}_{p-\frac{1}{2}}(\eta_j) - \mb{F}_h^\delta(0,\eta_j) \Big] \revb{c_L(\xi)} + \mb{F}_h^\delta(\xi,\eta_j) + \Big[\mb{F}_{p+\frac{1}{2}}(\eta_j) - \mb{F}_h^\delta(1,\eta_j) \Big] \revb{c_R(\xi)},\\
\mb{G}_h(\xi_i,\eta) &= \Big[\mb{G}_{q-\frac{1}{2}}(\xi_i) - \mb{G}_h^\delta(\xi_i,0) \Big] \revb{c_L(\eta)} + \mb{G}_h^\delta(\xi_i,\eta) + \Big[\mb{G}_{q+\frac{1}{2}}(\xi_i) - \mb{G}_h^\delta(\xi_i,1) \Big] \revb{c_R(\eta)}.
\end{split}
\end{align}
In the above expressions, $\mb{F}_h^\delta$ and $\mb{G}_h^\delta$ are local estimates of the time average fluxes $\mb{F}$ and $\mb{G}$ which may be discontinuous at element boundaries, and are given by, 
\[
\mb{F}_h^\delta(\xi,\eta_j) = \sum_{i=0}^N \mb{F}_{ij} \ell_i(\xi)
\qquad \textrm{and} \qquad \mb{G}_h^\delta(\xi_i,\eta)=\sum_{j=0}^N \mb{G}_{ij} \ell_j(\eta),
\]
where the approximations of the time average fluxes at the solution points, $\mb{F}_{ij} \approx \mb{F}(\xi_i,\eta_j)$ and $\mb{G}_{ij}\approx \mb{G}(\xi_i,\eta_j)$ for $i,j=1,2,\dots,N$ are found from the approximate Lax-Wendroff procedure as will be explained in Section \ref{subsec: Time_avg_flux}. Again, $\mb{F}_{p\pm\frac{1}{2}}$, $\mb{G}_{q\pm\frac{1}{2}}$ are some numerical fluxes which approximate the fluxes $\mb{F}$, $\mb{G}$ at faces $p\pm\frac{1}{2}$ and $q\pm\frac{1}{2}$ respectively. 

The correction functions \revb{$c_L,c_R$} play an important role in making the time average fluxes continuous; they must satisfy, 
\[
\revb{c_L(0)}=1, \qquad \revb{c_L(1)}=0, \qquad \revb{c_R(0)}=0, \qquad \revb{c_R(1)}=1
\]
for the continuity of the flux across the whole domain, and again, it should have a minimal effect inside the element far from the boundaries. The choice of the correction function can affect the accuracy and stability of the scheme~\cite{huynh2007flux, vincent2011new}. Here we take the Radau correction function, which is a polynomial of degree $N+1$, which belongs to the family of the correction functions studied in~\cite{vincent2011new}. This correction function, when used with the Gauss-Legendre nodes, makes the flux reconstruction scheme equivalent to the nodal DG scheme with the Gauss-Legendre nodes as the solution points and quadrature points. This correction function is given by,
\begin{equation}
\revb{c_L(\nu)} = \frac{(-1)^N}{2}[P_N(2\nu -1) - P_{N+1}(2\nu - 1)],\
\ \ \ \revb{c_R(\nu)} = \frac{1}{2}[P_N(2\nu -1) + P_{N+1}(2\nu - 1)],
\end{equation}
where $P_N:[-1,1]\to \mathbb{R}$ is the degree-$N$ Legendre polynomial.

\subsection{Time average flux at the solution points} \label{subsec: Time_avg_flux}
Computing time average fluxes from (\ref{timeavgflux}), by replacing the temporal derivatives with the spatial derivatives directly using the chain rule would require a large number of algebraic computations, and so following~\cite{zorio2017approximate, BABBAR2022111423}, we first approximate the temporal derivatives with some finite difference formulae.

From (\ref{timeavgflux}), we get the following expression for the time average flux inside each element for the $2$-order scheme ($N=1$),
\begin{align*}
\mb{F}(\xi_i, \eta_j)&=\mb{f}(\mb{u}_{ij})+\frac{\Delta t}{2}\frac{\p \mb{f}(\mb{u}_{ij})}{\p t}\\
&\approx \mb{f}(\mb{u}_{ij}) + \frac{\Delta t}{2}\frac{\mb{f}\big(\mb{u}(\xi_i,\eta_j, t+\Delta t)\big)-\mb{f}\big(\mb{u}(\xi_i,\eta_j,t-\Delta t)\big)}{2 \Delta t}\\
&\approx \mb{f}(\mb{u}_{ij}) + \frac{\mb{f}\big(\mb{u}(\xi_i,\eta_j, t) + \mb{u}_t(\xi_i,\eta_j, t) \Delta t\big)-\mb{f}\big(\mb{u}(\xi_i,\eta_j, t) - \mb{u}_t(\xi_i,\eta_j, t) \Delta t\big)}{4}\\
&= \mb{f}(\mb{u}_{ij}) + \frac{1}{4}\bigg[\mb{f}\big(\mb{u}_{ij}+\Delta t (\mb{u}_{ij})_t\big)- \mb{f}\big(\mb{u}_{ij}-\Delta t (\mb{u}_{ij})_t\big)\bigg]
\end{align*}
where $(\mb{u}_{ij})_t$ can be computed as
\rev{
\[
(\mb{u}_{ij})_t=-\frac{1}{\Delta x_p}\frac{\p \mb{f}\big(\mb{u}\big)}{\p \xi}(\xi_i,\eta_j)-\frac{1}{\Delta y_q}\frac{\p \mb{g}\big(\mb{u}\big)}{\p \eta}(\xi_i,\eta_j).
\]
}
Introducing the notations, 
\revb{
\begin{equation}
\mb{f}^{(m)} = \Delta t^m \frac{\p^m \mb{f}}{\p t^m},\qquad \mb{u}^{(m)} = \Delta t^m \frac{\p^m \mb{u}}{\p t^m}, \qquad m=1,2,\ldots
\end{equation}
}
\rev{the approximation of the time average flux, $\mb{F}$ at solution point $(\xi_i, \eta_j)$ can be expressed in a simplified form as,}
\[
\mb{F}_{ij}=\mb{f}_{ij}+\frac{1}{2}\mb{f}_{ij}^{(1)},
\]
with
\begin{equation}
    \mb{u}_{ij}^{(1)} = -\frac{\Delta t}{\Delta x_p} \mb{f}_{ij,\xi} - \frac{\Delta t}{\Delta y_q} \mb{g}_{ij,\eta}, \qquad
    \mb{f}_{ij}^{(1)} = \frac{1}{2}\Big[\mb{f}\Big(\mb{u}_{ij}+\mb{u}_{ij}^{(1)}\Big)-\mb{f}\Big(\mb{u}_{ij}-\mb{u}_{ij}^{(1)}\Big) \Big]
\end{equation}
Here $\mb{f}_{ij,\xi},\ \mb{g}_{ij,\eta}$ are found by taking the derivatives of the polynomial approximations of the fluxes $\mb{f}(\mb{u}),\ \mb{g}(\mb{u})$ \rev{with respect to $\xi,\ \eta$ respectively at the solution point $(\xi_i,\eta_j)$}.

Similarly, defining $\mb{g}^{(m)} = \Delta t^m \frac{\p^m \mb{g}}{\p t^m}$ \rev{and taking $\mb{G}_{ij}\approx \mb{G}(\xi_i,\eta_j)$} we get,
\begin{equation}
\mb{G}_{ij}=\mb{g}_{ij}+\frac{1}{2}\mb{g}_{ij}^{(1)} \qquad \textrm{with} \qquad
    \mb{g}_{ij}^{(1)} = \frac{1}{2}\Big[\mb{g}\Big(\mb{u}_{ij}+\mb{u}_{ij}^{(1)}\Big)-\mb{g}\Big(\mb{u}_{ij}-\mb{u}_{ij}^{(1)}\Big) \Big].
\end{equation}
 Now, doing the same process and applying the finite difference formulae according to the order of the scheme, we get the approximations of the time average fluxes for the higher-order schemes as follows:
 \begin{equation}
    \mb{F}_{ij}=\sum\limits_{r=1}^{N} \frac{1}{(r+1)!}\mb{f}_{ij}^{(r)}, \qquad
    \mb{G}_{ij}=\sum\limits_{r=1}^{N} \frac{1}{(r+1)!}\mb{g}_{ij}^{(r)},
\end{equation}
for the scheme of order $N+1$. The terms in these expressions are calculated according to the order of the scheme.
 
For the $3$-order scheme ($N=2$),
\begin{align}
\begin{split}
    \mb{f}_{ij}^{(1)} &= \frac{1}{2}\left[\mb{f}\left(\sum\limits_{k=0}^1 \frac{1}{k!}\mb{u}_{ij}^{(k)}\right) - \mb{f}\left(\sum\limits_{k=0}^1 \frac{(-1)^2}{k!}\mb{u}_{ij}^{(k)}\right) \right],\\
    \mb{g}_{ij}^{(1)} &= \frac{1}{2}\left[\mb{g}\left(\sum\limits_{k=0}^1 \frac{1}{k!}\mb{u}_{ij}^{(k)}\right) - \mb{g}\left(\sum\limits_{k=0}^1 \frac{(-1)^2}{k!}\mb{u}_{ij}^{(k)}\right) \right],\\
    \mb{f}_{ij}^{(2)} &= \mb{f}\left(\sum\limits_{k=0}^2 \frac{1}{k!}\mb{u}_{ij}^{(k)}\right) - 2 \mb{f}(\mb{u}_{ij}) + \mb{f}\left(\sum\limits_{k=0}^2 \frac{(-1)^2}{k!}\mb{u}_{ij}^{(k)}\right),\\
    \mb{g}_{ij}^{(2)} &= \mb{g}\left(\sum\limits_{k=0}^2 \frac{1}{k!}\mb{u}_{ij}^{(k)}\right) - 2 \mb{g}(\mb{u}_{ij}) + \mb{g}\left(\sum\limits_{k=0}^2 \frac{(-1)^2}{k!}\mb{u}_{ij}^{(k)}\right),
\end{split}
\end{align}
where
\begin{align*}
    \mb{u}_{ij}^{(m)} = -\frac{\Delta t}{\Delta x_p} \mb{f}_{ij,\xi}^{(m-1)} - \frac{\Delta t}{\Delta y_q} \mb{g}_{ij,\eta}^{(m-1)},\ \ \ \ \text{for}\ m=1,2.
\end{align*}

Similarly, for the $4$-order scheme ($N=3$),
\begin{align}
\begin{split}
    \mb{f}_{ij}^{(1)} &= \frac{1}{12}\left[-\mb{f}\left(\sum\limits_{k=0}^1 \frac{2^k}{k!}\mb{u}_{ij}^{(k)}\right) + 8\mb{f}\left(\sum\limits_{k=0}^1 \frac{1}{k!}\mb{u}_{ij}^{(k)}\right) - 8\mb{f}\left(\sum\limits_{k=0}^1 \frac{(-1)^k}{k!}\mb{u}_{ij}^{(k)}\right) + \mb{f}\left(\sum\limits_{k=0}^1 \frac{(-2)^k}{k!}\mb{u}_{ij}^{(k)}\right)\right],\\
    \mb{g}_{ij}^{(1)} &= \frac{1}{12}\left[-\mb{g}\left(\sum\limits_{k=0}^1 \frac{2^k}{k!}\mb{u}_{ij}^{(k)}\right) + 8\mb{g}\left(\sum\limits_{k=0}^1 \frac{1}{k!}\mb{u}_{ij}^{(k)}\right) - 8\mb{g}\left(\sum\limits_{k=0}^1 \frac{(-1)^k}{k!}\mb{u}_{ij}^{(k)}\right) + \mb{g}\left(\sum\limits_{k=0}^1 \frac{(-2)^k}{k!}\mb{u}_{ij}^{(k)}\right)\right],\\
    \mb{f}_{ij}^{(2)} &= \mb{f}\left(\sum\limits_{k=0}^2 \frac{1}{k!}\mb{u}_{ij}^{(k)}\right) - 2 \mb{f}(\mb{u}_{ij}) + \mb{f}\left(\sum\limits_{k=0}^2 \frac{(-1)^2}{k!}\mb{u}_{ij}^{(k)}\right),\\
    \mb{g}_{ij}^{(2)} &= \mb{g}\left(\sum\limits_{k=0}^2 \frac{1}{k!}\mb{u}_{ij}^{(k)}\right) - 2 \mb{g}(\mb{u}_{ij}) + \mb{g}\left(\sum\limits_{k=0}^2 \frac{(-1)^2}{k!}\mb{u}_{ij}^{(k)}\right),\\
    \mb{f}_{ij}^{(3)} &= \frac{1}{2}\left[\mb{f}\left(\sum\limits_{k=0}^3 \frac{2^k}{k!}\mb{u}_{ij}^{(k)}\right) - 2\mb{f}\left(\sum\limits_{k=0}^3 \frac{1}{k!}\mb{u}_{ij}^{(k)}\right) + 2\mb{f}\left(\sum\limits_{k=0}^3 \frac{(-1)^k}{k!}\mb{u}_{ij}^{(k)}\right) - \mb{f}\left(\sum\limits_{k=0}^3 \frac{(-2)^k}{k!}\mb{u}_{ij}^{(k)}\right) \right],\\
    \mb{g}_{ij}^{(3)} &= \frac{1}{2}\left[\mb{g}\left(\sum\limits_{k=0}^3 \frac{2^k}{k!}\mb{u}_{ij}^{(k)}\right) - 2\mb{g}\left(\sum\limits_{k=0}^3 \frac{1}{k!}\mb{u}_{ij}^{(k)}\right) + 2\mb{g}\left(\sum\limits_{k=0}^3 \frac{(-1)^k}{k!}\mb{u}_{ij}^{(k)}\right) - \mb{g}\left(\sum\limits_{k=0}^3 \frac{(-2)^k}{k!}\mb{u}_{ij}^{(k)}\right) \right],
\end{split}
\end{align}
where
\begin{align*}
    \mb{u}_{ij}^{(m)} = -\frac{\Delta t}{\Delta x_p} \mb{f}_{ij,\xi}^{(m-1)} - \frac{\Delta t}{\Delta y_q} \mb{g}_{ij,\eta}^{(m-1)},\ \ \ \ \text{for}\ m=1,2,3.
\end{align*}

For the $5$-order scheme ($N=4$),
\begin{align}
\begin{split}
    \mb{f}_{ij}^{(1)} &= \frac{1}{12}\left[-\mb{f}\left(\sum\limits_{k=0}^1 \frac{2^k}{k!}\mb{u}_{ij}^{(k)}\right) + 8\mb{f}\left(\sum\limits_{k=0}^1 \frac{1}{k!}\mb{u}_{ij}^{(k)}\right) - 8\mb{f}\left(\sum\limits_{k=0}^1 \frac{(-1)^k}{k!}\mb{u}_{ij}^{(k)}\right) + \mb{f}\left(\sum\limits_{k=0}^1 \frac{(-2)^k}{k!}\mb{u}_{ij}^{(k)}\right)\right],\\
    \mb{g}_{ij}^{(1)} &= \frac{1}{12}\left[-\mb{g}\left(\sum\limits_{k=0}^1 \frac{2^k}{k!}\mb{u}_{ij}^{(k)}\right) + 8\mb{g}\left(\sum\limits_{k=0}^1 \frac{1}{k!}\mb{u}_{ij}^{(k)}\right) - 8\mb{g}\left(\sum\limits_{k=0}^1 \frac{(-1)^k}{k!}\mb{u}_{ij}^{(k)}\right) + \mb{g}\left(\sum\limits_{k=0}^1 \frac{(-2)^k}{k!}\mb{u}_{ij}^{(k)}\right)\right],\\
    \mb{f}_{ij}^{(2)} &= \frac{1}{12}\Bigg[-\mb{f}\left(\sum\limits_{k=0}^2 \frac{2^k}{k!}\mb{u}_{ij}^{(k)}\right) + 16\mb{f}\left(\sum\limits_{k=0}^2 \frac{1}{k!}\mb{u}_{ij}^{(k)}\right) - 30\mb{f}(\mb{u}_{ij}) + 16\mb{f}\left(\sum\limits_{k=0}^2 \frac{(-1)^k}{k!}\mb{u}_{ij}^{(k)}\right) \\
    &\hspace{30em} - \mb{f}\left(\sum\limits_{k=0}^2 \frac{(-2)^k}{k!}\mb{u}_{ij}^{(k)}\right) \Bigg],\\
    \mb{g}_{ij}^{(2)} &= \frac{1}{12}\Bigg[-\mb{g}\left(\sum\limits_{k=0}^2 \frac{2^k}{k!}\mb{u}_{ij}^{(k)}\right) + 16\mb{g}\left(\sum\limits_{k=0}^2 \frac{1}{k!}\mb{u}_{ij}^{(k)}\right) - 30\mb{g}(\mb{u}_{ij}) + 16\mb{g}\left(\sum\limits_{k=0}^2 \frac{(-1)^k}{k!}\mb{u}_{ij}^{(k)}\right) \\
    &\hspace{30em} - \mb{g}\left(\sum\limits_{k=0}^2 \frac{(-2)^k}{k!}\mb{u}_{ij}^{(k)}\right) \Bigg],\\
    \mb{f}_{ij}^{(3)} &= \frac{1}{2}\left[\mb{f}\left(\sum\limits_{k=0}^3 \frac{2^k}{k!}\mb{u}_{ij}^{(k)}\right) - 2\mb{f}\left(\sum\limits_{k=0}^3 \frac{1}{k!}\mb{u}_{ij}^{(k)}\right) + 2\mb{f}\left(\sum\limits_{k=0}^3 \frac{(-1)^k}{k!}\mb{u}_{ij}^{(k)}\right) - \mb{f}\left(\sum\limits_{k=0}^3 \frac{(-2)^k}{k!}\mb{u}_{ij}^{(k)}\right) \right],\\
    \mb{g}_{ij}^{(3)} &= \frac{1}{2}\left[\mb{g}\left(\sum\limits_{k=0}^3 \frac{2^k}{k!}\mb{u}_{ij}^{(k)}\right) - 2\mb{g}\left(\sum\limits_{k=0}^3 \frac{1}{k!}\mb{u}_{ij}^{(k)}\right) + 2\mb{g}\left(\sum\limits_{k=0}^3 \frac{(-1)^k}{k!}\mb{u}_{ij}^{(k)}\right) - \mb{g}\left(\sum\limits_{k=0}^3 \frac{(-2)^k}{k!}\mb{u}_{ij}^{(k)}\right) \right],\\
    \mb{f}_{ij}^{(4)} &= \mb{f}\left(\sum\limits_{k=0}^4 \frac{2^k}{k!}\mb{u}_{ij}^{(k)}\right) - 4\mb{f}\left(\sum\limits_{k=0}^4 \frac{1}{k!}\mb{u}_{ij}^{(k)}\right) + 6\mb{f}(\mb{u}_{ij}) - 4\mb{f}\left(\sum\limits_{k=0}^4 \frac{(-1)^k}{k!}\mb{u}_{ij}^{(k)}\right) + \mb{f}\left(\sum\limits_{k=0}^4 \frac{(-2)^k}{k!}\mb{u}_{ij}^{(k)}\right),\\
    \mb{g}_{ij}^{(4)} &= \mb{g}\left(\sum\limits_{k=0}^4 \frac{2^k}{k!}\mb{u}_{ij}^{(k)}\right) - 4\mb{g}\left(\sum\limits_{k=0}^4 \frac{1}{k!}\mb{u}_{ij}^{(k)}\right) + 6\mb{g}(\mb{u}_{ij}) - 4\mb{f}\left(\sum\limits_{k=0}^4 \frac{(-1)^k}{k!}\mb{u}_{ij}^{(k)}\right) + \mb{g}\left(\sum\limits_{k=0}^4 \frac{(-2)^k}{k!}\mb{u}_{ij}^{(k)}\right),
    \end{split}
\end{align}
where
\begin{align*}
    \mb{u}_{ij}^{(m)} = -\frac{\Delta t}{\Delta x_p} \mb{f}_{ij,\xi}^{(m-1)} - \frac{\Delta t}{\Delta y_q} \mb{g}_{ij,\eta}^{(m-1)},\ \ \ \ \text{for}\ m=1,2,3,4.
\end{align*}

We can continue in a similar way to explain this method for higher-order schemes.

\subsection{Numerical flux}\label{sec: Numerical_flux}
There are a variety of numerical fluxes available in the literature. In this work, we use a Rusanov type flux~\cite{rusanov1962calculation}, which only requires an estimate of local wave speed and no other characteristic information. For the time average flux, it is given by,
\begin{align}
    \mb{F}_{p+\frac{1}{2}}(\eta_j) &= \frac{1}{2}\big[\mb{F}^l_{p+\frac{1}{2}}(\eta_j) + \mb{F}^r_{p+\frac{1}{2}}(\eta_j) \big] - \frac{1}{2}\lambda_{p+\frac{1}{2}}\big[\mb{U}^r_{p+\frac{1}{2}}(\eta_j) - \mb{U}^l_{p+\frac{1}{2}}(\eta_j) \big],\ \ \  0\leq j \leq N,\label{eq:numerical_flux_x}\\
    \mb{G}_{q+\frac{1}{2}}(\xi_i) &= \frac{1}{2}\big[\mb{G}^d_{q+\frac{1}{2}}(\xi_i) + \mb{G}^u_{q+\frac{1}{2}}(\xi_i) \big] - \frac{1}{2}\lambda_{q+\frac{1}{2}}\big[\mb{U}^u_{q+\frac{1}{2}}(\xi_i) - \mb{U}^d_{q+\frac{1}{2}}(\xi_i) \big],\ \ \  0\leq i \leq N, \label{eq:numerical_flux_y}
\end{align}
where superscripts $(\cdot)^l$, $(\cdot)^r$, $(\cdot)^d$ and $(\cdot)^u$ are used to denote the trace values of the respective quantities computed from the left, right, lower, and upper elements respectively and
\[
    \lambda_{p+\frac{1}{2}}=\max\big\{r\big(\mb{f}'(\Bar{\mb{u}}_{pq}^n)\big),r(\mb{f}'\big(\Bar{\mb{u}}_{p+1,q}^n)\big)\big\},\qquad    \lambda_{q+\frac{1}{2}}=\max\big\{r\big(\mb{g}'(\Bar{\mb{u}}_{pq}^n)\big),r(\mb{g}'\big(\Bar{\mb{u}}_{p,q+1}^n)\big)\big\}.
\]
Here, $\Bar{\mb{u}}_{pq}^n$, $\Bar{\mb{u}}_{p+1,q}^n$ and $\Bar{\mb{u}}_{p,q+1}^n$ are the averages of the solutions at time level $n$ in the elements $\Omega_{pq}$, $\Omega_{p+1,q}$ and $\Omega_{p,q+1}$ respectively with $r(\cdot)$ denoting the spectral radius. Note that, here $\mb{f}'$ and $\mb{g}'$ are denoting the flux Jacobians. In the dissipative part of the numerical fluxes we have used the time average solution,
\begin{equation}
\mb{U} = \sum^N_{m=0}\frac{\Delta t^m}{(m+1)!}\frac{\p^m \mb{u}}{\p t^m}
\end{equation}
instead of the solution at time level $n$ because it gives a stable scheme with comparatively higher CFL number~\cite{BABBAR2022111423}. \revc{More precisely, we use the trace values of the time average solutions at the faces, which are computed using the Lagrange polynomial extrapolation of the time average solutions at the solution points. The derivatives of $\mb{u}$ at the solution points are approximated as,
\begin{equation}
    \frac{\p^m \mb{u}}{\p t^m}(\xi_i, \eta_j) \approx \frac{1}{\Delta t^m}\mb{u}_{ij}^{(m)}
\end{equation}
Note that the quantities $\mb{u}_{ij}^{(m)}$ are already available during the approximate Lax-Wendroff procedure (Section~\ref{subsec: Time_avg_flux}), hence we do not need to recompute it.}

The natural way to find the trace values of the fluxes $\mb{F}^l_{p+\frac{1}{2}}$, $\mb{F}^r_{p+\frac{1}{2}}$, $\mb{G}^d_{q+\frac{1}{2}}$, and $\mb{G}^u_{q+\frac{1}{2}}$ is to extrapolate the time average fluxes $\mb{F}_h^\delta, \mb{G}_h^\delta$ to the respective boundaries of the elements. However, \revc{it is observed in~\cite{BABBAR2022111423} that this procedure can lead to suboptimal convergence rates (one order lower) for some nonlinear problems. Thus, we construct the time average flux at the element faces directly which gives optimal convergence rates. We first extrapolate the solution available in an element at the solution points to the corresponding faces by polynomial extrapolation with Lagrange polynomials and then compute the trace values of the fluxes there by the approximate Lax-Wendroff procedure of Section~\ref{subsec: Time_avg_flux}. The reader may refer to Section 5.2 of~\cite{BABBAR2022111423} for a detailed algorithm.
}

\subsection{Conservative nature of the scheme}
The Lax-Wendroff theorem says that a consistent and conservative numerical method for a system of hyperbolic conservation laws converges to a weak solution, provided it is convergent. So, the conservative nature of the scheme becomes an important factor. From (\ref{LWupdate}) we get,
\begin{align}
&\sum^N_{i,j=0}w_iw_j(\mb{u}^e_{ij})^{n+1} = \sum^N_{i,j=0}w_iw_j(\mb{u}^e_{ij})^n - \frac{\Delta t}{\Delta x_p}\sum^N_{i,j=0}w_iw_j\frac{d \mb{F}_h}{d \xi}(\xi_i, \eta_j) - \frac{\Delta t}{\Delta y_q}\sum^N_{i,j=0}w_iw_j\frac{d \mb{G}_h}{d \eta}(\xi_i,\eta_j),\nonumber\\
 &\hspace{37em} 0\leq i,j\leq N\nonumber\\
\implies &\iint\limits_{\Omega_{pq}} (\mb{u}_h)^{n+1} = \iint\limits_{\Omega_{pq}} (\mb{u}_h)^{n}- \Delta t \Bigg[\sum^{N}_{j=0}w_j\Big( \mb{F}_{p+\frac{1}{2}}(\eta_j)-\mb{F}_{p-\frac{1}{2}}(\eta_j)\Big) - \sum^{N}_{i=0}w_i\Big(\mb{G}_{q+\frac{1}{2}}(\xi_i)-\mb{G}_{q-\frac{1}{2}}(\xi_i)\Big)\Bigg],
\end{align}
which says that the total quantity of $\mb{u}$ inside an element changes with time solely because of the fluxes at the element faces, and hence the scheme preserves the conservative nature.

\subsection{Boundary conditions}
In the case of conservation laws, the boundary conditions are usually implemented weakly through the numerical fluxes. Here, we will explain the implementation of some common boundary conditions that are required in the numerical tests discussed later.

\subsubsection{Outflow boundary condition}
When the characteristics exit the domain through a boundary, it is called an outflow boundary, and the numerical flux at these types of boundaries is computed in an upwind manner from the interior solutions.

Suppose the left boundary $x=x_a$ is an outflow boundary. Then the numerical flux at this boundary is computed as $\mb{F}_{x_a}=\mb{F}^r_{x_a}$. The process for finding the trace value of the flux $\mb{F}^r_{x_a}$ is explained in Section \ref{sec: Numerical_flux}.

\subsubsection{Inflow boundary condition}
When the characteristics enter the domain through a boundary, it is called an inflow boundary, and we have a specified boundary value defined there. At these types of boundaries, we need to construct the numerical flux from the given boundary value. 

Suppose the left boundary $x=x_a$ is an inflow boundary, and we have the boundary value, \revb{$\mb{u}(x_a,t)=\mb{u}_a(t),\ t\geq 0$}, then the numerical flux at this boundary is computed as,
\begin{equation}
\mb{F}_{x_a}=\mb{F}_{x_a}^r = \frac{1}{\Delta t}\int_{t_n}^{t_{n+1}}\mb{f}\big( \revb{\mb{u}_a(t)} \big)\ud t.
\end{equation}
We approximate this integral using the degree-$N$ Gauss-Legendre quadrature, which ensures the required order of accuracy.

\subsubsection{Periodic boundary condition}\label{sec: Periodic Boundary Condition}
Suppose we have discretized the computational domain into the elements $\Omega_{pq}$ for $p=0,1,\dots,M_x$ and $q=0,1,\dots,M_y$.
The numerical fluxes at periodic boundaries are computed using equations \eqref{eq:numerical_flux_x} and \eqref{eq:numerical_flux_y} by introducing some ghost values, $\mb{F}^l_{-\frac{1}{2}}$, $\mb{U}^l_{-\frac{1}{2}}$, $\Bar{\mb{u}}_{-1,q}^n$ 
for the left boundary; $\mb{F}^r_{M_x + \frac{1}{2}}$, $\mb{U}^r_{M_x + \frac{1}{2}}$, $\Bar{\mb{u}}_{M_x + 1,q}^n$ for the right boundary; $\mb{G}^d_{-\frac{1}{2}}$, $\mb{U}^d_{-\frac{1}{2}}$, $\Bar{\mb{u}}_{p,-1}^n$ for the lower boundary; and $\mb{G}^u_{M_y + \frac{1}{2}}$, $\mb{U}^u_{M_y + \frac{1}{2}}$, $\Bar{\mb{u}}_{p, M_y + 1}^n$ for the upper boundary.

Suppose the problem has periodic boundary condition in $x$-direction, then we set the corresponding ghost values as,
\[
    \mb{F}^l_{-\frac{1}{2}}, \mb{F}^r_{M_x + \frac{1}{2}} = \mb{F}^l_{M_x+\frac{1}{2}}, \mb{F}^r_{- \frac{1}{2}},\qquad
    \mb{U}^l_{-\frac{1}{2}}, \mb{U}^r_{M_x + \frac{1}{2}} = \mb{U}^l_{M_x+\frac{1}{2}}, \mb{U}^r_{- \frac{1}{2}},\qquad
    \Bar{\mb{u}}_{-1,q}^n, \Bar{\mb{u}}_{M_x + 1, q}^n = \Bar{\mb{u}}_{M_x,q}^n, \Bar{\mb{u}}_{0, q}^n.
\]
Similarly, if the problem has periodic boundary condition in $y$-direction, then we set the respective ghost values as,
\[
    \mb{G}^d_{-\frac{1}{2}}, \mb{G}^u_{M_y + \frac{1}{2}} = \mb{G}^d_{M_y+\frac{1}{2}}, \mb{G}^u_{- \frac{1}{2}},\qquad
    \mb{U}^d_{-\frac{1}{2}}, \mb{U}^u_{M_y + \frac{1}{2}} = \mb{U}^d_{M_y+\frac{1}{2}}, \mb{U}^u_{- \frac{1}{2}},\qquad
    \Bar{\mb{u}}_{p,-1}^n, \Bar{\mb{u}}_{p,M_y + 1}^n = \Bar{\mb{u}}_{p,M_y}^n, \Bar{\mb{u}}_{p, 0}^n.
\]

\subsubsection{Reflective boundary condition}
These types of boundaries are also known as solid walls since the waves hitting these boundaries get reflected in the opposite direction with the same velocity, making the boundaries act as solid walls. These boundaries are also dealt with by introducing the ghost values as in Section \ref{sec: Periodic Boundary Condition}.

To explain how we update the ghost values, consider the reflective boundary condition at the left boundary, $x=x_a$. Then the corresponding fluid velocity, $v_1$ is assumed to be zero at this boundary, and hence we reflect the velocities in the ghost values in the normal direction to the boundary. More precisely, we calculate $\mb{F}^l_{-\frac{1}{2}}$, $\mb{U}^l_{-\frac{1}{2}}$ and $\Bar{\mb{u}}_{-1,q}^n$ from $\mb{F}^r_{-\frac{1}{2}}$, $\mb{U}^r_{-\frac{1}{2}}$ and $\Bar{\mb{u}}_{0,q}^n$ by reversing the sign in the velocity $v_1$ and keeping all the other quantities unchanged. Then, the numerical flux at the boundary is computed in the usual manner using equations \eqref{eq:numerical_flux_x} and \eqref{eq:numerical_flux_y}. 

\section{Controlling spurious oscillations}\label{sec: oscillation_control}
Higher-order methods are prone to produce oscillations when the solution is discontinuous, and hence limiters are needed in practice. In~\cite{BABBAR2022111423}, the authors have taken the idea of the TVD limiter~\cite{cockburn1991runge,cockburn1989tvb} to give a posteriori limiter for the LWFR framework. Basically, the idea is to limit the slope of the solution to control its oscillations. However, using this limiter results in losing a lot of information near the discontinuities as it replaces the solution with a linear or constant polynomial. It also sometimes detects the smooth extrema as discontinuities, leading to loss of accuracy, which is a significant drawback for this limiter.
The authors have then modified it to the TVB limiter following~\cite{cockburn1991runge}, which can fix this problem to some extent, but it requires a proper choice of some parameters that are problem-dependent. Further, the TVB limiter loses most of the sub-cell information, which leads to poor performance. This can be seen from Figure~\ref{DP_ind_limiter} in Section~\ref{fourthRP_section}. \reva{In~\cite{dumbser2014posteriori,dumbser2016simple} the authors have applied a sub-cell based approach in \'a posteriori fashion by dividing an element into $2N+1$ sub-cells for $N$-degree polynomial approximation of the solution. This idea is also implemented in relativistic regime in~\cite{zanotti2015solving} for solving the relativistic magnetohydrodynamics equations.} \rev{However, in this work, we use a priori blending limiter~\cite{babbar2024admissibility} which preserves sub-cell information \revc{and is inspired from the sub-cell limiter of~\cite{hennemann2021provably} for the discontinuous Galerkin method}; the particular choice of sub-cells gives a natural way to modify the time average numerical fluxes for admissibility preservation of the solution mean value.}

The basic idea \rev{of this blending limiter} is to blend the high-order scheme with a low-order scheme near the discontinuities,
\begin{equation}
(\mb{u}^e)^{n+1} = (1-\alpha_e)(\mb{u}^e)^{H,n+1}+\alpha_e (\mb{u}^e)^{L,n+1}, \label{eq:blended.evolution}
\end{equation}
where $(\mb{u}^e)^{H,n+1}$ and $(\mb{u}^e)^{L,n+1}$ are the solution updates at the $(n+1)^{\text{th}}$ time level using high and low-order schemes respectively. \revc{The low-order scheme is chosen to be a first-order space-time discretization of~\eqref{RHD_equation} which is oscillation free and admissibility preserving (Appendix~\ref{sec: appendix}). It is then used to limit the high-order method towards the chosen low-order method by using the blending coefficient $\alpha_e \in [0,1]$ as in~\eqref{eq:blended.evolution}. The limiting reduces accuracy in both space and time, and thus the blending coefficient has to be chosen so that the numerical method adaptively switches to high or low-order method in regions where the solution is smooth or non-smooth, respectively.} The process of finding $\alpha_e$ by a discontinuity indicator model will be discussed in Section~\ref{sec:discontinuity_indicator}. Note that we have suppressed the indices for solution points for simplicity of notation.
\begin{figure}[ht]
\centering
\includegraphics[width=0.45\linewidth]{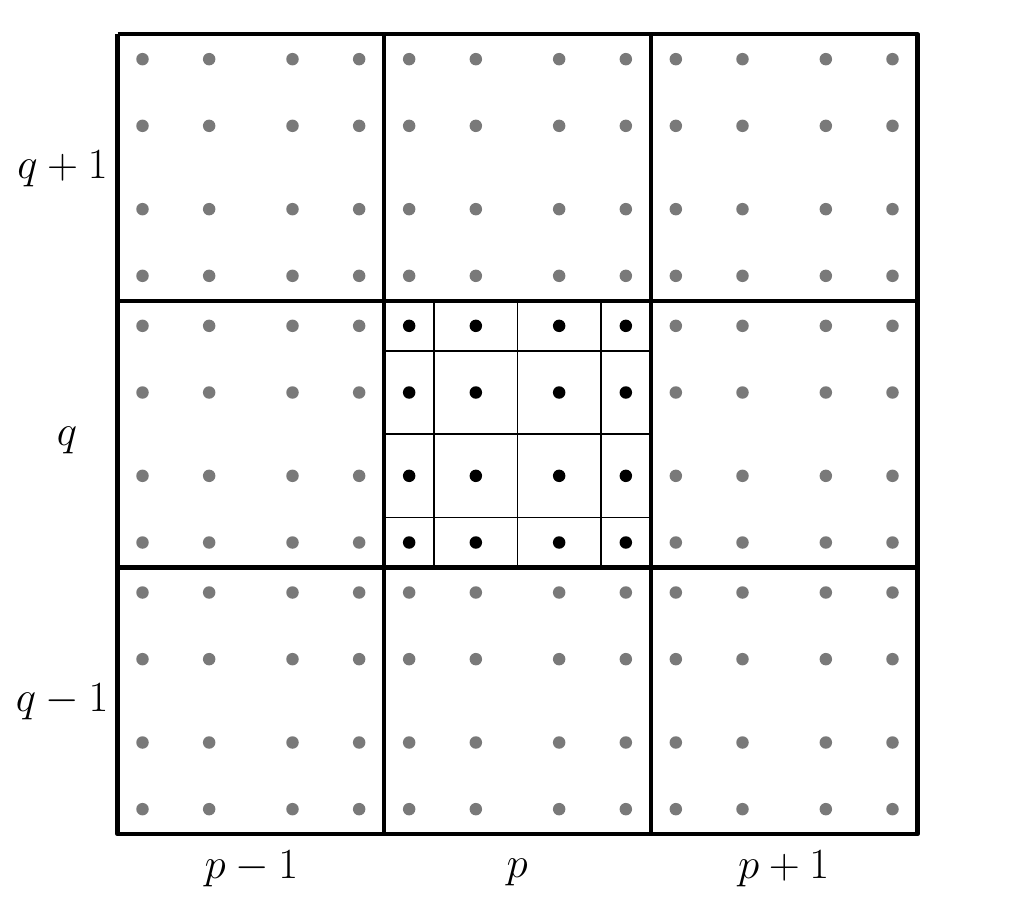}
\caption{Division of the element $\Omega_{pq}$ into $4\times 4$ sub-elements. (The dotted points inside the elements denote the solution points.)}\label{sub element division}
\end{figure}
\revc{The high-order scheme for computation of $(\mb{u}^e)^{H,n+1}$ is as in Section~\ref{sec:LWFR}, although the inter-element fluxes~(\ref{eq:numerical_flux_x},~\ref{eq:numerical_flux_y}) will be replaced by the blended numerical fluxes~\eqref{eq:loworder_scheme_notations} in order to maintain the conservation property (Section~\ref{sec:conservative_nature_section}).} For the low-order scheme, we divide the element $\Omega_{pq}$ into $(N+1)\times (N+1)$ sub-elements, each containing one solution point as shown in Figure~\ref{sub element division} and apply the low-order finite volume scheme \revc{to compute $(\mb{u}^e)^{L,n+1}$} in each of these sub-elements as explained below.

\rev{We first divide} the element $\Omega_e$ where $e=(p,q)$ into $(N+1)\times (N+1)$ sub-elements such that,
\begin{align}
\begin{split}
    x^p_{i+\frac{1}{2}}-x^p_{i-\frac{1}{2}}&=w_i \Delta x_p,\ \ \ \ \ 0\leq i\leq N,\\
    y^q_{j+\frac{1}{2}}-y^q_{j-\frac{1}{2}}&=w_j \Delta y_q,\ \ \ \ \ 0\leq j\leq N,
\end{split}
\end{align}
where $x^p_{i-\frac{1}{2}}, x^p_{i+\frac{1}{2}}$ and $y^q_{j-\frac{1}{2}}, y^q_{j+\frac{1}{2}}$ represent the left, right and lower, upper \textit{sub-faces} of $(i,j)^{\text{th}}$ sub-element respectively with $w_i$ as the $i^{\text{th}}$ quadrature weight associated to the solution points in reference coordinates.

\rev{Now obtain} the solution at the next time level with the help of the low-order scheme,
\begin{multline}\label{eq:loworder_scheme}
    (\mb{u}^{e}_{ij})^{L,n+1} = (\mb{u}^{e}_{ij})^n - \frac{\Delta t}{w_i \Delta x_p}[\hat{\mb{F}}_{i+\frac{1}{2}}(y_j)- \hat{\mb{F}}_{i-\frac{1}{2}}(y_j)] - \frac{\Delta t}{w_j \Delta y_q}[\hat{\mb{G}}_{j+\frac{1}{2}}(x_i)- \hat{\mb{G}}_{j-\frac{1}{2}}(x_i)]\\
     0\leq i,j\leq N,
\end{multline}
where
\begin{equation}\label{eq:loworder_scheme_notations}
\begin{split}
    &\hat{\mb{F}}_{-\frac{1}{2}}= \mb{F}_{p-\frac{1}{2}},  \hspace{5.7em}\hat{\mb{G}}_{-\frac{1}{2}}= \mb{G}_{q-\frac{1}{2}},\\
    &\hat{\mb{F}}_{i+\frac{1}{2}}= \mb{f}^{L,e}_{i+\frac{1}{2}},  \hspace{5.55em}\hat{\mb{G}}_{j+\frac{1}{2}}= \mb{g}^{L,e}_{j+\frac{1}{2}},\hspace{1cm}  0\leq i,j\leq N-1,\\
    &\hat{\mb{F}}_{N+\frac{1}{2}}= \mb{F}_{p+\frac{1}{2}},  \hspace{5em}\hat{\mb{G}}_{N+\frac{1}{2}}= \mb{G}_{q+\frac{1}{2}}.
\end{split}
\end{equation}
Here, $\mb{f}^{L,e}_{i+\frac{1}{2}}$, $\mb{g}^{L,e}_{j+\frac{1}{2}}$ denote the low-order sub-element fluxes. In this paper, we use the first-order finite volume scheme as the low-order scheme, and hence $\mb{f}^{L,e}_{i+\frac{1}{2}}(y_j)=\mb{f}^{\text{NF}}(\mb{u}_{ij}^e, \mb{u}_{i+1,j}^\rev{e})$ and $\mb{g}^{L,e}_{j+\frac{1}{2}}(x_i)=\mb{g}^{\text{NF}}(\mb{u}_{ij}^e, \mb{u}_{i,j+1}^\rev{e})$ with the numerical fluxes $\mb{f}^{\text{NF}}(\cdot,\cdot)$, $\mb{g}^{\text{NF}}(\cdot,\cdot)$ as the Rusanov flux~\cite{rusanov1962calculation}\revc{,
\begin{align}\label{eq: rusanov_flux}
\begin{split}
    \mb{f}^{\text{NF}}(\mb{u}^-, \mb{u}^+) &= \frac{1}{2}[\mb{f}(\mb{u}^-) + \mb{f}(\mb{u}^+)] - \frac{1}{2} \lambda_f [\mb{u}^+ - \mb{u}^-]\\
    \mb{g}^{\text{NF}}(\mb{u}^-, \mb{u}^+) &= \frac{1}{2}[\mb{g}(\mb{u}^-) + \mb{g}(\mb{u}^+)] - \frac{1}{2} \lambda_g [\mb{u}^+ - \mb{u}^-].
\end{split}
\end{align}
where,
\[
    \lambda_f = \max\{r\left( \mb{f}'(\mb{u}^-\right), r\left( \mb{f}'(\mb{u}^+\right) \}, \qquad \lambda_g = \max\{r\left( \mb{g}'(\mb{u}^-\right), r\left( \mb{g}'(\mb{u}^+\right) \}
\]
with $r(\cdot)$ denoting the spectral radius. 

Here, we denote the time average inter-element fluxes described in (\ref{eq:numerical_flux_x},~\ref{eq:numerical_flux_y}) as $\timeavgflux{\mb{F}}_{p\pm\frac{1}{2}}, \timeavgflux{\mb{G}}_{q\pm\frac{1}{2}}$. To maintain the conservation property of the blended scheme~\eqref{eq:blended.evolution}, the inter-element fluxes used for both low and high-order schemes should be same (refer to Section~\ref{sec:conservative_nature_section}) and while the low-order numerical fluxes $\mb{f}^{\text{NF}}(\cdot,\cdot)$, $\mb{g}^{\text{NF}}(\cdot,\cdot)$ control oscillations, the high-order numerical fluxes $\timeavgflux{\mb{F}}_{p\pm\frac{1}{2}}, \timeavgflux{\mb{G}}_{q\pm\frac{1}{2}}$~(\ref{eq:numerical_flux_x},~\ref{eq:numerical_flux_y}) produce high order accurate solutions. Thus, the initial guess of the blended numerical fluxes have to be chosen carefully to balance accuracy and robustness. Thus we compute the initial guess of the blended numerical fluxes as,}
\begin{equation}
\label{eq:initial.guess.blended.flux}
\begin{split}
\revc{\iniguessflux{\mb{F}}_{\norev{p\pm\frac{1}{2}}}}(y_j)&=(1-\alpha_{p\pm\frac{1}{2}})\revc{\timeavgflux{\mb{F}}_{\norev{p\pm\frac{1}{2}}}}(y_j)+\alpha_{p\pm\frac{1}{2}}\revc{\mb{f}^L_{p\pm\frac{1}{2}}(y_j)},\\
\revc{\iniguessflux{\mb{G}}_{\norev{q\pm\frac{1}{2}}}}(x_i)&=(1-\alpha_{q\pm\frac{1}{2}})\revc{\timeavgflux{\mb{G}}_{\norev{q\pm\frac{1}{2}}}}(x_i)+\alpha_{q\pm\frac{1}{2}}\revc{\mb{g}^L_{q\pm\frac{1}{2}}(x_i)}.
\end{split}
\end{equation}
\revc{where
\begin{align*}
    &\mb{f}^L_{p+\frac{1}{2}}(y_j) = \mb{f}^{\text{NF}}(\mb{u}^{p,q}_{N,j},\mb{u}^{p+1,q}_{0,j}), \qquad \mb{f}^L_{p-\frac{1}{2}}(y_j) = \mb{f}^{\text{NF}}(\mb{u}^{p-1,q}_{N,j},\mb{u}^{p,q}_{0,j}),\\
    &\mb{g}^L_{q+\frac{1}{2}}(x_i) = \mb{g}^{\text{NF}}(\mb{u}^{p,q}_{i,N},\mb{u}^{p,q+1}_{i,0}), \qquad \mb{g}^L_{q-\frac{1}{2}}(x_i) = \mb{g}^{\text{NF}}(\mb{u}^{p,q-1}_{i,N},\mb{u}^{p,q}_{i,0}).
\end{align*}

The blended coefficients at the faces are computed from the elements adjacent to the faces as,
\begin{align*}
&\alpha_{p-\frac{1}{2}} = \frac{1}{2}(\alpha_{p-1} + \alpha_{p}) \qquad \alpha_{p+\frac{1}{2}} = \frac{1}{2}(\alpha_{p} + \alpha_{p+1}) \\
&\alpha_{q-\frac{1}{2}} = \frac{1}{2}(\alpha_{q-1} + \alpha_{q})\ \qquad \alpha_{q+\frac{1}{2}} = \frac{1}{2}(\alpha_{q} + \alpha_{q+1}).
\end{align*}
This choice ensures that we use the high-order accurate fluxes in regions where the solution is smooth while using the robust low-order fluxes in non-smooth regions like shocks and other discontinuities. The procedure for finding the blending coefficients in an element is explained in Section~\ref{sec:discontinuity_indicator}.

Note that we will further modify these initial guesses of the blended numerical fluxes~\eqref{eq:initial.guess.blended.flux} in Section~\ref{sec: admissibility} for the admissibility in means (Definition~\ref{defn:adm.means}) of the blended scheme and will refer those as the \textit{blended numerical fluxes}. 
Now, all the fluxes needed in~\eqref{eq:loworder_scheme_notations} are available and we can compute the low-order evolution $(\mb{u}^{e}_{ij})^{L,n+1}$~\eqref{eq:loworder_scheme}. We can again compute the high-order evolution $(\mb{u}^{e}_{ij})^{H,n+1}$ from~(\ref{LWupdate},~\ref{LWupdate.fluxes}) using the blended numerical fluxes $\mb{F}_{p\pm\frac{1}{2}}, \mb{G}_{q\pm\frac{1}{2}}$ from Section~\ref{sec: admissibility} as the inter-element fluxes, prescribing the blending scheme in~\eqref{eq:blended.evolution}.}

\subsection{Conservative nature of the blended scheme} \label{sec:conservative_nature_section}
\revc{Consider the blended numerical fluxes $\mb{F}_{p\pm\frac{1}{2}}, \mb{G}_{q\pm\frac{1}{2}}$ as the common inter-element fluxes in both the high-order~\eqref{LWupdate} and low-order~\eqref{eq:loworder_scheme} schemes. Now from equations~(\ref{LWupdate},~\ref{LWupdate.fluxes}) with the blended numerical fluxes  at the inter-element faces,} one can easily see that the element average of the solution at time level $n+1$ for the high-order scheme is given by,
\begin{align*}
    (\Bar{\mb{u}}^e)^{H,n+1} &= \sum^N_{i,j=0} w_i w_j(\mb{u}_{ij}^e)^{H,n+1}\\
    &=(\Bar{\mb{u}}^e)^{n} - \frac{\Delta t}{\Delta x_p} \Bigg[\sum^{N}_{j=0}w_j\Big(\mb{F}_{p+\frac{1}{2}}(\eta_j)-\mb{F}_{p-\frac{1}{2}}(\eta_j)\Big)\Bigg] -\frac{\Delta t}{\Delta y_q}\Bigg[\sum^{N}_{i=0}w_i\Big( \mb{G}_{q+\frac{1}{2}}(\xi_i)-\mb{G}_{q-\frac{1}{2}}(\xi_i)\Big)\Bigg].
\end{align*}
Again, from equations~(\ref{eq:loworder_scheme},~\ref{eq:loworder_scheme_notations}) \revc{using the blended numerical fluxes at the inter-element faces} we get a similar expression for the low-order scheme,
\begin{align*}
    (\Bar{\mb{u}}^e)^{L,n+1}
    =&(\Bar{\mb{u}}^e)^{n} - \frac{\Delta t}{\Delta x_p} \Bigg[\sum^{N}_{j=0}w_j\Big(\mb{F}_{p+\frac{1}{2}}(\eta_j)-\mb{F}_{p-\frac{1}{2}}(\eta_j)\Big)\Bigg]-\frac{\Delta t}{\Delta y_q}\Bigg[\sum^{N}_{i=0}w_i\Big( \mb{G}_{q+\frac{1}{2}}(\xi_i)-\mb{G}_{q-\frac{1}{2}}(\xi_i)\Big)\Bigg].
\end{align*}
Consequently, for the blended scheme we get,
\begin{align}
     (\Bar{\mb{u}}^e)^{n+1}  &(1-\alpha_e)(\Bar{\mb{u}}^e)^{H,n+1}+\alpha_e (\Bar{\mb{u}}^e)^{L,n+1}\nonumber\\
    &= (\Bar{\mb{u}}^e)^{n} - \frac{\Delta t}{\Delta x_p} \Bigg[\sum^{N}_{j=0}w_j\Big(\mb{F}_{p+\frac{1}{2}}(\eta_j)-\mb{F}_{p-\frac{1}{2}}(\eta_j)\Big)\Bigg] -\frac{\Delta t}{\Delta y_q}\Bigg[\sum^{N}_{i=0}w_i\Big( \mb{G}_{q+\frac{1}{2}}(\xi_i)-\mb{G}_{q-\frac{1}{2}}(\xi_i)\Big)\Bigg].
\end{align}
Hence, the blended scheme is conservative. 

Note that for the conservative nature of the blended scheme, the inter-element fluxes used in the low and high-order schemes should be the same, because the flux leaving an element through a boundary should enter the adjacent element with the same boundary. If we had used different inter-element fluxes in the low and high-order schemes, then conservation would have required
\begin{equation}
\begin{split}
(1-\alpha_{pq})\mb{F}^H_{p+\frac{1}{2}}+\alpha_{pq}\mb{F}^L_{p+\frac{1}{2}}&=(1-\alpha_{p+1,q})\mb{F}^H_{p+\frac{1}{2}}+\alpha_{p+1,q}\mb{F}^L_{p+\frac{1}{2}},\\
(1-\alpha_{pq})\mb{G}^H_{q+\frac{1}{2}}+\alpha_{pq}\mb{G}^L_{q+\frac{1}{2}}&=(1-\alpha_{p,q+1})\mb{G}^H_{q+\frac{1}{2}}+\alpha_{p,q+1}\mb{G}^L_{q+\frac{1}{2}},
\end{split}\label{eq:why.same.flux}
\end{equation}
where the superscripts $(\cdot)^H$, $(\cdot)^L$ denote the high and low-order schemes respectively. The above relations must hold for any values of the blending coefficients and hence $\mb{F}^H_{p+\frac{1}{2}}=\mb{F}^L_{p+\frac{1}{2}}$ and $\mb{G}^H_{q+\frac{1}{2}}=\mb{G}^L_{q+\frac{1}{2}}$ for all $p,q$.

\section{Discontinuity indicator model} \label{sec:discontinuity_indicator}
Following the ideas of Persson and Peraire~\cite{persson2006sub}, which are also adopted in~\cite{klockner2011viscous}, the authors in~\cite{hennemann2021provably} constructed an indicator model for the Euler equations which was later used in~\cite{babbar2024admissibility}. The idea is to expand the solution polynomial or some quantity derived from the components of the solution with respect to the Legendre orthogonal basis functions and then analyze the decay of the coefficients. In~\cite{hennemann2021provably, babbar2024admissibility}, the authors have used the product of fluid density and pressure as the indicator quantity for Euler equations.
However, in Section~\ref{sec: Numerical_verifications}, we will show through numerical experiments that this indicator model does not adequately capture the discontinuities of the RHD equations, leading to spurious oscillations. Hence, we propose a new indicator model specific to Gauss-Legendre solution points that do not contain endpoint values of the interval $[0,1]$.

Let us take the indicator quantity as the product of fluid density, pressure, and Lorentz factor, $K = \rho p \Gamma$, which will take care of the jumps in velocities as well as with those of pressure and density, primarily when the velocity is near the speed of light.
Unlike the indicator model of~\cite{hennemann2021provably}, where the authors have considered only the solution points, we take the boundary points $(x_{-1},y_j),\ (x_{N+1},y_j),\ (x_i,y_{-1}),\ (x_i,y_{N+1})$ for $-1\leq i,j\leq N+1$ of the reference element for $\Omega_{pq}$ also in consideration, where $(x_{-1}, x_{N+1})=(y_{-1},y_{N+1})=(0,1)$ and find the solution values at these points by extrapolating those from the neighboring elements and taking averages. This gives us a degree $N+2$ solution approximation within each element. Now, let $K$ be the indicator quantity, and let us take the modal representation of $K$ in terms of the Legendre polynomials $P_i$,
\begin{align} \label{eq:kh.defn}
K_h(\xi,\eta) = \sum^{N+2}_{i,j=0} \hat{K}_{ij} P_i(2\xi -1)P_j(2\eta -1),\hspace{1cm} \xi,\eta\in [0,1],
\end{align}
where the coefficients are given by,
\begin{equation} \label{eq:k.hat.defn}
\hat{K}_{ij} = \iint\limits_{[0,1]^2} K\big(\mb{u}_h(\xi,\eta)\big)P_i(2\xi -1)P_j(2\eta -1) \ud\xi \ud\eta.
\end{equation}
Here, $\mb{u}_h(\xi,\eta)$ is the degree $N+2$ polynomial approximation of the solution inside an element considering the solution values at all the solution points, including the boundary points. The integration is performed numerically by the Gauss-Legendre quadrature to get sufficient accuracy.

The coefficients $\hat{K}_{ij}$ of a smooth function decreases rapidly in comparison to non-smooth functions~\cite{persson2006sub,canuto2007spectral} and hence we introduce two matrix notations,
\begin{align}
\begin{split}
&K^{+1} = [\hat{K}_{ij}],\qquad 0\leq i,j \leq N+2,\\
&K^{-1} = [\hat{K}_{ij}],\qquad 0\leq i,j \leq N,
\end{split}
\end{align}
and take the ratio of the energy corresponding to the highest-frequency modes and the energy corresponding to all the modes except the lowest-frequency ones as,
\begin{align}
    E = \frac{\| K^{+1}\|^2 -\| K^{-1}\|^2}{\| K^{+1}\|^2 - \hat{K}_{0,0}^2},
\end{align}
where the norm $\| \cdot \| : M_{n\times n}(\mathbb{R})\to \mathbb{R}$, $n\in \mathbb{Z}$ is the Frobenius norm of a matrix defined as $\| M\| =\sqrt{\sum\limits_{i,j=1}^n M^2_{ij}}$. The lowest mode corresponds to a constant function, which does not affect the smoothness of the function. In~\cite{hennemann2021provably}, the degree-$N$ solution polynomial was used in~(\ref{eq:kh.defn},~\ref{eq:k.hat.defn}) and the constant mode was not removed. These are the fundamental differences in our proposed model.

In~\cite{hennemann2021provably}, the authors have converted their ratio of energy to a value inside $[0, 1]$ using a logistic function and a threshold, and then they have clipped the value of $\alpha'_e$. But here, we will clip on the value of $E$ itself. If $E>E_u$, we consider the solution to be discontinuous and set $\alpha'_e=\alpha'_{\text{max}}$, and if $E<E_l$, we consider it to be smooth and set $\alpha'_e= 0$. Between $E_l$ and $E_u$ we take $\alpha'_e$ as a smooth function with respect to $E$, given by,
\begin{equation}\label{eq:amax.defn}
\alpha'_e = \alpha'_{\text{max}} \sin\bigg(\frac{\pi}{2}y^2\bigg), \quad y= \frac{\log (E/E_l)}{\log (E_u/E_l)},
\end{equation}
so that the slope of the function increases slowly near $E_l$ compared to how it decreases near $E_u$ with both the slopes at $E_l$ and $E_u$ vanished. 
Here, $E_u$ and $E_l$ are taken to be $E_u=0.009$ and $E_l=\frac{E_u}{1000}$ respectively, determined through numerical experiments with the aim of controlling spurious oscillations while maintaining accuracy. Naturally, one may take $\alpha'_\text{max}=1$, but we can choose $\alpha'_\text{max}\in (0,1)$ to control the effect of the low-order scheme according to the specific needs of a problem. A discontinuity in one element can affect the solution in neighboring elements as well, since it can propagate into them during the next few time steps. Hence, we perform smoothing of the blending coefficient as $\alpha_e=\max\limits_{e_n}\{\alpha_e',0.5 \alpha_{e_n}\}$, where $e_n$ goes over all the elements having a common face with the element $\Omega_e$. We remark that we do not always reduce the scheme from high-order to first-order finite volume if an oscillatory solution is detected, but gradually blend the two schemes depending on how oscillatory the solution is locally.

\section{Admissibility of the solution}\label{sec: admissibility}
From Lemma~\ref{equivalent_admissibility_region}, the solution of the $d$-dimensional RHD equations has to be in the convex set $\Uad' \subset \mathbb{R}^{d+2}$ for it to be physically meaningful.  Regarding the admissibility preservation of the numerical schemes, we have the following two definitions~\cite{babbar2024admissibility},
    \begin{definition}
        For the flux reconstruction scheme to be admissibility preserving, it has to satisfy,
        \[
        (\mb{u}^e_{ij})^n\in \Uad'\ \ \ \forall e,i,j \implies (\mb{u}^e_{ij})^{n+1}\in \Uad' \ \ \ \forall e,i,j.
        \]
    \end{definition}
    \begin{definition}\label{defn:adm.means}
        For the flux reconstruction scheme to be admissibility preserving in means it has to satisfy,
        \[
        (\mb{u}^e_{ij})^n\in \Uad' \ \forall e,i,j \implies (\Bar{\mb{u}}^e)^{n+1}\in \Uad' \ \ \ \forall e.
        \]
    \end{definition}
    To make the blended scheme admissible, we first ensure that the scheme is admissible in means. From Section~\ref{sec:conservative_nature_section}, we can see that the element means of the solution of the blended scheme and the low-order scheme are the same, hence ensuring the admissibility in means of the low-order scheme is sufficient to ensure that for the blended scheme. 
    
    From Section~\ref{sec:conservative_nature_section}, we also see that the blended numerical flux that we use in the blended scheme has to be applied as the inter-element flux in the low-order scheme~\eqref{eq:loworder_scheme}, possibly violating the admissibility of the low-order evolutions. Specifically, we have the admissibility of the low-order evolutions $(\mb{u}_{ij}^e)^{L,n+1}$ for $1\leq i,j\leq N-1$~\eqref{eq:loworder_scheme} as we use the low-order flux for those solution points, and first-order finite volume method is admissibility preserving with the low-order inter-element flux,  \rev{see Theorem~\ref{theorem: admissility_rusanov_scheme}}. However, the points adjacent to the element faces require modification of the inter-element fluxes to make the evolutions to be admissible. 
    
    Following~\cite{babbar2024admissibility}, the process to make the evolution of solution at the solution point $(0,0)$ admissible is explained in the following steps. The point $(0,0)$ is a corner point of the element and thus requires modification of both $x, y$-directional fluxes. Therefore, the description below can be applied to ensure the admissibility of low-order evolution~\eqref{eq:loworder_scheme} at any solution point. For the implementation, the correction of $x,y$-directional fluxes will be intertwined with the $x$ and $y$-directional loops, respectively (see Algorithms 1,2 of~\cite{babbar2024admissibility}). However, we describe the flux correction of both $x,y$-directional fluxes without the interface loops for simplicity of explanation.
     \paragraph{Step 1.} We begin by taking the \revc{initial guesses of the blended numerical fluxes $\iniguessflux{\mb{F}}_{p-\frac{1}{2}}(y_0)$, $\iniguessflux{\mb{G}}_{q-\frac{1}{2}}(x_0)$ described in equation~\eqref{eq:initial.guess.blended.flux}} at the faces $p-\frac{1}{2}$, $q-\frac{1}{2}$ respectively, adjacent to the solution point $(0,0)$. Then take $k_x, k_y>0$ such that $k_x+k_y=1$. We compute low-order evolutions using the first-order finite volume method,
     \begin{align}\label{eq: directional low-order evolutions}
     \begin{split}
     \hat{\mb{u}}_{x,0}^{L,n+1} &= \big(\mb{u}_{0,0}^{p,q}\big)^{L,n} - \frac{\Delta t}{k_x w_0 \Delta x_{p}}\Big[\mb{f}^{L,p,q}_{\frac{1}{2}}(y_0)- \mb{f}^{\text{NF}}(\mb{u}^{p-1,q}_{N,0},\mb{u}^{p,q}_{0,0})\Big],\\
    \hat{\mb{u}}_{x,N}^{L,n+1} &= \big(\mb{u}_{N,0}^{p-1,q}\big)^{L,n} - \frac{\Delta t}{k_x w_\rev{N} \Delta x_{p-1}}\Big[\mb{f}^{\text{NF}}(\mb{u}^{p-1,q}_{N,0},\mb{u}^{p,q}_{0,0}) - \mb{f}^{L,p-1,q}_{N-\frac{1}{2}}(y_0)\Big],\\
    \hat{\mb{u}}_{y,0}^{L,n+1} &= \big(\mb{u}_{0,0}^{p,q}\big)^{L,n} - \frac{\Delta t}{k_y w_0 \Delta y_{q}}\Big[\mb{g}^{L,p,q}_{\frac{1}{2}}(x_0)- \mb{g}^{\text{NF}}(\mb{u}^{p,q-1}_{0,N},\mb{u}^{p,q}_{0,0})\Big],\\
    \hat{\mb{u}}_{y,N}^{L,n+1} &= \big(\mb{u}_{0,N}^{p,q-1}\big)^{L,n} - \frac{\Delta t}{k_y w_\rev{N} \Delta y_{q-1}}\Big[\mb{g}^{\text{NF}}(\mb{u}^{p,q-1}_{0,N},\mb{u}^{p,q}_{0,0}) - \mb{g}^{L,p,q-1}_{N-\frac{1}{2}}(x_0)\Big].
    \end{split}
     \end{align}
     These evolutions are admissible because of the admissibility of the first-order finite volume method \rev{(Remark~\ref{remark: admissility_rusanov_scheme})}. Following~\cite{babbar2024admissibility}, we take $k_x= k_y=\frac{1}{2}$. We now compute the low-order evolutions using the \revc{initial guess of the blended numerical fluxes as the inter-element fluxes},
     \begin{align}
     \begin{split}
    \mb{u}_{x,0}^{L,n+1} &= \big(\mb{u}_{0,0}^{p,q}\big)^{L,n} - \frac{\Delta t}{k_x w_0 \Delta x_{p}}\Big[\mb{f}^{L,p,q}_{\frac{1}{2}}(y_0)- \revc{\iniguessflux{\mb{F}}_{\norev{p-\frac{1}{2}}}}(y_0)\Big],\\
    \mb{u}_{x,N}^{L,n+1} &= \big(\mb{u}_{N,0}^{p-1,q}\big)^{L,n} - \frac{\Delta t}{k_x w_\rev{N} \Delta x_{p-1}}\Big[\revc{\iniguessflux{\mb{F}}_{\norev{p-\frac{1}{2}}}}(y_0) - \mb{f}^{L,p-1,q}_{N-\frac{1}{2}}(y_0)\Big],\\
    \mb{u}_{y,0}^{L,n+1} &= \big(\mb{u}_{0,0}^{p,q}\big)^{L,n} - \frac{\Delta t}{k_y w_0 \Delta y_{q}}\Big[\mb{g}^{L,p,q}_{\frac{1}{2}}(x_0)- \revc{\iniguessflux{\mb{G}}_{\norev{q-\frac{1}{2}}}}(x_0)\Big],\\
    \mb{u}_{y,N}^{L,n+1} &= \big(\mb{u}_{0,N}^{p,q-1}\big)^{L,n} - \frac{\Delta t}{k_y w_\rev{N} \Delta y_{q-1}}\Big[\revc{\iniguessflux{\mb{G}}_{\norev{q-\frac{1}{2}}}}(x_0) - \mb{g}^{L,p,q-1}_{N-\frac{1}{2}}(x_0)\Big].
    \end{split}
     \end{align}
     Our objective is to choose the blended numerical fluxes so that the above evolutions are admissible.
    \paragraph{Step 2.} If  $\mb{u}_{x,0}^{L,n+1}, \mb{u}_{x,N}^{L,n+1} \in \Uad'$, there is no need to change $\revc{\iniguessflux{\mb{F}}_{\norev{p-\frac{1}{2}}}}$ and we simply assign $\revc{\mb{F}}_{p-\frac{1}{2}}=\revc{\iniguessflux{\mb{F}}_{\norev{p-\frac{1}{2}}}}$. Again if $\mb{u}_{y,0}^{L,n+1}, \mb{u}_{y,N}^{L,n+1} \in \Uad'$, assign $\revc{\mb{G}}_{q-\frac{1}{2}}=\revc{\iniguessflux{\mb{G}}_{\norev{q-\frac{1}{2}}}}$. But if this is not the case, update the \revc{initial guess of the} blended numerical fluxes as follows,
    \begin{align}\label{eq:first.blended.flux.for.admissibility}
    \begin{split}
    \revc{\mb{F}}_{p-\frac{1}{2}}(y_0) = \theta_x \revc{\iniguessflux{\mb{F}}_{\norev{p-\frac{1}{2}}}}(y_0) + (1-\theta_x) \mb{f}^{\text{NF}}(\mb{u}^{p-1,q}_{N,0},\mb{u}^{p,q}_{0,0}),\\
    \revc{\mb{G}}_{q-\frac{1}{2}}(x_0) = \theta_y \revc{\iniguessflux{\mb{G}}_{\norev{q-\frac{1}{2}}}}(x_0) + (1-\theta_y) \mb{g}^{\text{NF}}(\mb{u}^{p,q-1}_{0,N},\mb{u}^{p,q}_{0,0}),
    \end{split}
    \end{align}
    where
    \begin{multline*}
    \theta_\kappa = \text{min}\Bigg\{\Bigg|\frac{\frac{1}{10}D\Big(\hat{\mb{u}}^{L,n+1}_{\kappa,0}\Big)-D\Big(\rev{\hat{\mb{u}}_{\kappa,0}^{L,n+1}} \Big)}{D\Big(\mb{u}^{L,n+1}_{\kappa,0}\Big)-D\Big(\hat{\mb{u}}_{\kappa,0}^{L,n+1} \Big)} \Bigg|, 
    \ \Bigg|\frac{\frac{1}{10}D\Big(\hat{\mb{u}}^{L,n+1}_{\kappa,N}\Big)-D\Big(\rev{\hat{\mb{u}}_{\kappa,N}^{L,n+1}} \Big)}{D\Big(\mb{u}^{L,n+1}_{\kappa,N}\Big)-D\Big(\hat{\mb{u}}_{\kappa,N}^{L,n+1} \Big)} \Bigg|,\ 1 \Bigg\},\qquad 
    \text{for}\ \kappa = x,y,
    \end{multline*}
    \revc{which is taken following~\cite{babbar2024admissibility}. The factor of $\frac{1}{10}$ is used following~\cite{rueda2021subcell} and this choice of $\theta_\kappa$ enforces a stricter condition than what is required for positivity (Theorem \ref{admissibility_theorem_fc}).}

    \paragraph{Step 3.} Repeat the last step replacing $D$ with $q$ and using the updated fluxes~\eqref{eq:first.blended.flux.for.admissibility} in place of $\revc{\iniguessflux{\mb{F}}_{\norev{p-\frac{1}{2}}}}$, $\revc{\iniguessflux{\mb{G}}_{\norev{q-\frac{1}{2}}}}$.
    
    \paragraph{Step 4.} Finally, the low-order evolutions are computed with the modified numerical fluxes as
\begin{align}
\begin{split}
    \mb{u}_{x,0}^{L,n+1} &= \big(\mb{u}_{0,0}^{p,q}\big)^{L,n} - \frac{\Delta t}{k_x w_0 \Delta x_{p}}\Big[\mb{f}^{L,p,q}_{\frac{1}{2}}(y_0)- \revc{\mb{F}}_{p-\frac{1}{2}}(y_0)\Big],\\
    \mb{u}_{y,0}^{L,n+1} &= \big(\mb{u}_{0,0}^{p,q}\big)^{L,n} - \frac{\Delta t}{k_y w_0 \Delta y_{q}}\Big[\mb{g}^{L,p,q}_{\frac{1}{2}}(x_0)- \revc{\mb{G}}_{q-\frac{1}{2}}(x_0)\Big],
\end{split}
\end{align}
where $\revc{\mb{F}}_{p-\frac{1}{2}}$, $\revc{\mb{G}}_{q-\frac{1}{2}}$ are the updated fluxes after performing Step~2 and Step~3, \revc{which we refer as the \textit{blended numerical fluxes}} and this solution belongs to the admissible set $\Uad'$ (refer to Theorem~\ref{admissibility_theorem_fc}).

Finally, calculate the required solution as,
\begin{equation}\label{eq:split_admissibility}
\big(\mb{u}_{0,0}^{p,q}\big)^{L,n+1} = k_x \mb{u}_{x,0}^{L,n+1} + k_y \mb{u}_{y,0}^{L,n+1}.
\end{equation}

Repeat these steps for all the required solution points.

\begin{remark}
\begin{enumerate}
\item To modify the flux \rev{as} $\revc{\mb{F}}_{p-\frac{1}{2}}(y_0)$ \rev{in the above steps}, we have considered both the adjacent solution points in the neighboring elements $\Omega_{p-1,q},\Omega_{pq}$ and so the steps need not be repeated during the correction of the inter-element flux for the other solution point. Similarly, for the case of $\revc{\mb{G}}_{q-\frac{1}{2}}(x_0)$. 
\item For the non-corner points, we need to carry out the above procedure to modify the numerical flux in one direction only, as for the other direction we are already using the low-order numerical flux.
\item We can carry out the above procedure to modify the inter-element fluxes inside the interface loops. Consequently, the process becomes less expensive in contrast to other works like~\cite{moe2017positivity} that require additional loops.
\rev{\item For a more optimal choice of $k_x, k_y$, one can refer to~\cite{cui2023classic,cui2024optimal}, where the authors investigate a similar context when studying optimal cell average decomposition. For the one-dimensional case, we do not need to split the solution as in equation~\eqref{eq:split_admissibility}, hence we can proceed without considering the splitting coefficients $k_x,k_y$ (for more details, refer to~\cite{babbar2024admissibility}).}
\end{enumerate}
\end{remark}

    Let us now state a theorem regarding admissibility.
    \begin{theorem} \label{admissibility_theorem_fc}
        The quantities, $\mb{u}_{x,0}^{L,n+1}$ and $\mb{u}_{y,0}^{L,n+1}$ in Step~4 of the inter-element flux correction procedure are admissible.
    \end{theorem}
    \begin{proof}
        Clearly $D$ is a concave function of the conservative variables, consequently for the low-order evolutions $\mb{u}_{\kappa,0}^{L,n+1}$, $\kappa = x, y$ after the Step~2 we have,
        \begin{align*}              
        D\Big(\mb{u}_{\kappa,0}^{L,n+1}\Big)
        &=D\Big(\theta_\kappa \mb{u}_{\kappa,0}^{L,n+1,\text{old}}+(1-\theta_\kappa) \hat{\mb{u}}_{\kappa,0}^{L,n+1}\Big)\\
        &\geq \theta_\kappa D\Big(\mb{u}_{\kappa,0}^{L,n+1,\text{old}}\Big) + (1-\theta_\kappa) D\Big(\hat{\mb{u}}_{\kappa,0}^{L,n+1}\Big)\\
        &> \frac{1}{10} \bigg(D\Big(\hat{\mb{u}}_{\kappa,0}^{L,n+1}\Big)\bigg),
        \end{align*}
        
        where `old' is used in the superscript to denote the quantity computed before the flux correction for constraint $D$ is applied, and $\hat{(\cdot)}$ denotes the update using the low-order inter-element flux.

        Again, from Lemma~\ref{q_concave}, we have that $q$ is also a concave function of the conservative variables, and hence we can re-write every step of the above procedure using $q$ in place of $D$ for the low-order evolutions after Step~3.
    \end{proof}
Thus we have the solution, computed using the low-order scheme in the admissibility region, giving the admissibility in means of the low-order scheme. Finally, using the scaling limiter from~\cite{zhang2010maximum} and using the fact that the element means of the solution is in the admissibility region, we can make the solution admissible at all solution points.
\section{Numerical verifications}\label{sec: Numerical_verifications}
In this section, several numerical simulations are run with the above scheme to demonstrate its accuracy and robustness. \revc{The size of the time step is calculated as,
\begin{equation} \label{eq: time_step.fourier}
    \Delta t= l_s \text{CFL}(N)\min_{pq}{\left(\frac{r\left(\mb{f}'(\Bar{\mb{u}}_{pq})\right)}{\Delta x_p}+ \frac{r\left(\mb{g}'(\Bar{\mb{u}}_{pq})\right)}{\Delta y_q}\right)^{-1}}
\end{equation}
where, $r(\cdot)$ denotes the spectral radius and $l_s\leq 1$ denotes a safety factor~\cite{babbar2024admissibility}. Here the minimum is taken over all the elements $\Omega_{pq}$ and $\text{CFL}(N)$ is the optimal CFL number for degree $N$ solution polynomial, which is obtained in~\cite{BABBAR2022111423} by Fourier stability analysis; one can refer to Table 1 and Table 2 of~\cite{BABBAR2022111423} for the CFL numbers in one and two dimensions respectively. Note that this time step restriction does not guarantee the admissibility of the low-order evolutions in~\eqref{eq:split_admissibility}, which need the CFL type condition as in Remark~\ref{remark: admissility_rusanov_scheme}. But in our tests, all the simulations could be performed with time steps calculated in~\eqref{eq: time_step.fourier}; in case the solution is not admissible in some element, the update can be re-computed with a reduced time step.}

Here we present the results using one and two-dimensional schemes. The simulations use specific heat ratio $\gamma = \frac{5}{3}$ in all test cases and $\alpha'_{\text{max}}=1$~\eqref{eq:amax.defn} unless stated otherwise. \revc{We use the safety factor, $l_s = 0.95$ for all the test cases unless mentioned otherwise.} The numerical experiments are performed by making additions to \texttt{Tenkai.jl}~\cite{tenkai}, which include the proposed smoothness indicator (Section~\ref{sec:discontinuity_indicator}) and capability to solve the RHD equations.

\subsection{1-D simulations}
Before going to the test cases with discontinuities, let us check the accuracy of the proposed scheme with two smooth test cases by performing grid convergence studies.

\subsubsection{First smooth problem}

\begin{table}[ht]
    \centering
    \begin{tabular}{|l|l|l|l|l|l|l|}
    \hline
    \multirow{2}{*}{No. of cells} & \multicolumn{6}{|c|}{\rev{Degree $N=3$}} \\
    \cline{2-7}
    & $L^1$ error & Order & $L^2$ error & Order  & $L^{\infty}$ error & Order\\
    \hline
 8   &    3.07898e-05     &      -             &          3.94642e-05       &        -           &          8.67502e-05           &          -\\ 
16   &    1.75443e-06     &      4.13338       &      2.11296e-06      &       4.22321      &       4.36599e-06        &          4.31249\\
32   &    1.15319e-07     &      3.92729       &      1.33483e-07      &       3.98454      &       2.61541e-07        &          4.06120\\
64   &    7.20460e-09     &      4.00057       &      8.24737e-09      &       4.01657      &       1.56746e-08        &          4.06054\\
128   &    4.35338e-10     &      4.04871       &      4.95295e-10      &       4.05757      &       9.25566e-10        &          4.08195\\
         \hline
    \end{tabular}
    \vspace{3pt}
    \caption{Numerical results for the fluid density using the \rev{scheme with degree $N=3$}.}
    \label{table1}
\end{table}

\begin{table}[htbp]
    \centering
    \begin{tabular}{|l|l|l|l|l|l|l|}
    \hline
    \multirow{2}{*}{No. of cells} & \multicolumn{6}{|c|}{\rev{Degree $N=4$}} \\
    \cline{2-7}
    & $L^1$ error & Order & $L^2$ error & Order  & $L^{\infty}$ error & Order\\
    \hline
 8   &    2.72454e-06     &      -             &          3.18900e-06       &        -           &          5.70499e-06           &          -\\ 
16   &    9.82598e-08     &      4.79327       &      1.13125e-07      &       4.81711      &       1.86065e-07        &          4.93835\\
32   &    3.17686e-09     &      4.95093       &      3.67525e-09      &       4.94393      &       6.35705e-09        &          4.87131\\
64   &    1.00755e-10     &      4.97868       &      1.16327e-10      &       4.98158      &       2.02323e-10        &          4.97362\\
128   &    3.18210e-12     &      4.98472       &      3.66787e-12      &       4.98711      &       6.38423e-12        &          4.98601\\
         \hline
    \end{tabular}
    \vspace{3pt}
    \caption{Numerical results for the fluid density using the \rev{scheme with degree $N=4$}.}
    \label{table2}
\end{table}
Consider the domain of computation to be $[0, 1]$ with periodic boundary conditions. At the initial time, let the fluid have density $\rho(x,0)=2+\sin(2\pi x)$ moving with a velocity $v_1(x,0) = 0.5$ and let the initial pressure in the whole domain be unity. At time $t$, the fluid density will change to $\rho(x,t)=2+\sin\big(2\pi (x-0.5 t)\big)$, and the other quantities will remain unchanged. We present the results of the simulations in terms of the error norms at time $t=2.0$ in Table~\ref{table1} and Table~\ref{table2}; we see that the methods converge at the rate of $O(\Delta x)^{N+1}$ for solution polynomial degrees $N=3,4$.

\reva{
\begin{table}[htbp]
    \centering
    \begin{tabular}{|l|l|l|l|l|l|l|}
    \hline
    \multirow{2}{*}{No. of cells} & \multicolumn{6}{|c|}{\rev{Degree $N=3$}} \\
    \cline{2-7}
    & $L^1$ error & Order & $L^2$ error & Order  & $L^{\infty}$ error & Order\\
\hline
50   &    7.73934e-05     &      -             &          3.69156e-04       &        -           &          3.22987e-03           &          -\\ 
100   &    4.26847e-06     &      4.18042       &      2.89081e-05      &       3.67468      &       4.10458e-04        &          2.97617\\
150   &    7.47121e-07     &      4.29823       &      5.09525e-06      &       4.28104      &       6.93707e-05        &          4.38462\\
200   &    2.21167e-07     &      4.23144       &      1.54992e-06      &       4.13688      &       1.89548e-05        &          4.50986\\
250   &    9.13198e-08     &      3.96404       &      6.30260e-07      &       4.03249      &       7.60146e-06        &          4.09475\\
300   &    4.28780e-08     &      4.14657       &      3.00045e-07      &       4.07084      &       4.06040e-06        &          3.43930\\
350   &    2.25800e-08     &      4.16018       &      1.60143e-07      &       4.07305      &       2.37870e-06        &          3.46885\\
400   &    1.35011e-08     &      3.85145       &      9.40446e-08      &       3.98632      &       1.50103e-06        &          3.44790\\
450   &    8.61042e-09     &      3.81889       &      5.82451e-08      &       4.06772      &       8.89326e-07        &          4.44413\\
500   &    5.70038e-09     &      3.91456       &      3.85174e-08      &       3.92511      &       6.24703e-07        &          3.35218\\
550   &    3.84508e-09     &      4.13112       &      2.61789e-08      &       4.05159      &       4.06400e-07        &          4.51094\\
\hline
    \end{tabular}
    \vspace{3pt}
    \caption{Numerical results for the fluid density using the \rev{scheme with degree $N=3$}.}
    \label{table: 1Disen_table1}
\end{table}
\begin{table}[htbp]
    \centering
    \begin{tabular}{|l|l|l|l|l|l|l|}
    \hline
    \multirow{2}{*}{No. of cells} & \multicolumn{6}{|c|}{\rev{Degree $N=4$}} \\
    \cline{2-7}
    & $L^1$ error & Order & $L^2$ error & Order  & $L^{\infty}$ error & Order\\
\hline
50   &    1.54653e-05     &      -             &          8.20833e-05       &        -           &          8.97655e-04           &          -\\ 
100   &    4.43904e-07     &      5.12265       &      2.74783e-06      &       4.90072      &       3.87646e-05        &          4.53335\\
150   &    6.85380e-08     &      4.60760       &      4.35908e-07      &       4.54080      &       4.96279e-06        &          5.06958\\
200   &    1.67303e-08     &      4.90181       &      1.16190e-07      &       4.59605      &       1.93661e-06        &          3.27108\\
250   &    5.46325e-09     &      5.01551       &      4.00874e-08      &       4.76898      &       6.33049e-07        &          5.01088\\
300   &    2.32497e-09     &      4.68587       &      1.67012e-08      &       4.80242      &       2.57334e-07        &          4.93729\\
350   &    9.99547e-10     &      5.47622       &      7.82249e-09      &       4.92035      &       1.31791e-07        &          4.34094\\
400   &    5.40976e-10     &      4.59763       &      4.07498e-09      &       4.88378      &       7.71896e-08        &          4.00617\\
450   &    3.07294e-10     &      4.80178       &      2.28215e-09      &       4.92215      &       4.31989e-08        &          4.92812\\
500   &    1.81732e-10     &      4.98547       &      1.34171e-09      &       5.04151      &       2.37150e-08        &          5.69195\\
\hline
    \end{tabular}
    \vspace{3pt}
    \caption{Numerical results for the fluid density using the \rev{scheme with degree $N=4$}.}
    \label{table: 1Disen_table2}
\end{table}
\subsubsection{Isentropic smooth flow problem}
We now consider one more test case to check the accuracy of our scheme which is non-trivial compared to the last case, where the solution was just advected. This test is used in~\cite{zhang2006ram,marti2015grid,bhoriya2020entropy} to check the convergence of the high-order schemes, and a similar test case can also be found in~\cite{colella2006cartesian}. In this test, we introduce a symmetric pulse in the initial profiles of fluid density, velocity, and pressure; which moves in a nontrivial way with time, forming a very high gradient structure on one side, which ultimately forms a shock. Here we consider the case till time $t=0.8$ before the shock formation time, where the solution is still smooth. The initial fluid density is
\[
    \rho(x) = \rho_{\text{ref}} [1+ \alpha f(x)]
\]
where
\[
\rho_{\text{ref}} = 1, \qquad v_{\text{ref}}=0, \qquad p_{\text{ref}} = 100, \qquad
    f(x) =\begin{cases}
        \left[\left( \frac{x}{L}\right)^2 - 1 \right]^4, &\textrm{if}\ |x|<L\\
        0, & \textrm{otherwise}
    \end{cases}
\]
with the amplitude and width of the pulse being $\alpha = 1$ and $L = 0.3$, respectively. The initial pressure is computed from $p = C \rho^\gamma$ assuming the quantity $C$ as a constant, and the initial velocity inside the pulse is computed assuming the Riemann invariant
\[
    J = \frac{1}{2} \ln{\left(\frac{1+v_1}{1-v_1}\right)} - \frac{1}{\sqrt{\gamma - 1}} \ln{\left(\frac{\sqrt{\gamma - 1}+ s}{\sqrt{\gamma - 1} - s}\right)}, \qquad s = \textrm{speed of sound}
\]
is constant throughout the computational domain $[-0.35, 1]$. The exact solution can be found by the standard characteristic analysis~\cite{zhang2006ram, morsony2007temporal}.

From Table~\ref{table: 1Disen_table1} and Table~\ref{table: 1Disen_table2} we can see that the method with solution polynomial degree $N=3,4$ converges with the expected rate of $O(\Delta x)^{N+1}$. We have also shown the solutions with $50$ cells in Figure~\ref{1D_isentropic} for graphical visualisation of the result.

\begin{figure}[htbp]
    \centering
    
    \begin{subfigure}{0.49\textwidth}
    \includegraphics[width=\linewidth]{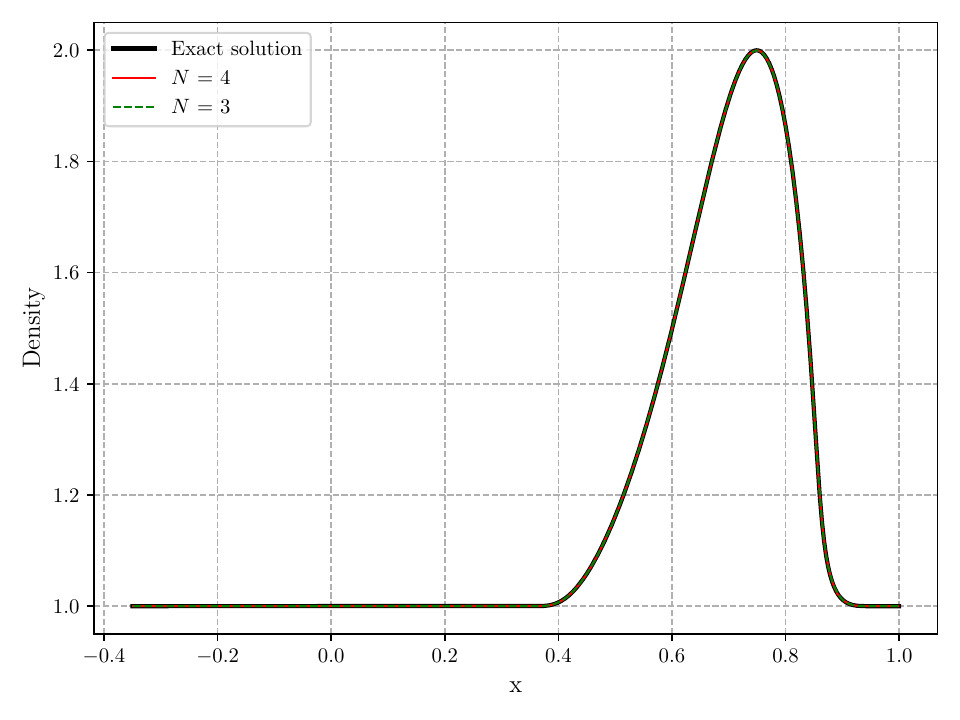}
    \caption{Density plot.}
    \end{subfigure}
    \begin{subfigure}{0.49\textwidth}
    \includegraphics[width=\linewidth]{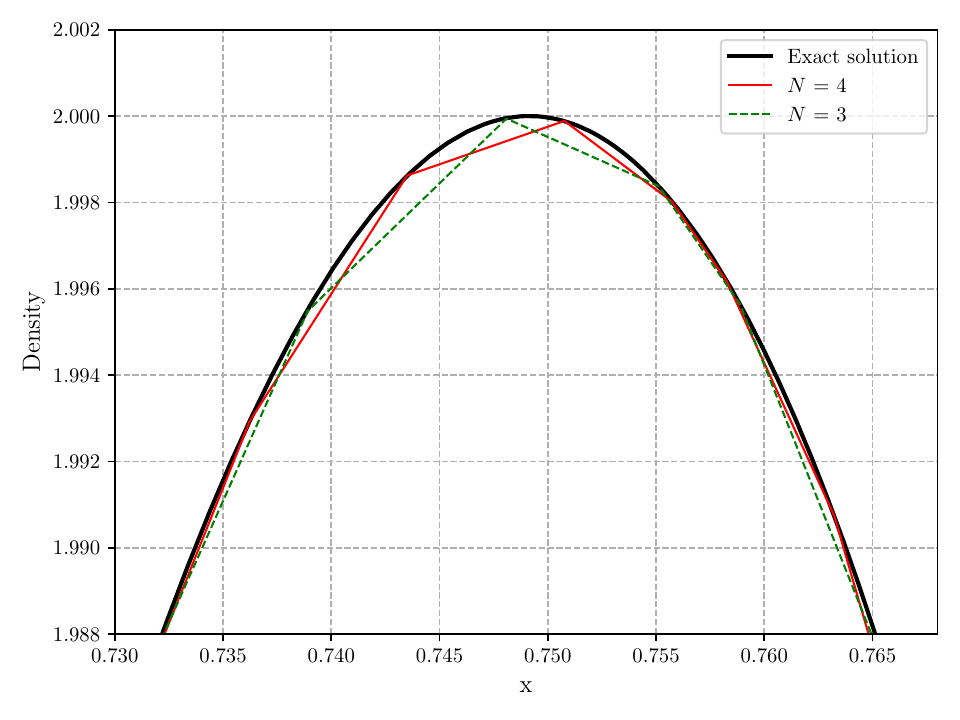}
    \caption{Density plot by zooming in $[0.730, 0.768]\times [1.988, 2.002]$.}
    \end{subfigure}
    \caption{Isentropic smooth flow problem in 1D: Plot of fluid density \rev{using the scheme with degrees $N=3,4$} and $50$ cells.}\label{1D_isentropic}
\end{figure}

}

\begin{figure}[htbp]
    \centering
    
    \begin{subfigure}{0.49\textwidth}
    \includegraphics[width=\linewidth]{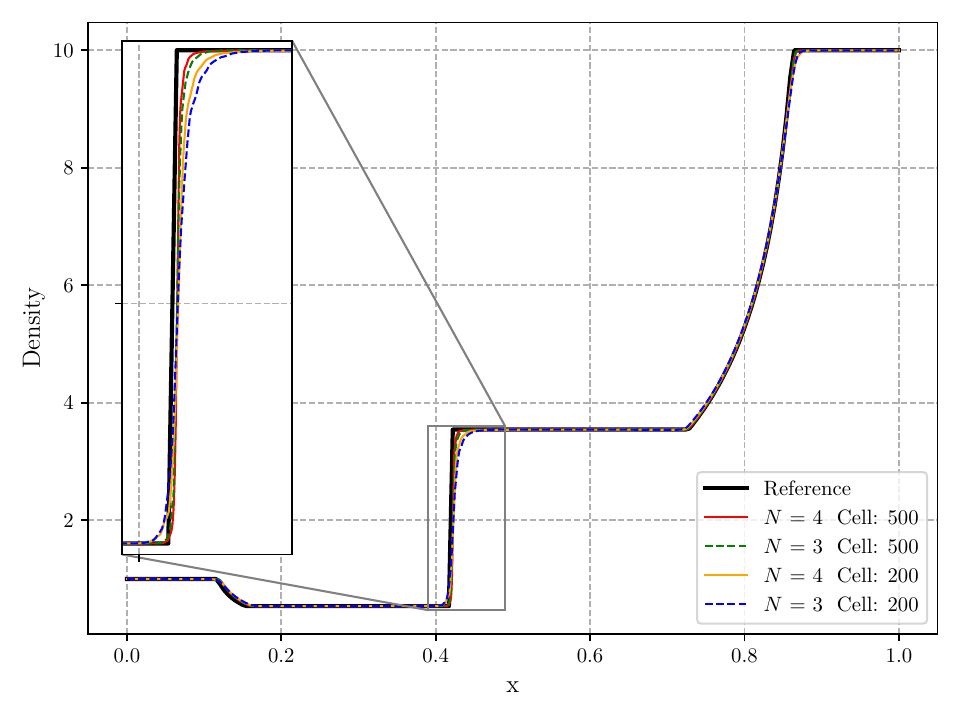}
    \caption{Density plot.}
    \end{subfigure}
    \begin{subfigure}{0.49\textwidth}
    \includegraphics[width=\linewidth]{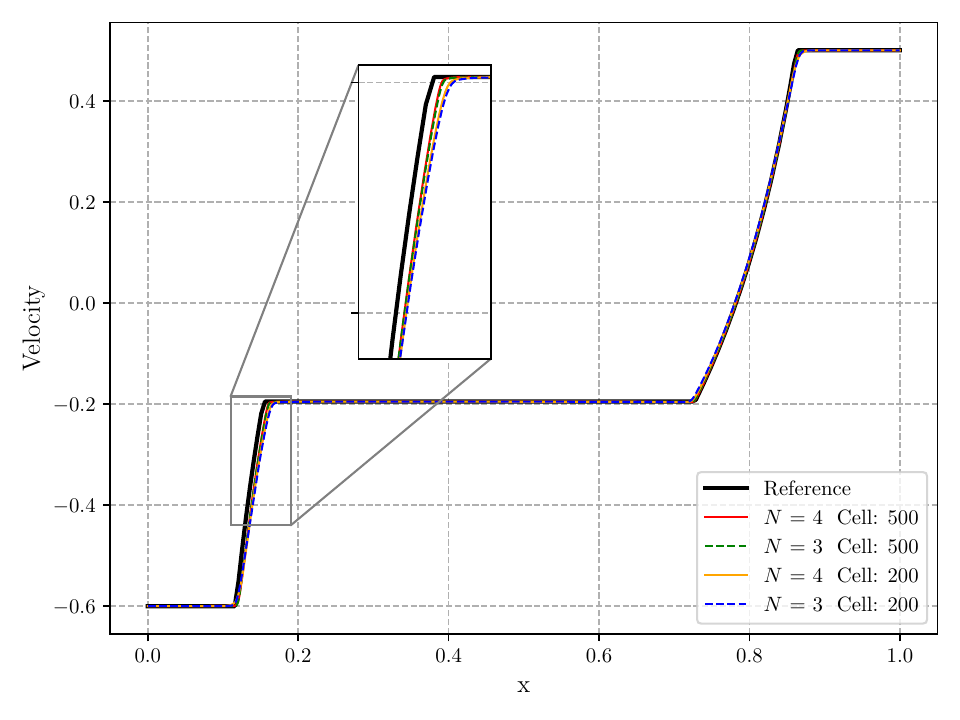}
    \caption{Velocity plot.}
    \end{subfigure}
    \begin{subfigure}{0.49\textwidth}
    \includegraphics[width=\linewidth]{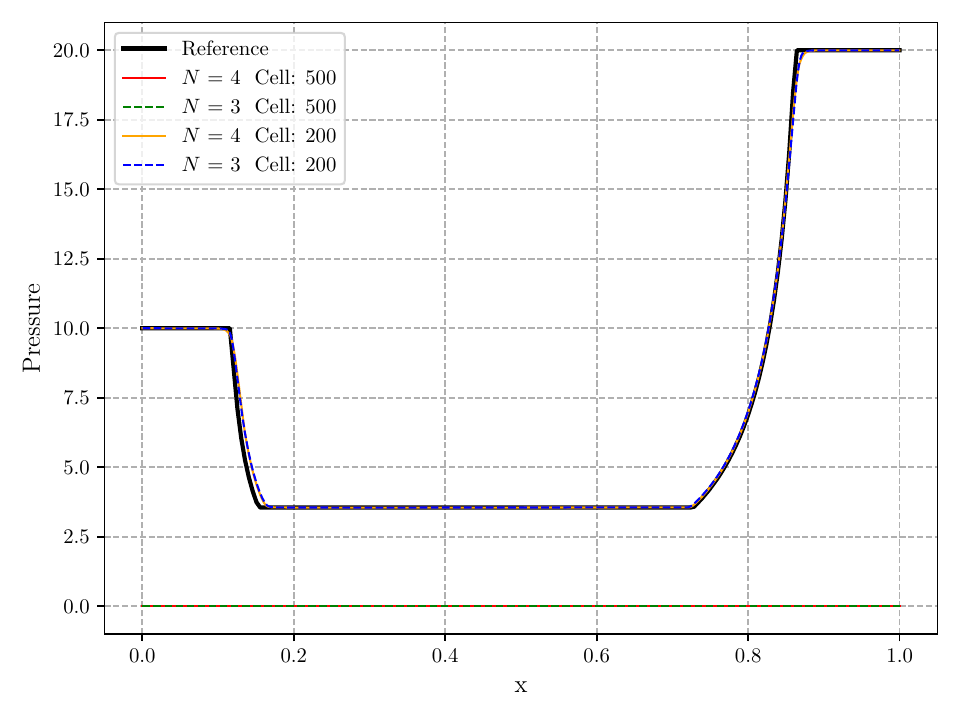}
    \caption{Pressure plot.}
    \end{subfigure}
    \caption{First Riemann problem in 1D: Plot of fluid density, pressure, and velocity \rev{using the scheme with degrees $N=3,4$} with $200$ and $500$ cells.}\label{RP1}
\end{figure}
\subsubsection{First Riemann problem}
This test case is taken from~\cite{mignone2005hllc}, where the discontinuity evolves with time, forming two rarefaction waves moving in the left and right directions and a contact discontinuity in the middle. The initial state of the fluid is given by,
\[
(\rho, v_1, p)= \begin{cases}
    (1, -0.6, 10) &  \textrm{if } x<0.5\\
    (10, 0.5, 20) &  \textrm{if } x>0.5.
\end{cases}
\]
We consider the outflow boundary conditions at the endpoints, $x=0,1$. The simulation is run till time $t=0.4$ \rev{using the scheme with polynomial degrees $N=3,4$ and with $200$ and $500$ cells}. The comparison is shown in Figure~\ref{RP1}. For this test case, we have found the exact solution following~\cite{marti2003numerical}, which we have used as a reference solution.

From Figure~\ref{RP1}, we observe that the solution approaches the reference solution with increasing number of cells and solution polynomial degree. We also observe that the scheme can adequately capture the rarefaction and contact waves.

\begin{figure}[htbp]
    \centering
    
    \begin{subfigure}{0.49\textwidth}
    \includegraphics[width=\linewidth]{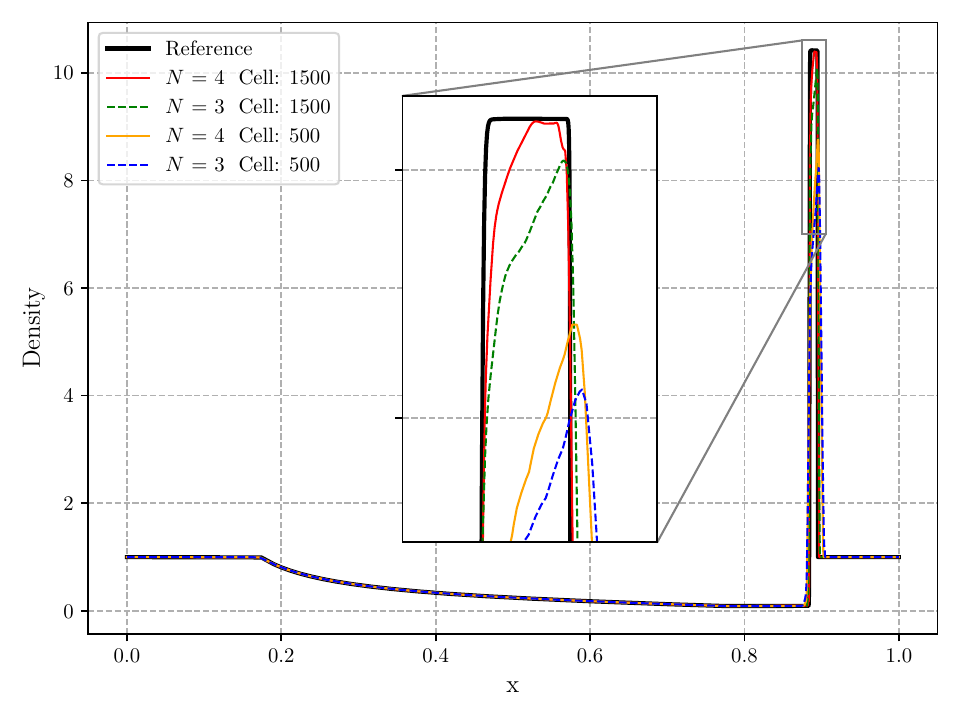}
    \caption{Density plot.}
    \end{subfigure}
    \begin{subfigure}{0.49\textwidth}
    \includegraphics[width=\linewidth]{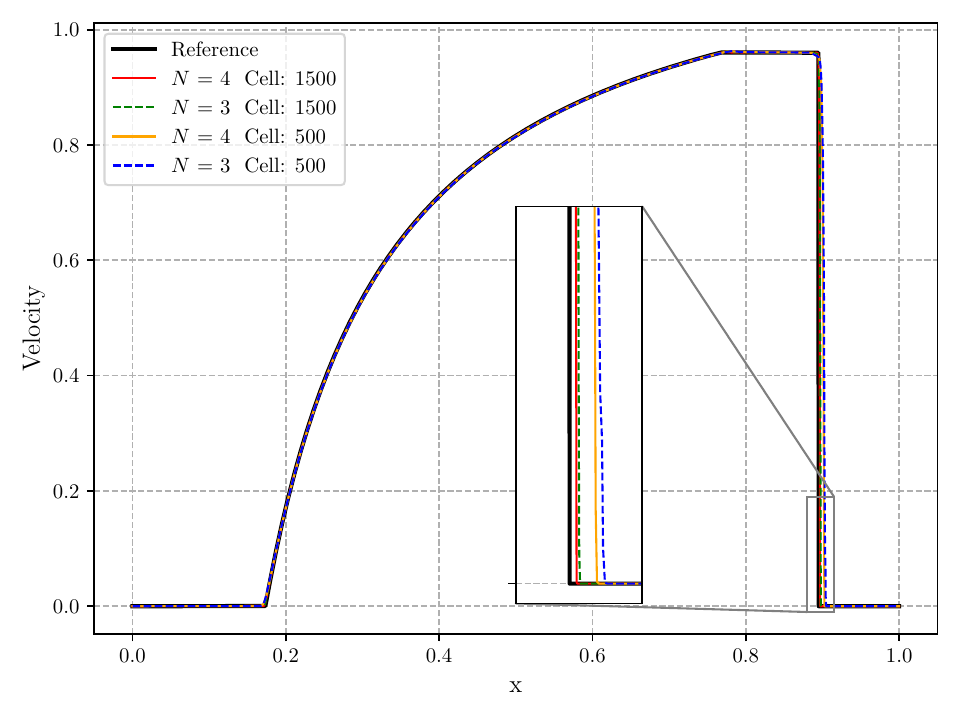}
    \caption{Velocity plot.}
    \end{subfigure}
    \begin{subfigure}{0.49\textwidth}
    \includegraphics[width=\linewidth]{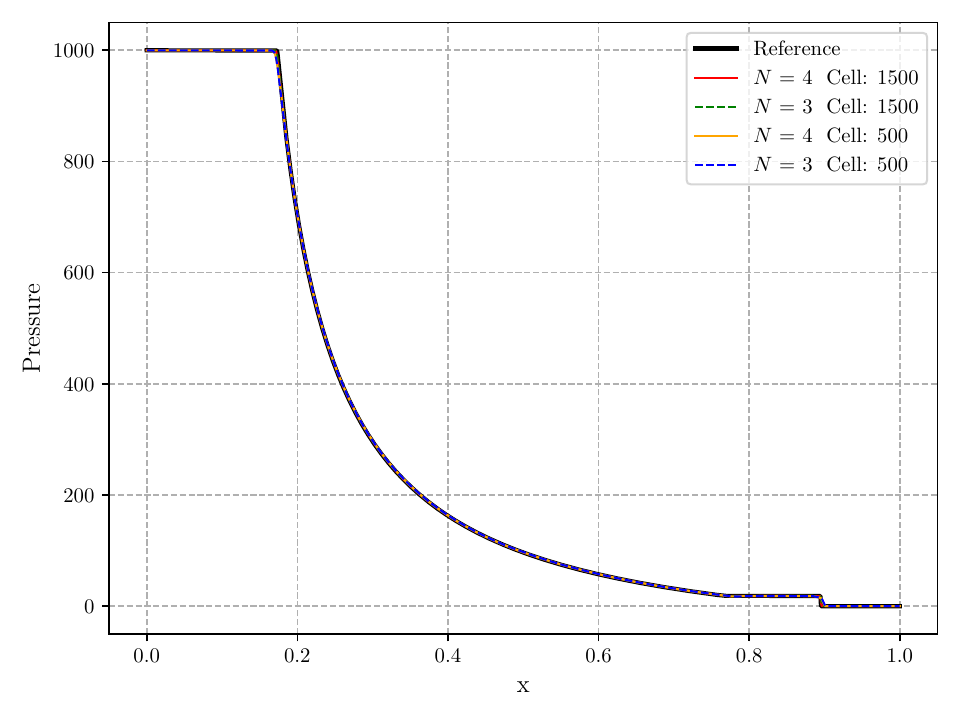}
    \caption{Pressure plot.}
    \end{subfigure}
    \caption{Second Riemann problem in 1D: Plot of fluid density, pressure, and velocity \rev{using the scheme with polynomial degrees $N=3,4$} with $500$ and $1500$ cells.}
    \label{RP2}
\end{figure}
\subsubsection{Second Riemann problem}
We consider this test case from~\cite{marti2003numerical}, which is a good test to check the accuracy of the scheme as the post-shock state in the solution has a narrow strip, which is difficult to capture with diffusive schemes. We compare our results with the results obtained by a second-order finite volume scheme~\cite{van1984relation, berthon2006muscl} using $100000$ cells as the reference solution. The initial condition is given by,
\[
(\rho, v_1, p)= \begin{cases}
    (1, 0, 10^3) & \text{if}\ x<0.5\\
    (1, 0, 10^{-2}) & \text{if}\ x>0.5
\end{cases}
\]
with the boundaries as the outflow boundaries. We show the results in Figure~\ref{RP2} after running the simulation till $t=0.4$ \rev{using the scheme with polynomial degrees $N=3,4$} with $500$ and $1500$ cells to capture the peak of the narrow strip.
 
We observe that our solution converges to the reference solution, attaining the narrow strip in the density profile by increasing the number of cells.

\subsubsection{Third Riemann problem}
This test case is again taken from~\cite{mignone2005hllc}, which has rarefaction, contact, and shock waves in its solution, making it a good test case to check the robustness of the scheme. For this case as well, we have found the exact solution following~\cite{marti2003numerical} and used it as the reference solution. The initial state of the fluid is given by,
\[
(\rho, v_1, p)= \begin{cases}
    (10, 0, \frac{40}{3}) & \text{if}\ x<0.5\\
    (1, 0, 10^{-6}) & \text{if}\ x>0.5
\end{cases}
\]
with the outflow boundary conditions at $x=0,1$. We have run the simulation till time $t=0.4$ \rev{using the scheme with polynomial degrees $N=3,4$} with $200$ and $500$ cells and $l_s = 0.75$. The comparison is shown in Figure~\ref{RP3}.
\begin{figure}[htbp]
    \centering
    \begin{subfigure}{0.49\textwidth}
    \includegraphics[width=\linewidth]{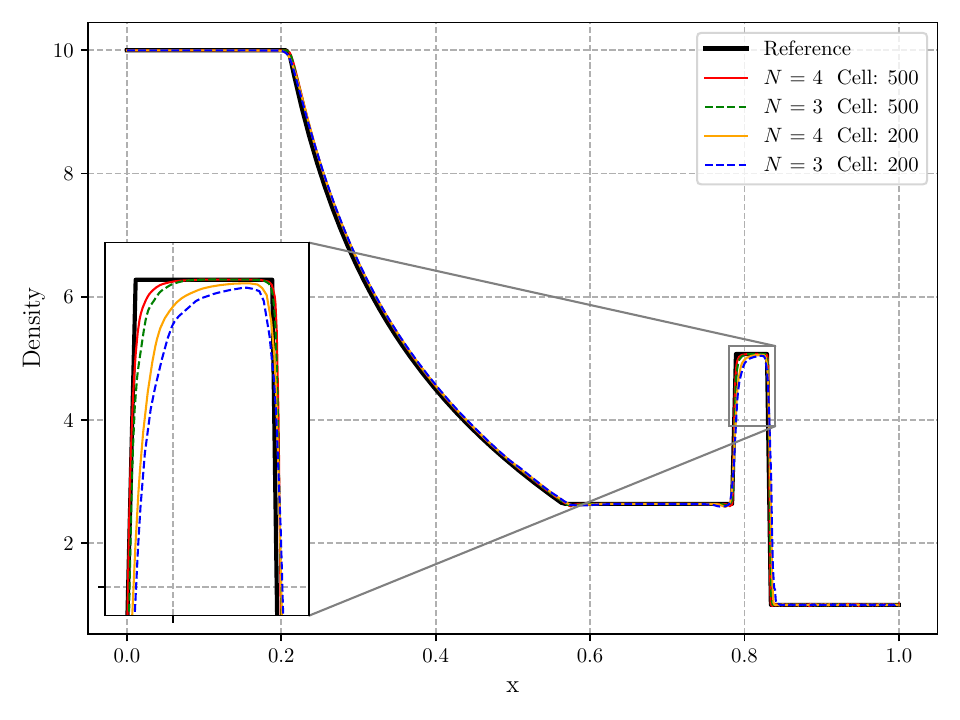}
    \caption{Density plot.}
    \end{subfigure}
    \begin{subfigure}{0.49\textwidth}
    \includegraphics[width=\linewidth]{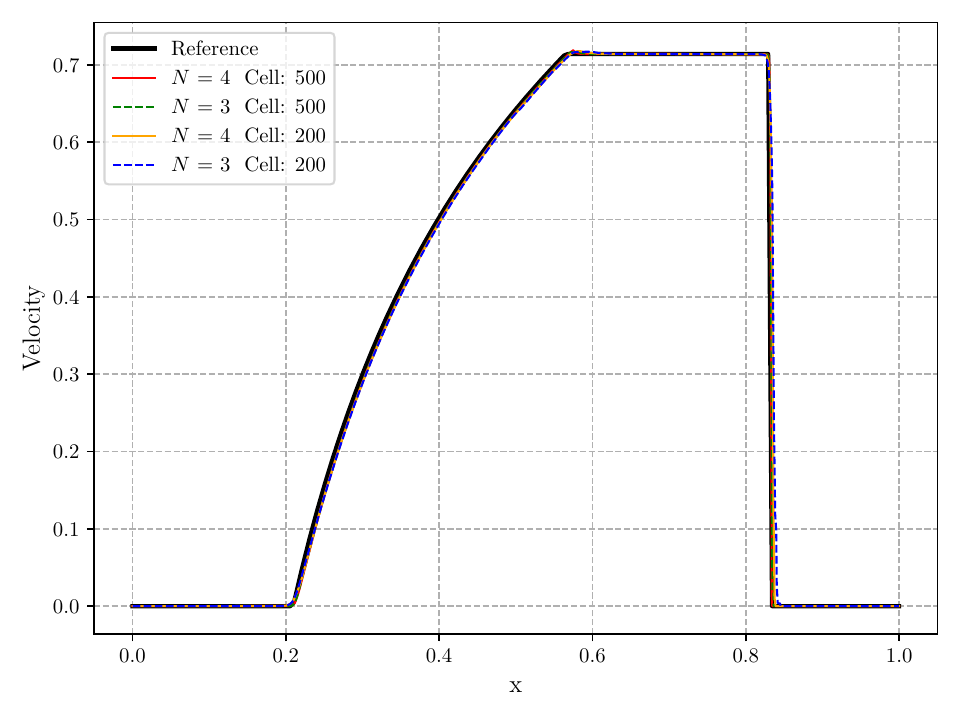}
    \caption{Velocity plot.}
    \end{subfigure}
    \begin{subfigure}{0.49\textwidth}
    \includegraphics[width=\linewidth]{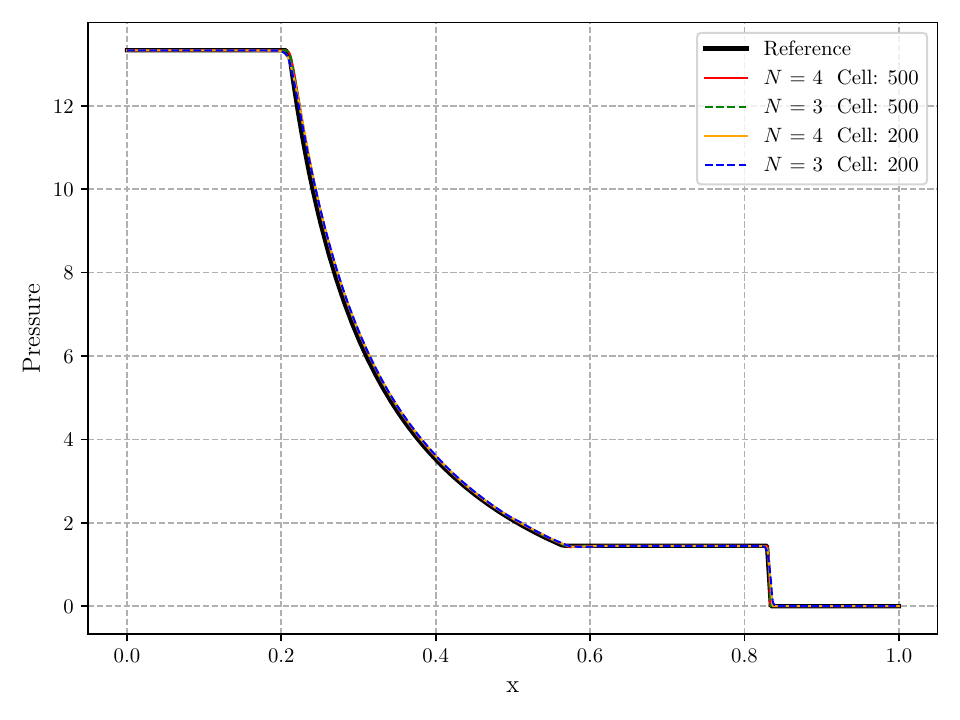}
    \caption{Pressure plot.}
    \end{subfigure}
    \caption{Third Riemann problem in 1D: Plot of fluid density, pressure, and velocity \rev{using the scheme with polynomial degrees $N=3,4$} with $200$ and $500$ cells.}
    \label{RP3}
\end{figure}

\begin{figure}[htbp]
    \centering
    
    \begin{subfigure}{0.49\textwidth}
    \includegraphics[width=\linewidth]{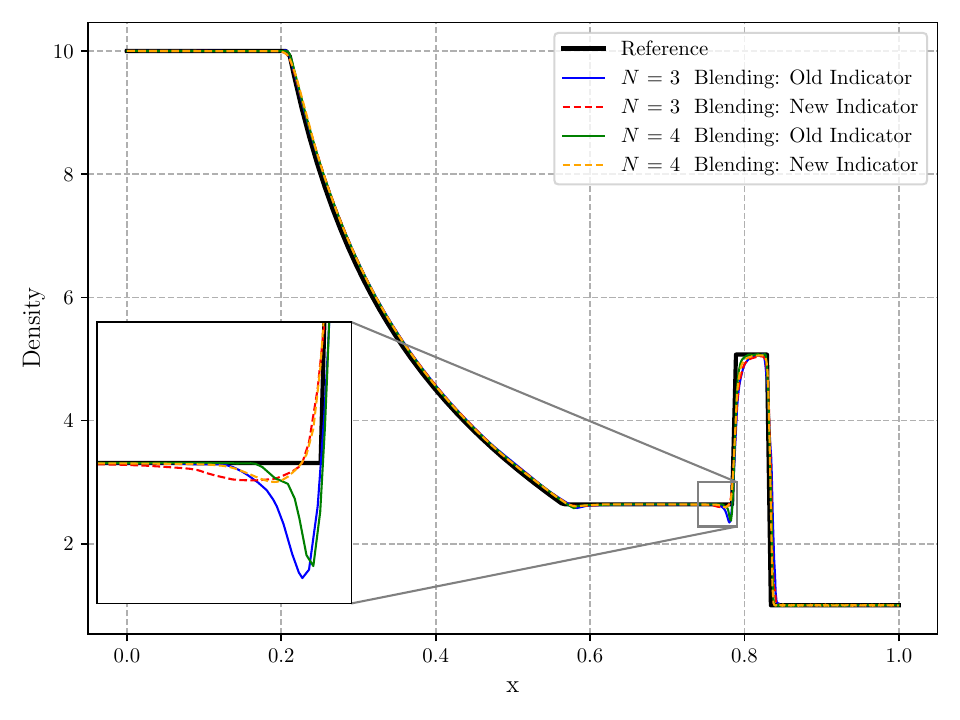}
    \caption{Density plot with $200$ cells.}
    \end{subfigure}
    \begin{subfigure}{0.49\textwidth}
    \includegraphics[width=\linewidth]{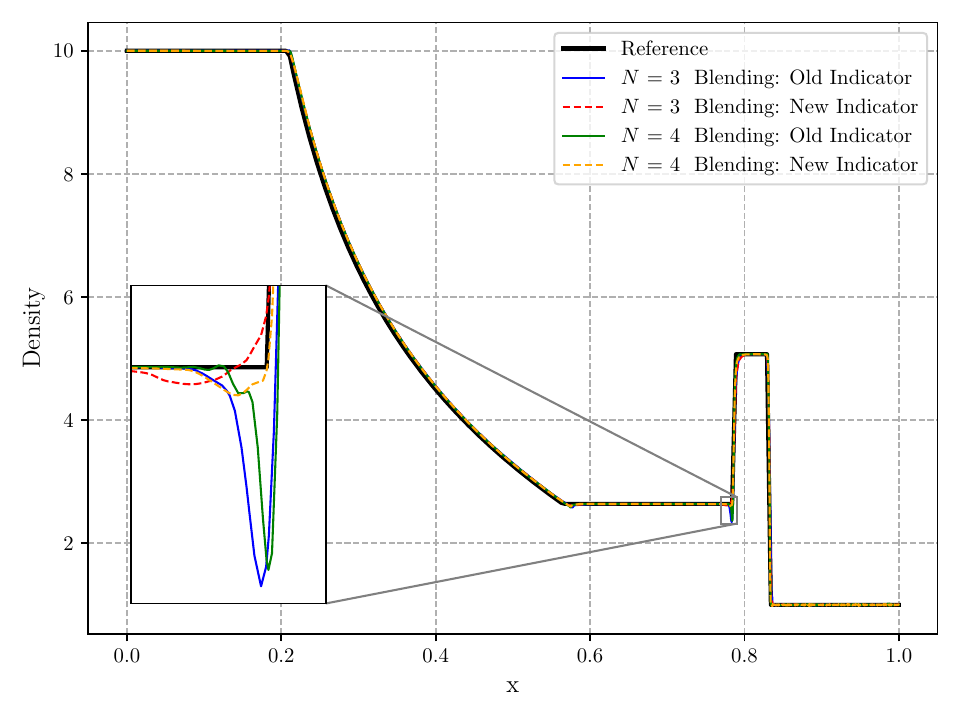}
    \caption{Density plot with $500$ cells.}
    \end{subfigure}
    \caption{Third Riemann problem in 1D: Plot of fluid density and velocity with \rev{using the scheme with polynomial degrees $N=3,4$} with the new indicator (our indicator model) and old indicator (indicator model of~\cite{hennemann2021provably}).}\label{RP3_ind}
\end{figure}
Here again, we can see that the solution approaches the reference solution by increasing the \rev{polynomial degree $N$} and number of cells we use, even though all the approximations can effectively capture the waves.
It is a good test to show the improvement of the results using our discontinuity indicator model. In Figure~\ref{RP3_ind}, we observe that the discontinuity indicator model used in~\cite{hennemann2021provably, babbar2024admissibility} can not effectively detect the discontinuities and hence produces larger oscillations mainly near the contact discontinuity but using our discontinuity indicator model overcomes this issue to a large extent without compromising much with the accuracy of the scheme. From Figure~\ref{RP3_ind}, we can also observe that our discontinuity indicator model works better even with a coarser mesh in controlling the oscillations.

\reva{
\subsubsection{Fourth Riemann problem}\label{sec: fourth_riemann_problem}
This Riemann problem is taken following~\cite{zhang2006ram} where two shock waves get formed in the solution, which moves in opposite directions. The initial state is given by,
\[
    (\rho, v_1, p)= \begin{cases}
    (1, 0.9, 1)\ &\text{if}\ x<0.5\\
    (1, 0, 10)\ &\text{if}\ x>0.5
\end{cases}
\]
with the computational domain as $[0, 1]$ having both outflow boundaries. We have run the simulation till $t=0.4$ using the scheme with degrees $N=3,4$ with $200$ and $500$ cells, and the comparison is shown in Figure~\ref{RP4}. The reference solution is found by a second-order finite volume method~\cite{van1984relation, berthon2006muscl} with $10000$ cells.

\begin{figure}[!htbp]
    \centering
    \begin{subfigure}{0.49\textwidth}
    \includegraphics[width=\linewidth]{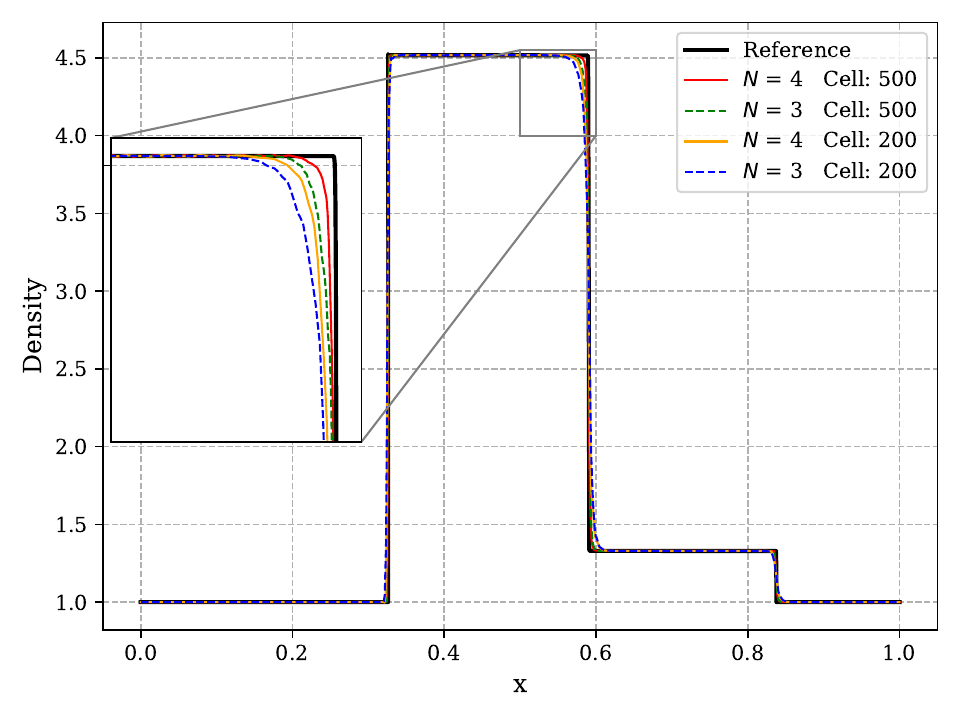}
    \caption{Density plot.}
    \end{subfigure}
    \begin{subfigure}{0.49\textwidth}
    \includegraphics[width=\linewidth]{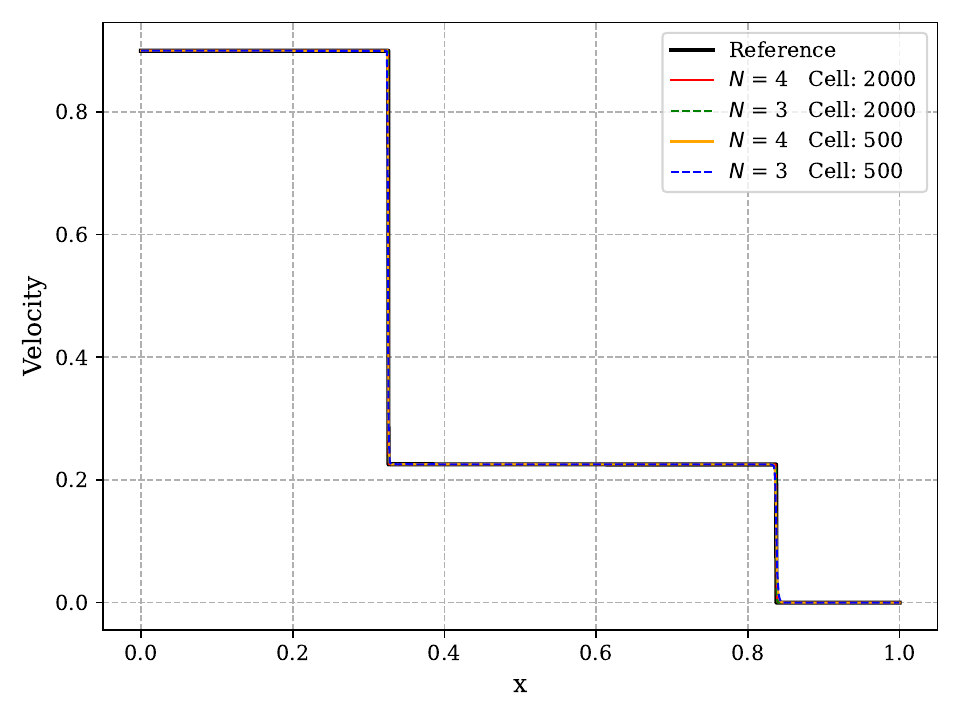}
    \caption{Velocity plot.}
    \end{subfigure}
    \begin{subfigure}{0.49\textwidth}
    \includegraphics[width=\linewidth]{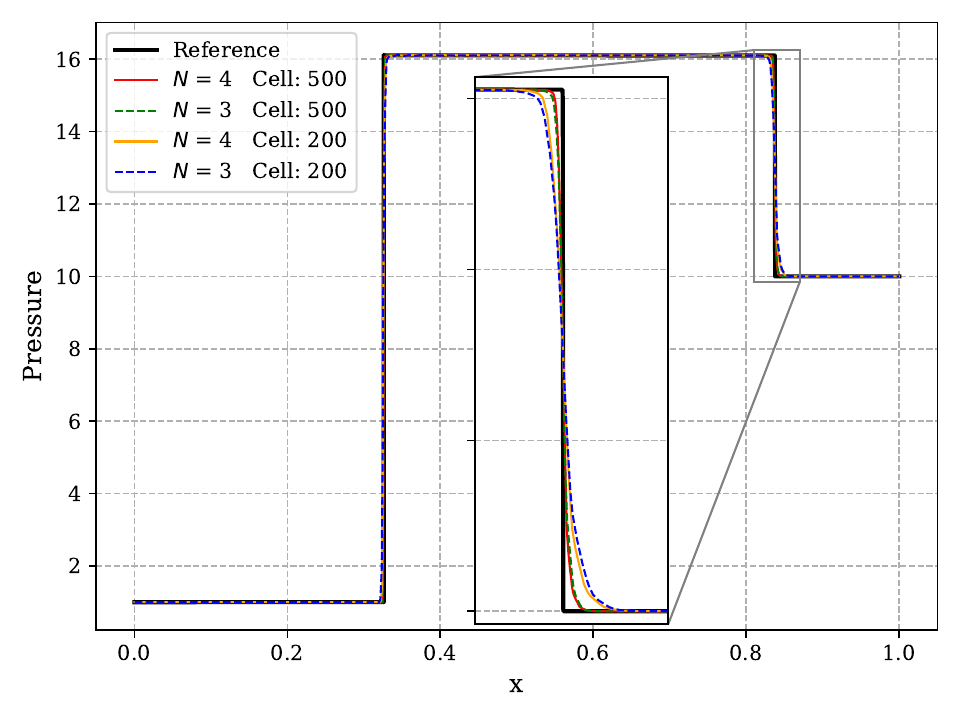}
    \caption{Pressure plot.}
    \end{subfigure}
    \caption{Fourth Riemann problem in 1D: Plot of fluid density, pressure, and velocity using the scheme with polynomial degrees $N=3,4$ with $200$ and $500$ cells.}\label{RP4}
\end{figure}

We observe from Figure~\ref{RP4} that the scheme can capture both the shock waves in the solution, and the accuracy of the result increases with increasing the degree $N$ and number of cells. 
}
\begin{figure}[htbp]
    \centering
    \begin{subfigure}{0.49\textwidth}
    \includegraphics[width=\linewidth]{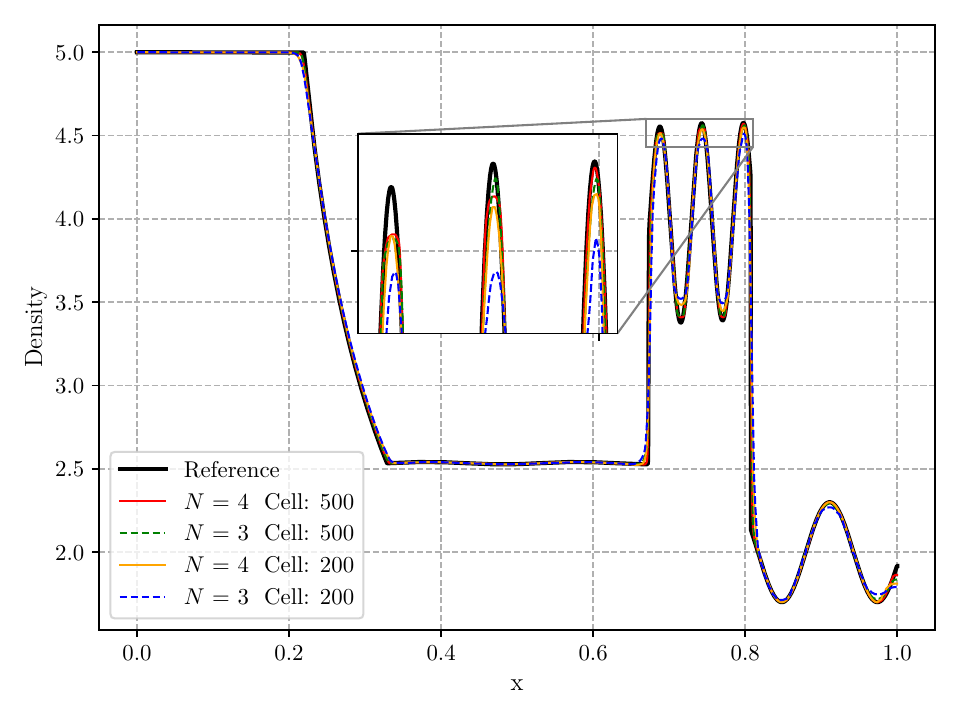}
    \caption{Density plot.}
    \end{subfigure}
    \begin{subfigure}{0.49\textwidth}
    \includegraphics[width=\linewidth]{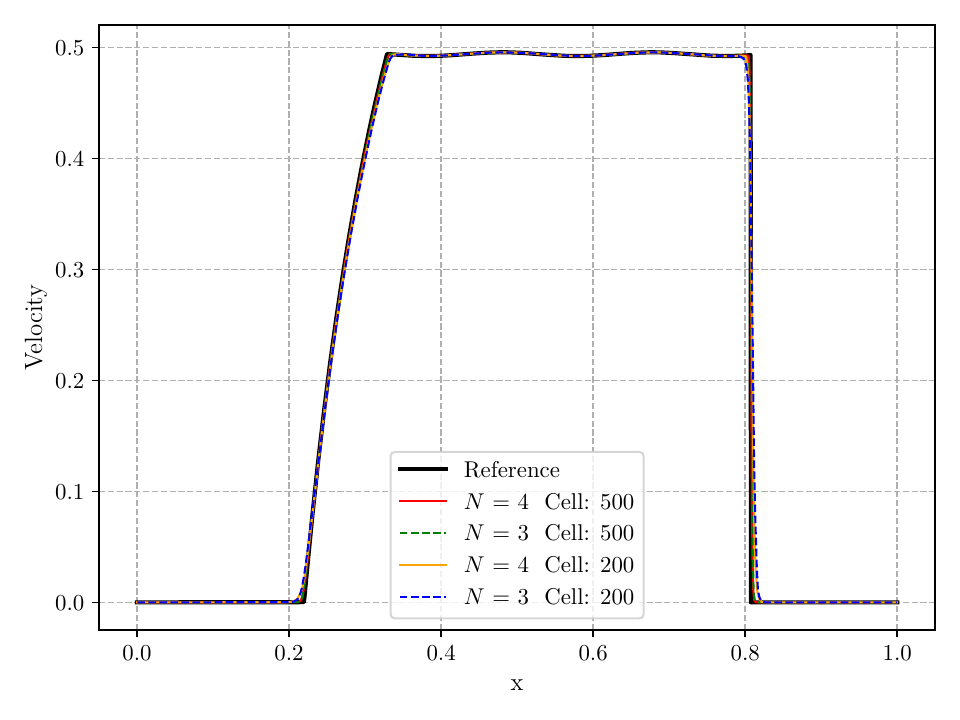}
    \caption{Velocity plot.}
    \end{subfigure}
    \begin{subfigure}{0.49\textwidth}
    \includegraphics[width=\linewidth]{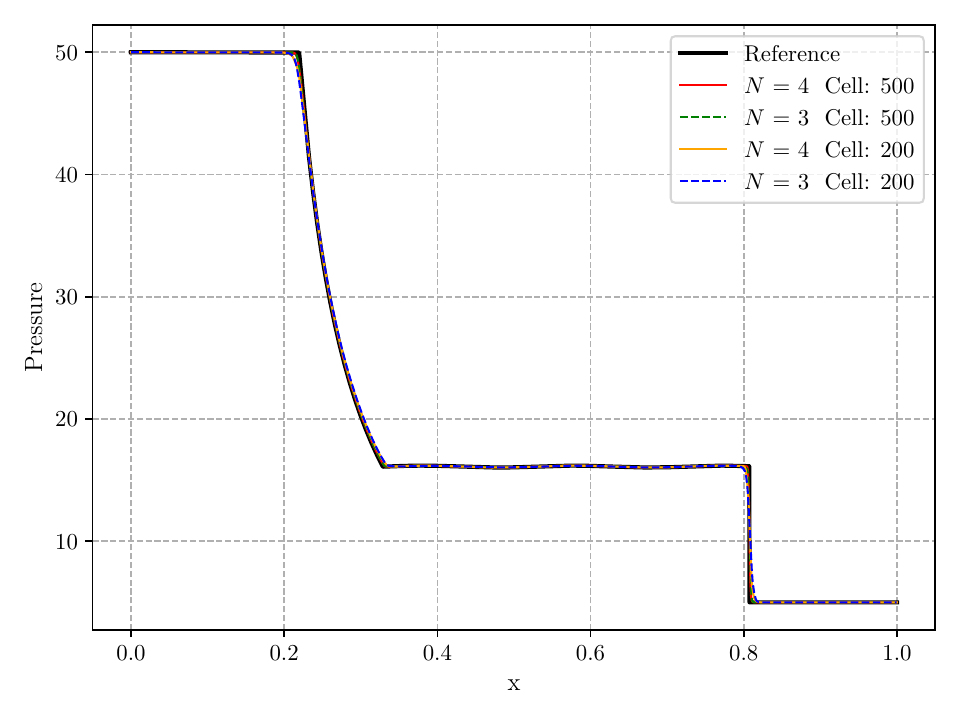}
    \caption{Pressure plot.}
    \end{subfigure}
    \caption{Density perturbation problem in 1D: Plot of fluid density, pressure, and velocity \rev{using the scheme with polynomial degrees $N=3,4$} with $200$ and $500$ cells.}\label{DP}
\end{figure}
\begin{figure}[htbp]
    \centering
    \begin{subfigure}{0.49\textwidth}
    \includegraphics[width=\linewidth]{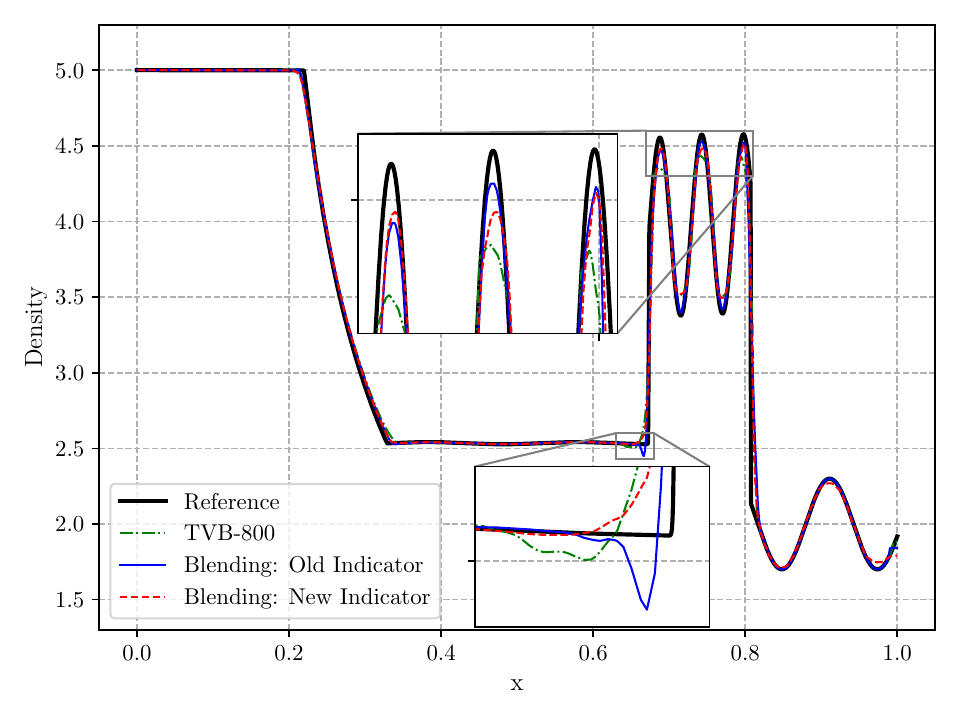}
    \caption{Density plot \rev{using the scheme with degree $N=3$}.}
    \end{subfigure}
    \begin{subfigure}{0.49\textwidth}
    \includegraphics[width=\linewidth]{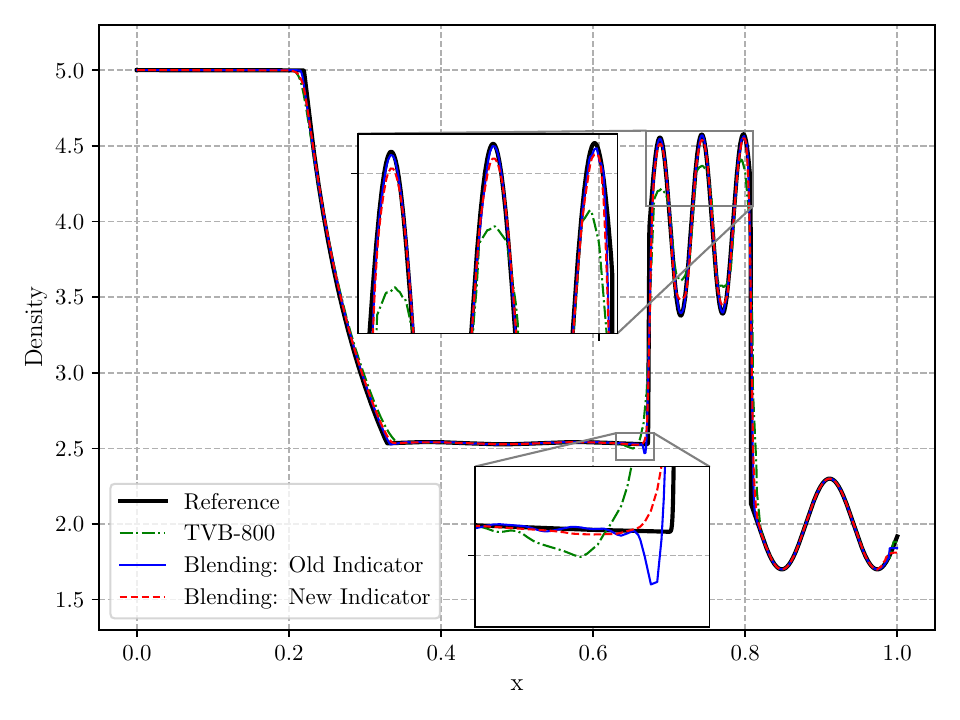}
    \caption{Density plot \rev{using the scheme with degree $N=4$}.}
    \end{subfigure}
    \caption{Density perturbation problem in 1D: Plot of fluid density using TVB limiter and blending limiter with new indicator (our indicator model) and old indicator (indicator model of~\cite{hennemann2021provably}) with 200 cells.}\label{DP_ind_limiter}
\end{figure}
\subsubsection{Density perturbation problem} \label{fourthRP_section}
We have taken this density perturbation test case from~\cite{del2002efficient}, which has both shock waves and smooth waves very close to each other. So, we need to capture the smooth waves properly when controlling the oscillations coming from the discontinuity, making it a very interesting test to check the efficiency of the scheme. The initial conditions are given by,
\[
(\rho, v_1, p)= \begin{cases}
    (5, 0, 50)\ &\text{if}\ x<0.5\\
    (2+0.3\sin(50x), 0, 5)\ &\text{if}\ x>0.5
\end{cases}
\]
and we have run the simulation till time $t=0.35$ with outflow boundaries at $x=0,1$. The comparison of the results is shown in Figure~\ref{DP} with the solution obtained by a second-order finite volume scheme~\cite{van1984relation, berthon2006muscl} using $100000$ cells as the reference solution.

We observe from the inset plot zoomed near the smooth extrema, that the solution approaches the reference solution by increasing the \rev{degree $N$} or the number of cells.

The presence of both shocks and smooth extrema makes this a good problem to show the benefit of the blending limiter over the TVB limiter, which does not capture sub-cell information as discussed in Section~\ref{sec: oscillation_control}. In Figure~\ref{DP_ind_limiter}, we compare the results obtained using the TVB limiter (applied on the characteristic variables) and blending limiter, showing both the new discontinuity indicator model and the indicator model from~\cite{hennemann2021provably}.

We can observe that the TVB limiter starts producing oscillations when we increase the parameter $M$ to $800$ (for details, refer to Section~7 of~\cite{BABBAR2022111423}) and still has inferior performance in attaining the extrema compared to the blending limiter, as expected from the discussion in Section~\ref{sec: oscillation_control}. The blending limiter with old indicator model~\cite{hennemann2021provably} performs better in attaining the extrema but produces oscillations near the contact discontinuity. Our indicator model performs better in controlling those oscillations, and it also attains the smooth extrema reasonably well.

\begin{figure}[htbp]
    \centering
    \begin{subfigure}{0.49\textwidth}
    \includegraphics[width=\linewidth]{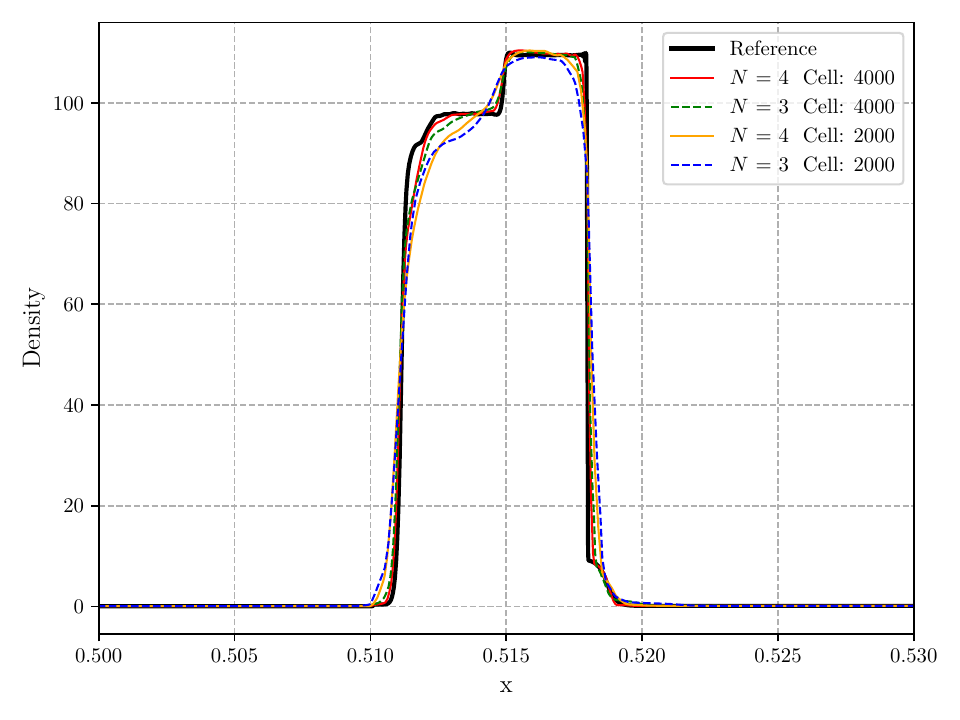}
    \caption{Density plot by zooming in $[0.50, 0.53]$.}
    \end{subfigure}
    \begin{subfigure}{0.49\textwidth}
    \includegraphics[width=\linewidth]{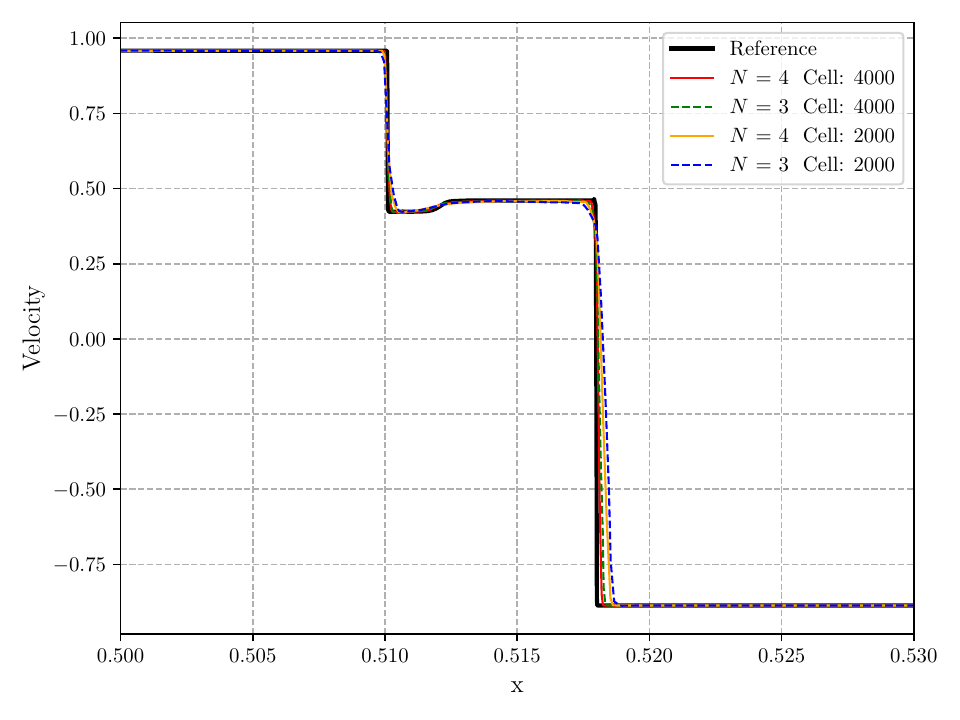}
    \caption{Velocity plot by zooming in $[0.50, 0.53]$.}
    \end{subfigure}
    \begin{subfigure}{0.49\textwidth}
    \includegraphics[width=\linewidth]{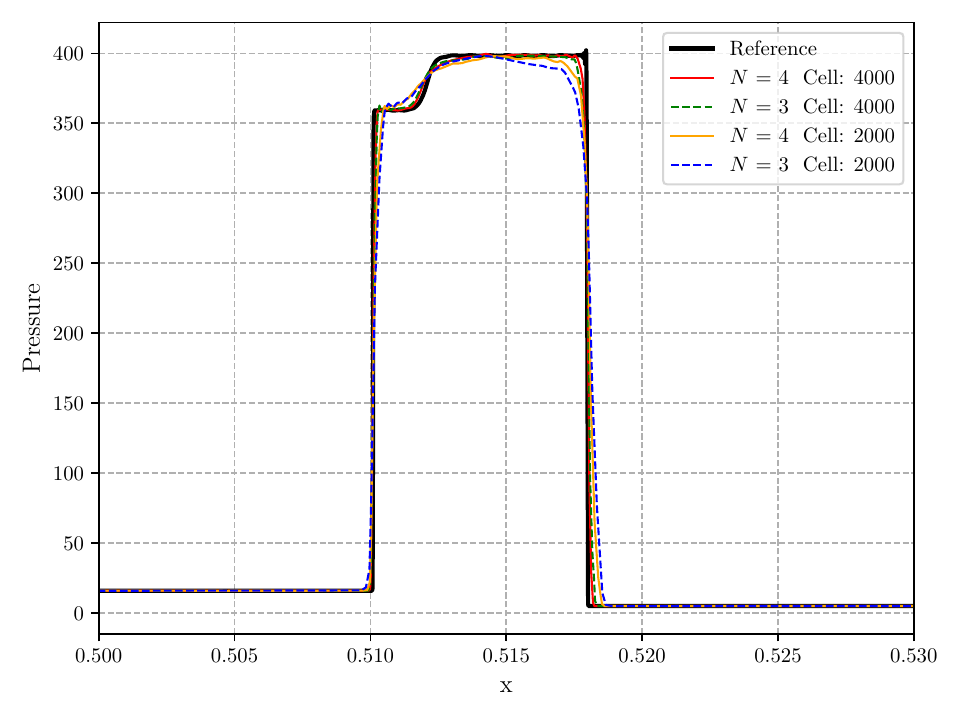}
    \caption{Pressure plot by zooming in $[0.50, 0.53]$.}
    \end{subfigure}
    \caption{Blast wave problem: Plot of fluid density, pressure, and velocity \rev{using the scheme with polynomial degrees $N=3,4$} with $2000$ and $4000$ cells.}\label{blast}
\end{figure}
\subsubsection{Blast wave problem}
This one-dimensional test case is taken from~\cite{marti1996extension}, where two strong shocks interact, and the scheme needs to capture a very narrow structure, proving its efficiency. The initial condition is given by,
\[
(\rho, v_1, p) = \begin{cases}
        (1, 0, 1000) & \text{if}\ x<0.1\\
        (1, 0, 0.01) & \text{if}\ 0.1<x<0.9\\
        (1, 0, 100)  & \text{if}\ x>0.9.
    \end{cases}
\]
We take outflow boundaries at $x = 0, 1$ and $\gamma = 1.43$ for this test. Here, the two blast waves will propagate with a speed nearing the speed of light in opposite directions, and finally, its interaction produces very narrow structures at time $t=0.43$. We have taken $2000$ and $4000$ cells in this test to capture these extremely small structures with $l_s = 0.75$. The results are shown in Figure~\ref{blast} by zooming in the interval $[0.50, 0.53]$. Here we have taken the solution obtained by a second-order finite volume scheme~\cite{van1984relation, berthon2006muscl} using $100000$ cells as the reference solution.

We observe that the scheme can capture small-scale structures more effectively with $4000$ cells than the $2000$ cells, as expected. Moreover, the indicator model effectively detects the discontinuities, producing more accurate and less oscillatory results compared to the entropy stable schemes~\cite{biswas2022entropy,bhoriya2020entropy}.

\subsection{2-D simulations}
We now perform two-dimensional tests to check the efficiency and robustness of the scheme. Here the velocity components $v_1, v_2$ are separately interpolated (spatially), making the constraint $|\mb{v}|<1$ more vulnerable. Multidimensional cases are also more important in the sense that for some one-dimensional Riemann problems, we can find exact solvers, which is not available in the case of higher dimensions.
\subsubsection{Isentropic vortex problem}
We test the accuracy of the scheme in the two-dimensional case with an isentropic vortex problem~\cite{ling2019physical, duan2019high} where a vortex moves diagonally with a constant velocity and returns to the initial position after a finite time. The initial density and pressure are given by,
\[
\rho(x,y) = \left(1 - a\exp(1-r^2)\right)^{\frac{1}{\gamma -1}},\qquad p(x,y) = \rho(x,y)^\gamma,
\]
where
\begin{align*}
    &a = \frac{(\gamma - 1)}{\gamma}\frac{\epsilon^2}{8\pi^2}, \qquad r = \sqrt{x_0^2 + y_0^2},\\
    &x_0 = x + \frac{\phi -1}{2}(x+y), \qquad y_0 = y + \frac{\phi -1}{2}(x+y), \qquad \phi = \frac{1}{\sqrt{1-w^2}}.
\end{align*}
The initial velocities in $x$ and $y$-directions are given by,
\begin{align*}
    &v_1 = \frac{1}{1-\frac{w(v_1'+v_2')}{\sqrt{2}}}\left(\frac{v_1'}{\phi} - \frac{w}{\sqrt{2}}+\frac{\phi w^2}{2(\phi +1)}(v_1'+v_2') \right),\\
    &v_2 = \frac{1}{1-\frac{w(v_1'+v_2')}{\sqrt{2}}}\left(\frac{v_2'}{\phi} - \frac{w}{\sqrt{2}}+\frac{\phi w^2}{2(\phi +1)}(v_1'+v_2') \right),
\end{align*}
where
\begin{align*}
    (v_1',v_2')=(-y_0, x_0) \sqrt{\frac{b}{1+br^2}},\qquad b=\frac{2\gamma a \exp(1-r^2)}{2\gamma - 1 -\gamma a \exp(1-r^2)}.
\end{align*}
Here, the vortex strength is taken as $\epsilon = 5$ and $w=0.5\sqrt{2}$.

We run the simulation in the domain $[-20, 20] \times [-20, 20]$ with periodic boundary conditions in both directions till time $t=80$ when the vortex returns to the initial position. The errors in relativistic density versus grid resolutions are presented in Tables~\ref{table_1_2d},~\ref{table_2_2d} for degrees $N=3,4$ respectively. We see that the scheme attains close to the optimal order of accuracy for both degrees \reva{as we decrease the mesh size}.

\begin{table}[htbp]
    \centering
    \begin{tabular}{|l|l|l|l|l|l|l|}
    \hline
    \multirow{2}{*}{No. of cells} & \multicolumn{6}{|c|}{\rev{Degree $N=3$}} \\
    \cline{2-7}
    & $L^1$ error & Order & $L^2$ error & Order  & $L^{\infty}$ error & Order\\
    \hline
         20$\times$20 & 2.39771e-3 & - & 9.12451e-3 & - & 3.77608e-1 & - \\
         40$\times$40 & 7.35697e-4 & 1.70447 & 1.91447e-3 & 2.25280  & 8.38738e-2 & 2.17060 \\
         80$\times$80 & 6.22063e-5 & 3.56398 & 2.04361e-4 & 3.22775 & 1.10895e-2 & 2.91903 \\
         160$\times$160 & 2.02681e-6 & 4.93978 & 7.00402e-6 & 4.86679 & 6.26013e-4 & 4.14685 \\
         \reva{320$\times$320} &  \reva{5.00244e-8} & \reva{5.34043} & \reva{2.00780e-7} & \reva{5.12449} & \reva{1.26247e-5} & \reva{5.63187} \\
         \reva{640$\times$640} & \reva{1.64023e-9} & \reva{4.93066} & \reva{1.00213e-8} & \reva{4.32448} & \reva{3.94292e-7}	& \reva{5.00085} \\
         \hline
    \end{tabular}
    \vspace{3pt}
    \caption{Numerical results using the \rev{scheme with degree $N=3$}.}
    \label{table_1_2d}
\end{table}

\begin{table}[htbp]
    \centering
    \begin{tabular}{|l|l|l|l|l|l|l|}
    \hline
    \multirow{2}{*}{No. of cells} & \multicolumn{6}{|c|}{\rev{Degree $N=4$}} \\
    \cline{2-7}
    & $L^1$ error & Order & $L^2$ error & Order  & $L^{\infty}$ error & Order\\
    \hline
         $20\times 20$ & 1.43133e-3 & - & 3.34925e-3 & - & 6.83311e-2  & - \\
         $40\times 40$ & 1.36445e-4 & 3.39097 & 6.48346e-4 & 2.36900 & 3.94521e-2 & 0.79244 \\
         $80\times 80$ & 2.89300e-6 & 5.55960 & 1.09788e-5 & 5.88397 & 1.19125e-3 & 5.04956 \\
         $160\times 160$ & 9.88368e-8 & 4.87137 & 2.38657e-7 & 5.52364 & 1.52526e-5 & 6.28728 \\
         \hline
    \end{tabular}
    \vspace{3pt}
    \caption{Numerical results using the \rev{scheme with degree $N=4$}.}
    \label{table_2_2d}
\end{table}

We have also plotted the solution with respect to the diagonal from bottom-left to upper-right corner by projecting it to the $x$-axis in Figure~\ref{fig: IV} using $80\times 80$ cells to see the comparison with the exact solution.
\begin{figure}[!htbp]
    \centering
    \begin{subfigure}{0.49\textwidth}
    \includegraphics[width=\linewidth]{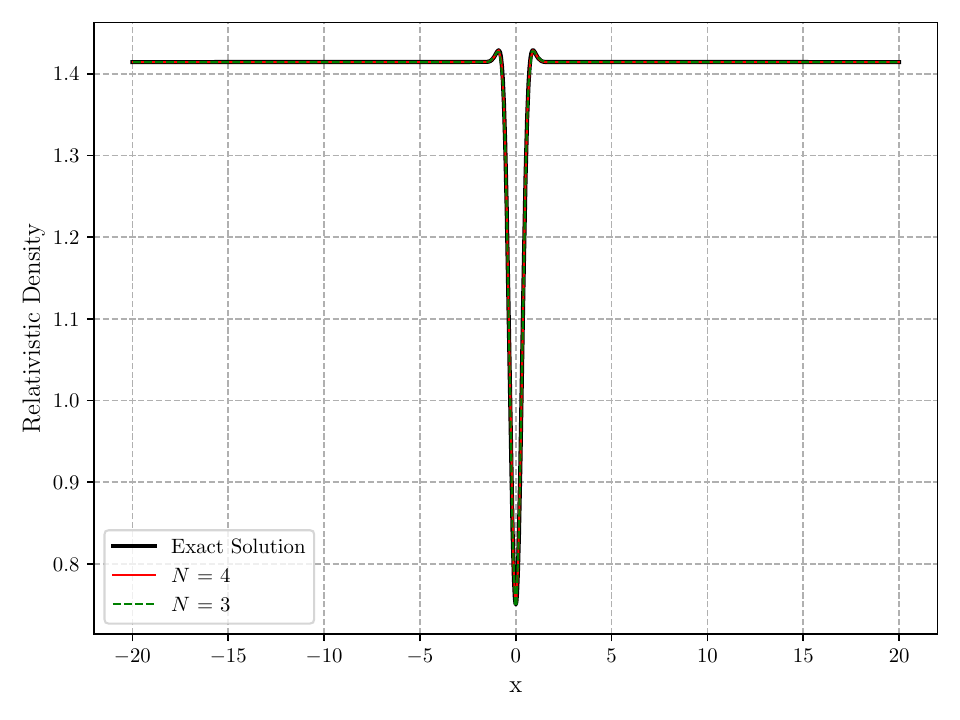}
    \caption{Plot of relativistic density.}
    \end{subfigure}
    \begin{subfigure}{0.49\textwidth}
    \includegraphics[width=\linewidth]{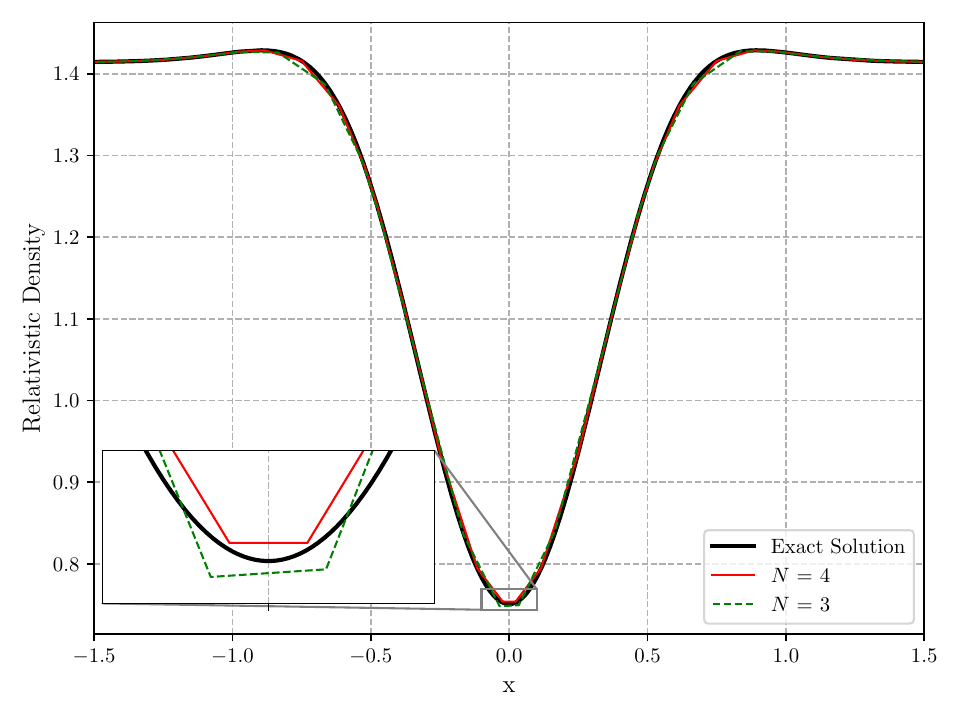}
    \caption{Plot of relativistic density by zooming in $[-1.5, 1.5]$.}
    \end{subfigure}
    \caption{Isentropic vortex problem in 2D: Line plot of relativistic density with respect to a diagonal by projecting it to $x$-axis \rev{using the scheme with polynomial degrees $N=3,4$} with $80\times 80$ cells.}\label{fig: IV}
\end{figure}

\begin{figure}
    \centering
    \begin{subfigure}{0.45\textwidth}
         \includegraphics[width=\linewidth]{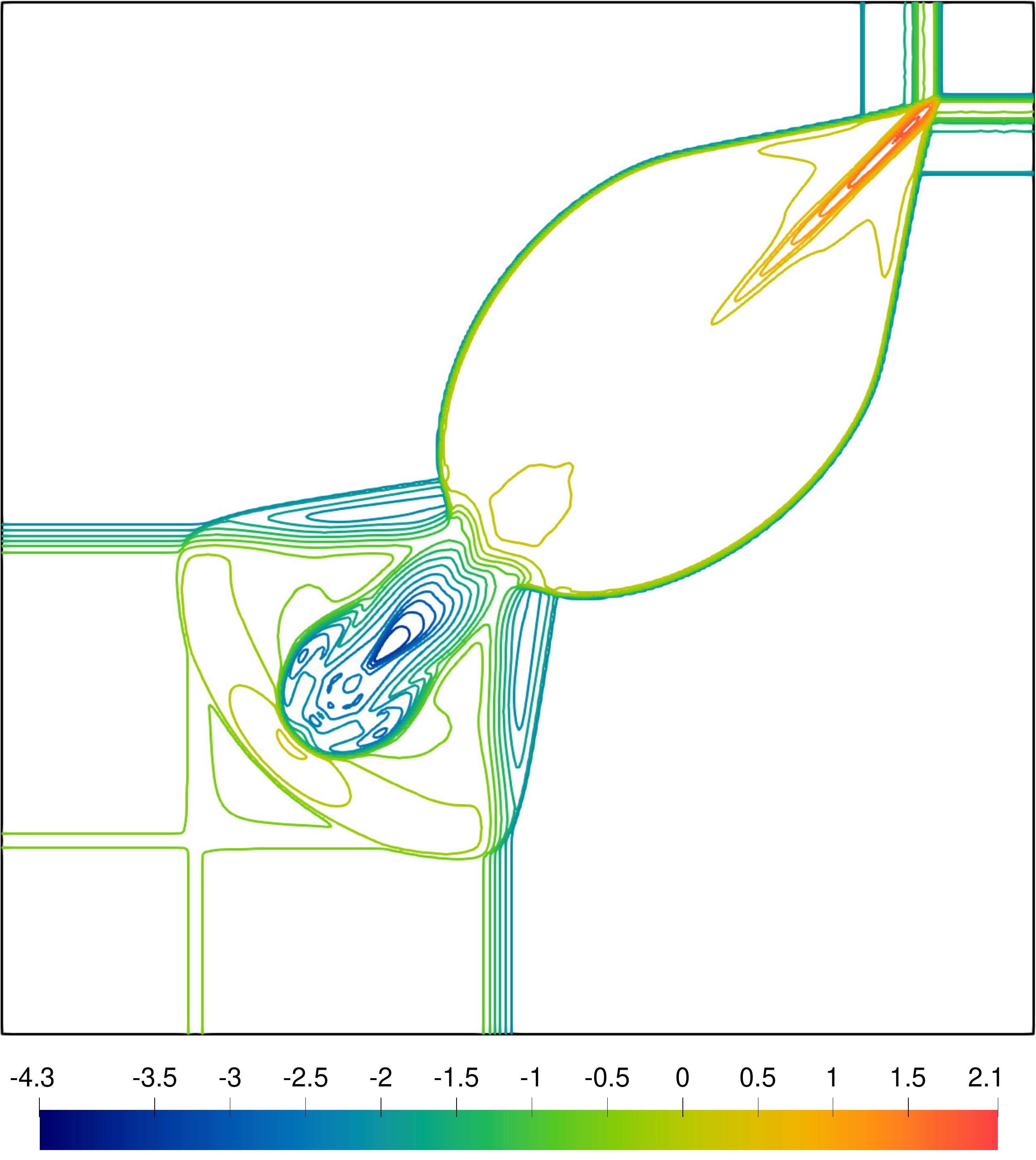}
         \caption{$\ln \rho$ with $25$ contour lines in $[-4.30, 2.09]$ using \rev{the scheme with degree $N=3$}.}
    \end{subfigure}
    \begin{subfigure}{0.45\textwidth}
         \includegraphics[width=\linewidth]{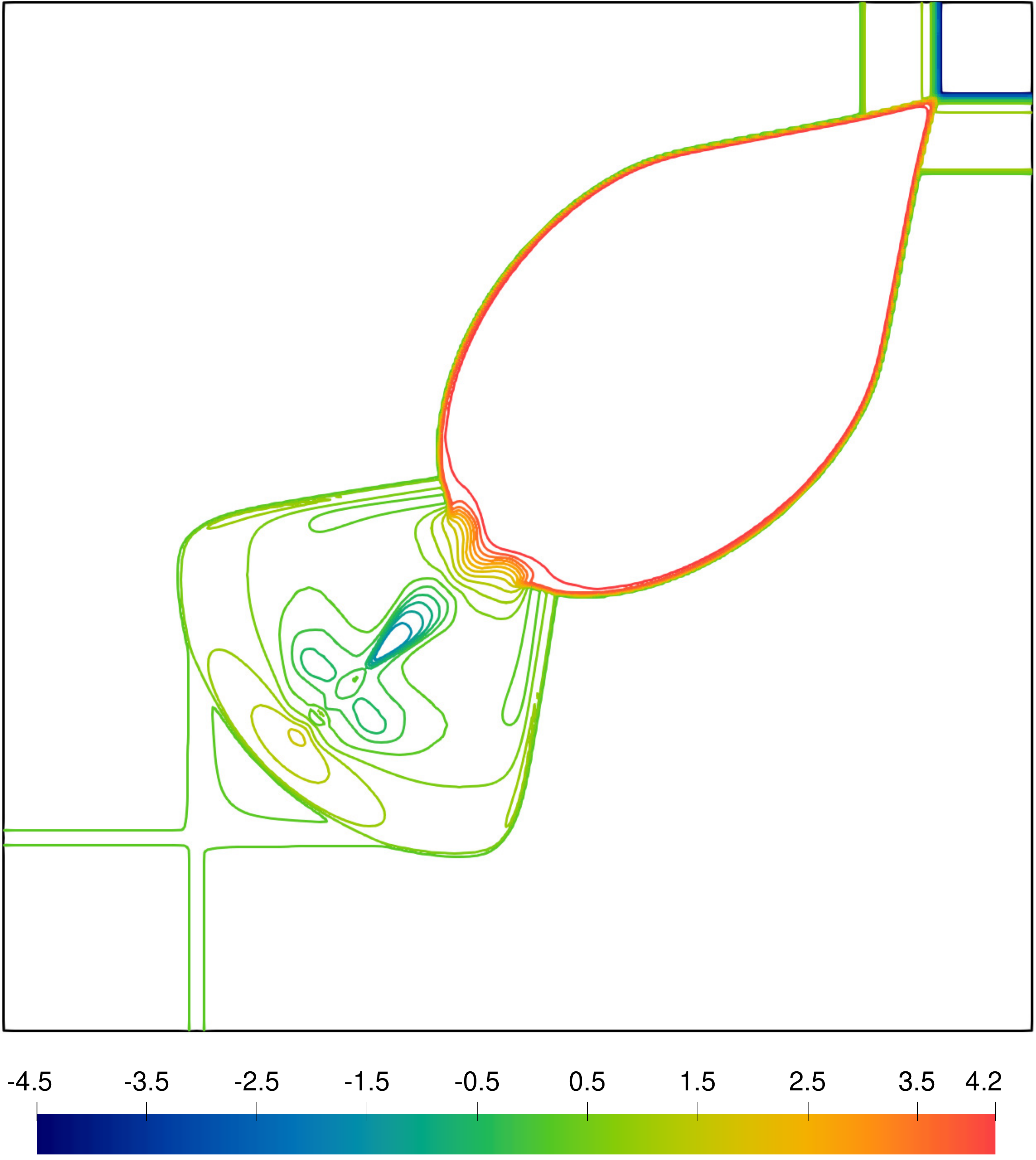}
         \caption{$\ln p$ with $25$ contour lines in $[-4.50, 3.10]$ using \rev{the scheme with degree $N=3$}.}
    \end{subfigure}
    \begin{subfigure}{0.45\textwidth}
         \includegraphics[width=\linewidth]{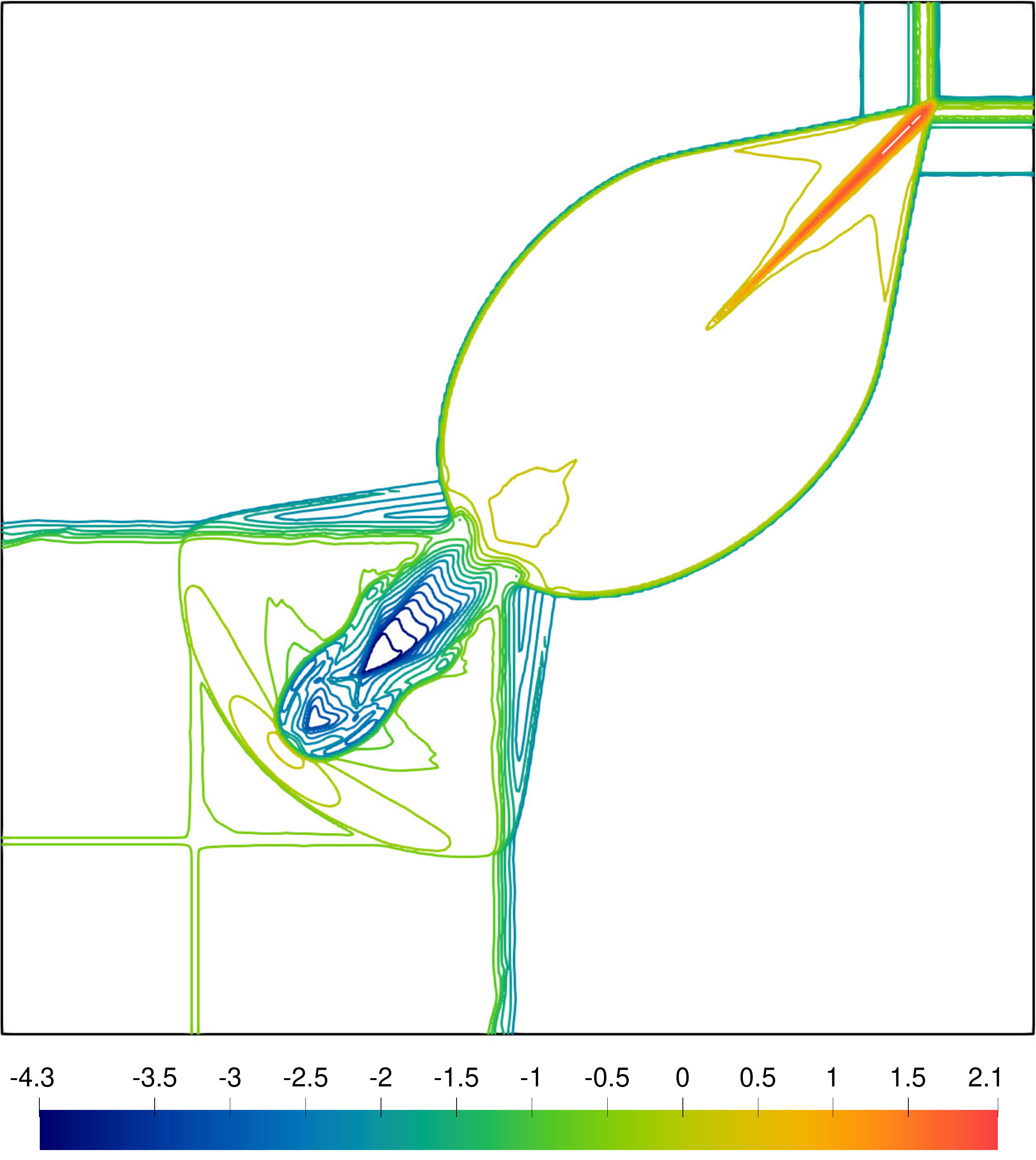}
         \caption{$\ln \rho$ with $25$ contour lines in $[-4.30, 2.09]$ using \rev{the scheme with degree $N=4$}.}
    \end{subfigure}
    \begin{subfigure}{0.45\textwidth}
         \includegraphics[width=\linewidth]{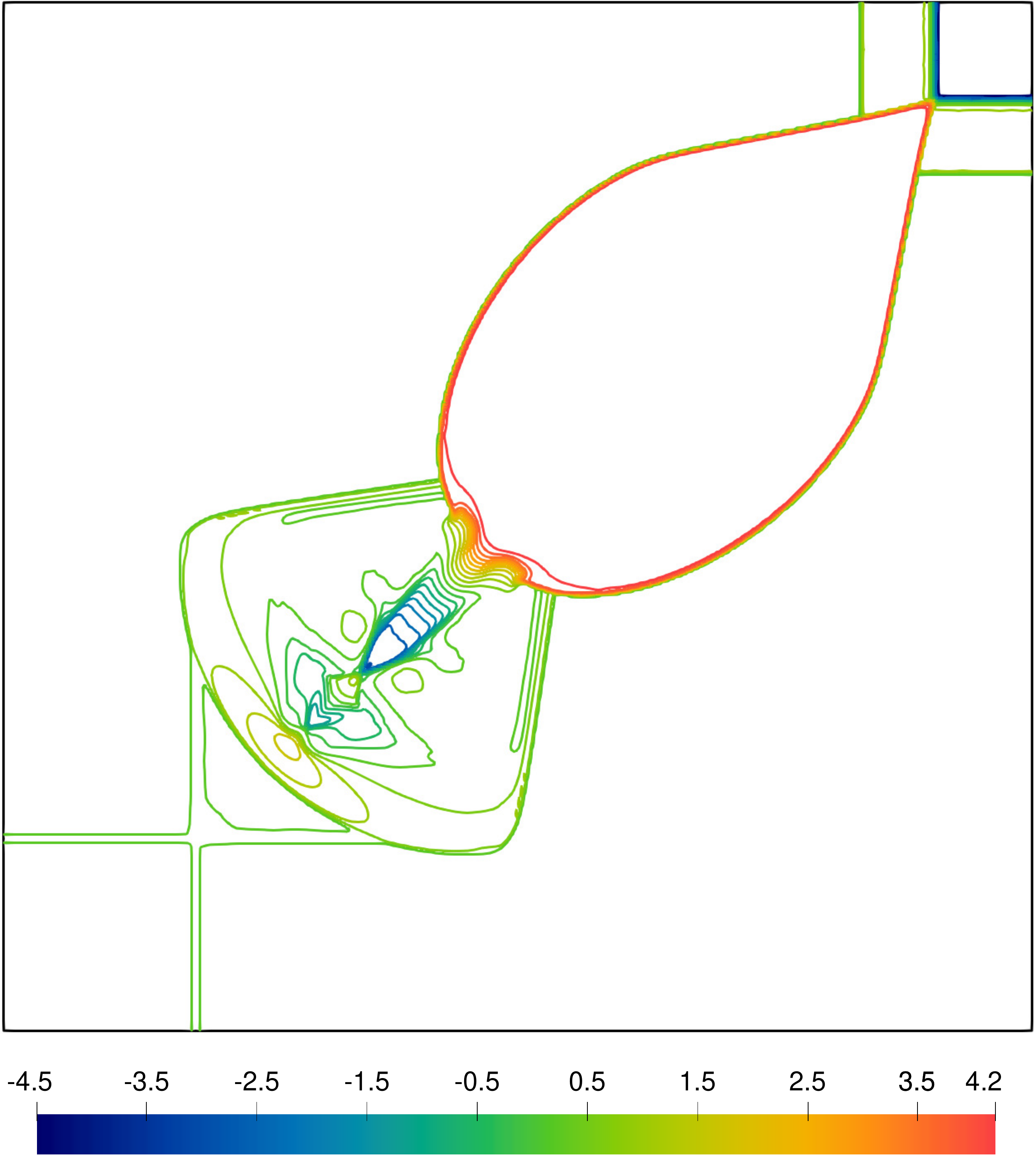}
         \caption{$\ln p$ with $25$ contour lines in $[-4.50, 3.10]$ using \rev{the scheme with degree $N=4$}.}
    \end{subfigure}
    \caption{First Riemann problem in 2D: Plot of log of density and log of pressure using $400\times 400$ cells.}
    \label{2DRP1}
\end{figure}
\subsubsection{First Riemann problem}
This problem is taken from~\cite{del2002efficient}, which has two one-dimensional shocks that are symmetric to the main diagonal and two contact discontinuities in its solution, making it suitable to check the efficiency of our scheme. It is also a good test to check the robustness of the scheme because of the presence of very high-velocity components in the second and fourth quadrants. Initially, we fill the domain by dividing it into four quadrants as follows,
\[
(\rho, v_1, v_2, p) = \begin{cases}
        (0.1, 0, 0, 0.01) & \text{if}\ x > 0.5,\ y > 0.5\\
        (0.1, 0.99, 0, 1) & \text{if}\ x < 0.5,\ y>0.5\\
        (0.5, 0, 0, 1) & \text{if}\ x < 0.5,\ y < 0.5\\
        (0.1, 0, 0.99, 1) & \text{if}\ x > 0.5,\ y<0.5.\\
    \end{cases}
\]
We present our result in Figure~\ref{2DRP1} running the simulation using $400\times 400$ cells till time $t=0.4$ \rev{using the scheme with degrees $N=3,4$}. We use outflow boundary conditions at all the boundaries. We observe that the \rev{the scheme with degree $N=4$} can capture all the waves with slightly higher accuracy in terms of sharp resolution compared to the \rev{scheme with degree $N=3$}.

This is also a good test to check our discontinuity indicator model in a two-dimensional case. From Figure~\ref{2DRP1} and Figure~\ref{2DRP1_gassner}, which uses the old indicator model, we can see that our discontinuity indicator model produces less oscillatory results compared to the indicator used in of~\cite{hennemann2021provably, babbar2024admissibility}.
\begin{figure}[htbp]
    \centering
    \begin{subfigure}{0.45\textwidth}
         \includegraphics[width=\linewidth]{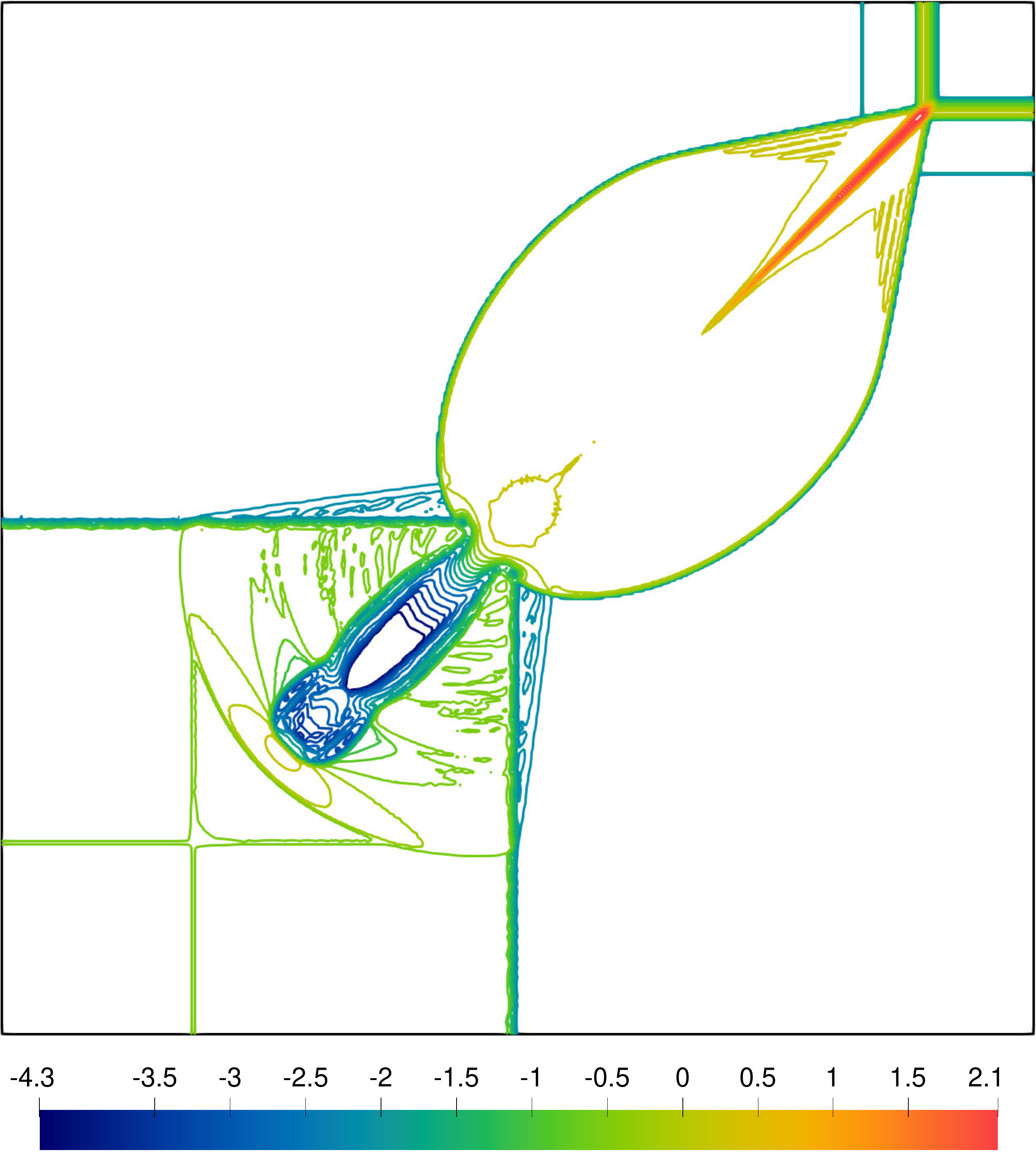}
         \caption{$\ln \rho$ with $25$ contour lines in $[-4.30, 2.09]$ using \rev{the scheme with degree $N=4$}.}
    \end{subfigure}
    \begin{subfigure}{0.45\textwidth}
         \includegraphics[width=\linewidth]{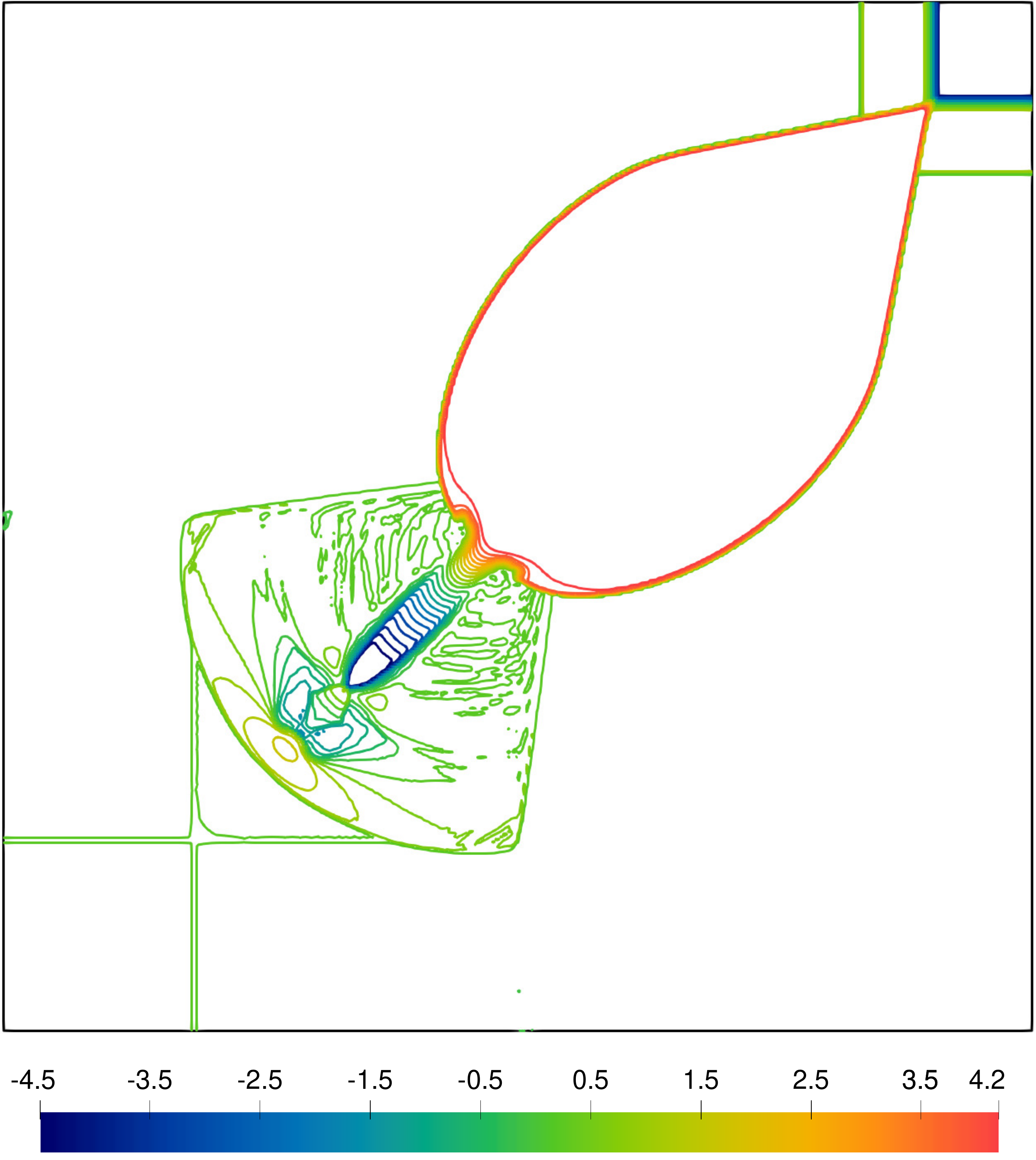}
         \caption{$\ln p$ with $25$ contour lines in $[-4.50, 3.10]$ using \rev{the scheme with degree $N=4$}.}
    \end{subfigure}
    \caption{First Riemann problem in 2D: Plot of log of density and log of pressure with $400\times 400$ cells using indicator model from~\cite{hennemann2021provably}.}
    \label{2DRP1_gassner}
\end{figure}

\begin{figure}
    \centering
    \begin{subfigure}{0.45\textwidth}
         \includegraphics[width=\linewidth]{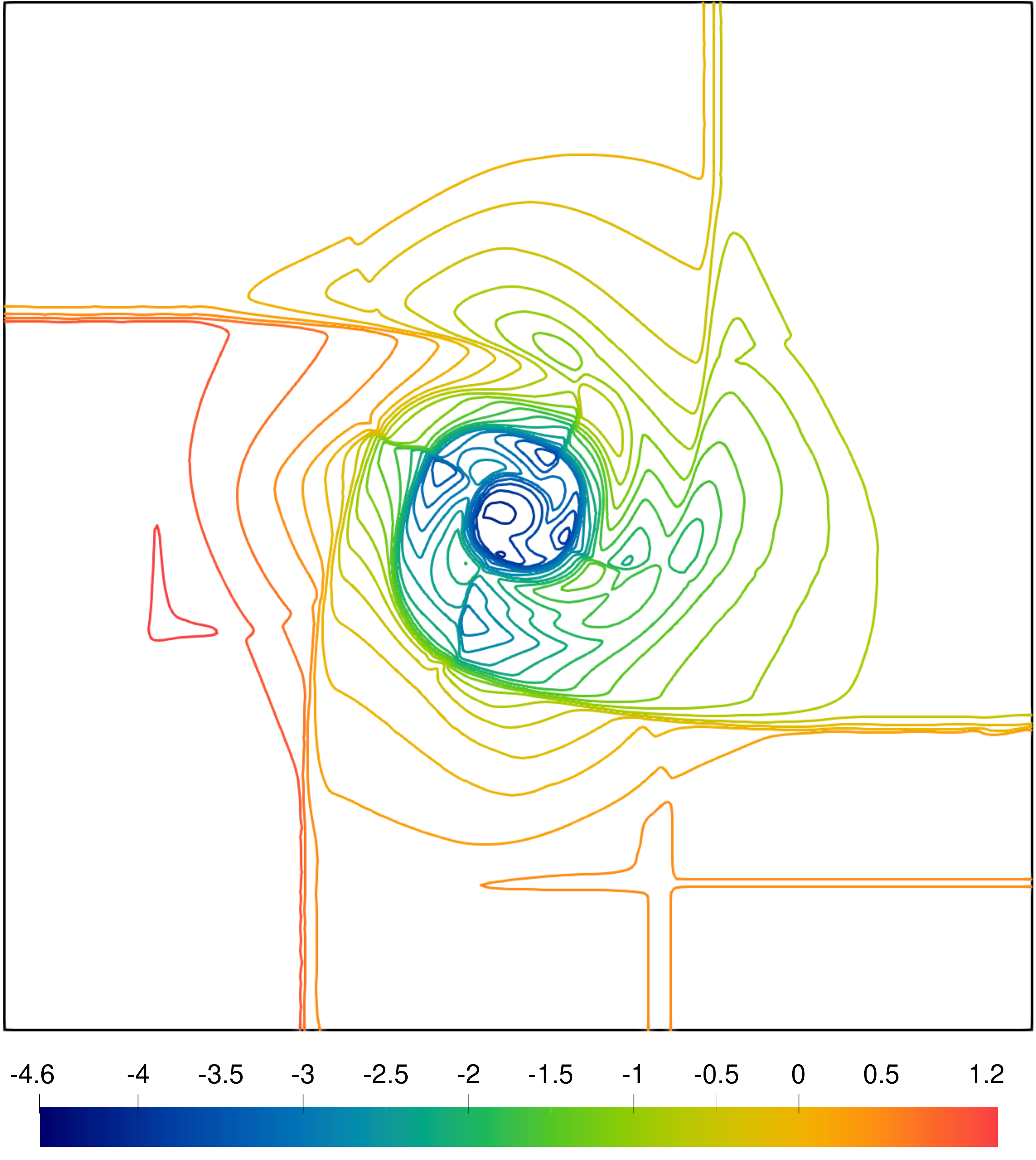}
         \caption{$\ln \rho$ with $25$ contour lines in $[-4.60, 1.20]$ using \rev{the scheme with degree $N=3$}.}
    \end{subfigure}
    \begin{subfigure}{0.45\textwidth}
         \includegraphics[width=\linewidth]{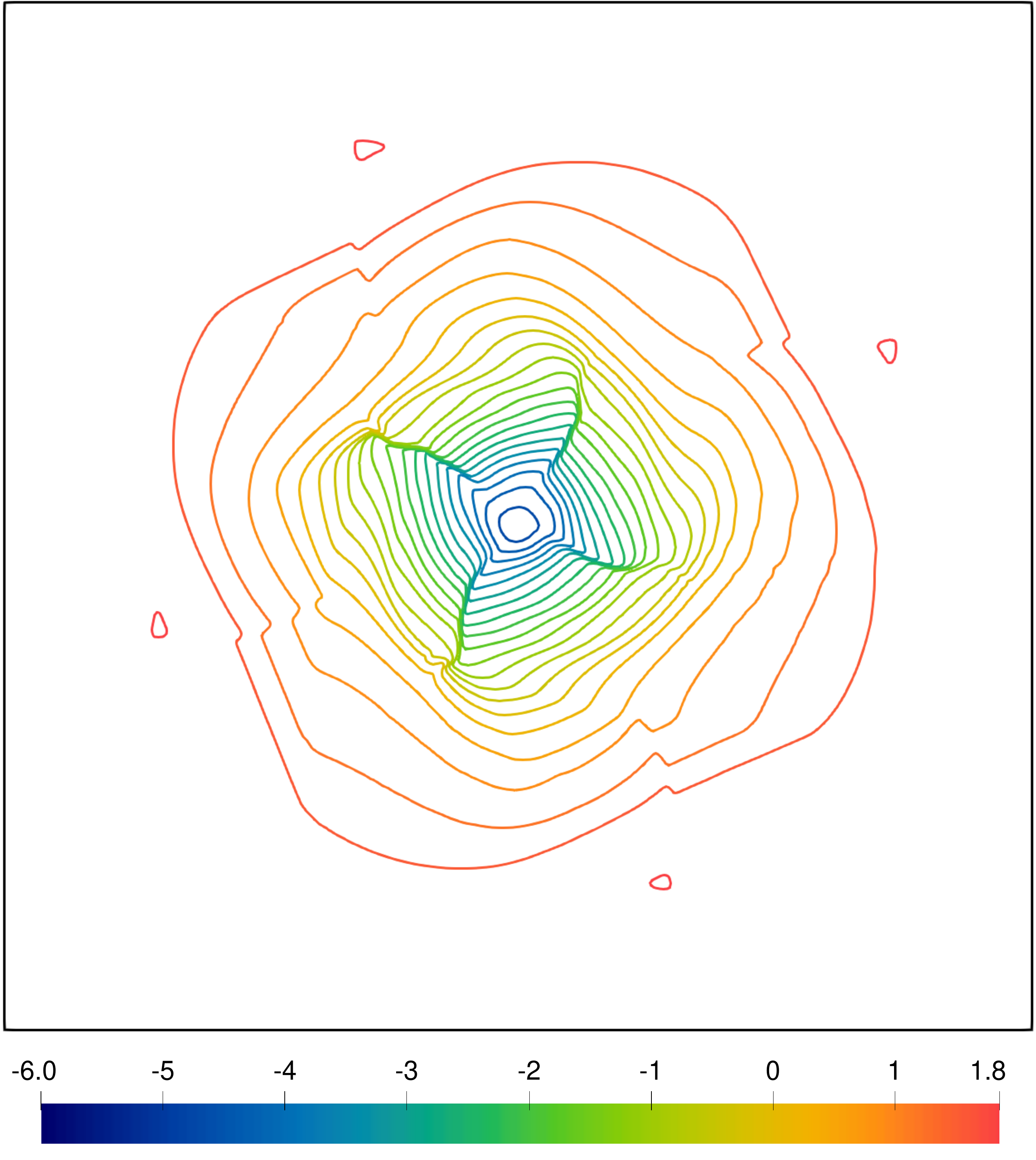}
         \caption{$\ln p$ with $25$ contour lines in $[-6.00, 1.85]$ using \rev{the scheme with degree $N=3$}.}
    \end{subfigure}
    \begin{subfigure}{0.45\textwidth}
         \includegraphics[width=\linewidth]{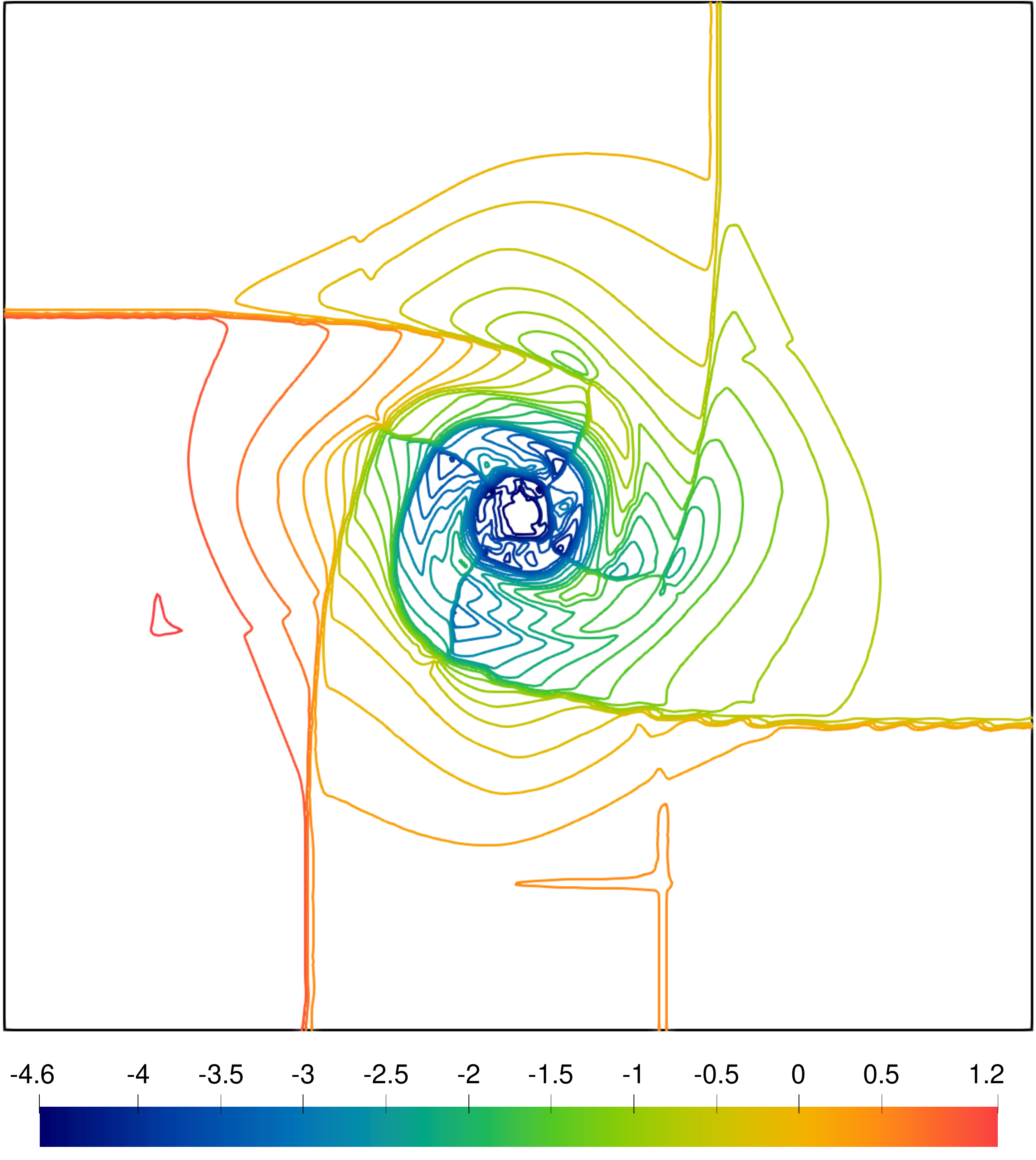}
         \caption{$\ln \rho$ with $25$ contour lines in $[-4.60, 1.20]$ using\rev{the scheme with degree $N=4$}.}
    \end{subfigure}
    \begin{subfigure}{0.45\textwidth}
         \includegraphics[width=\linewidth]{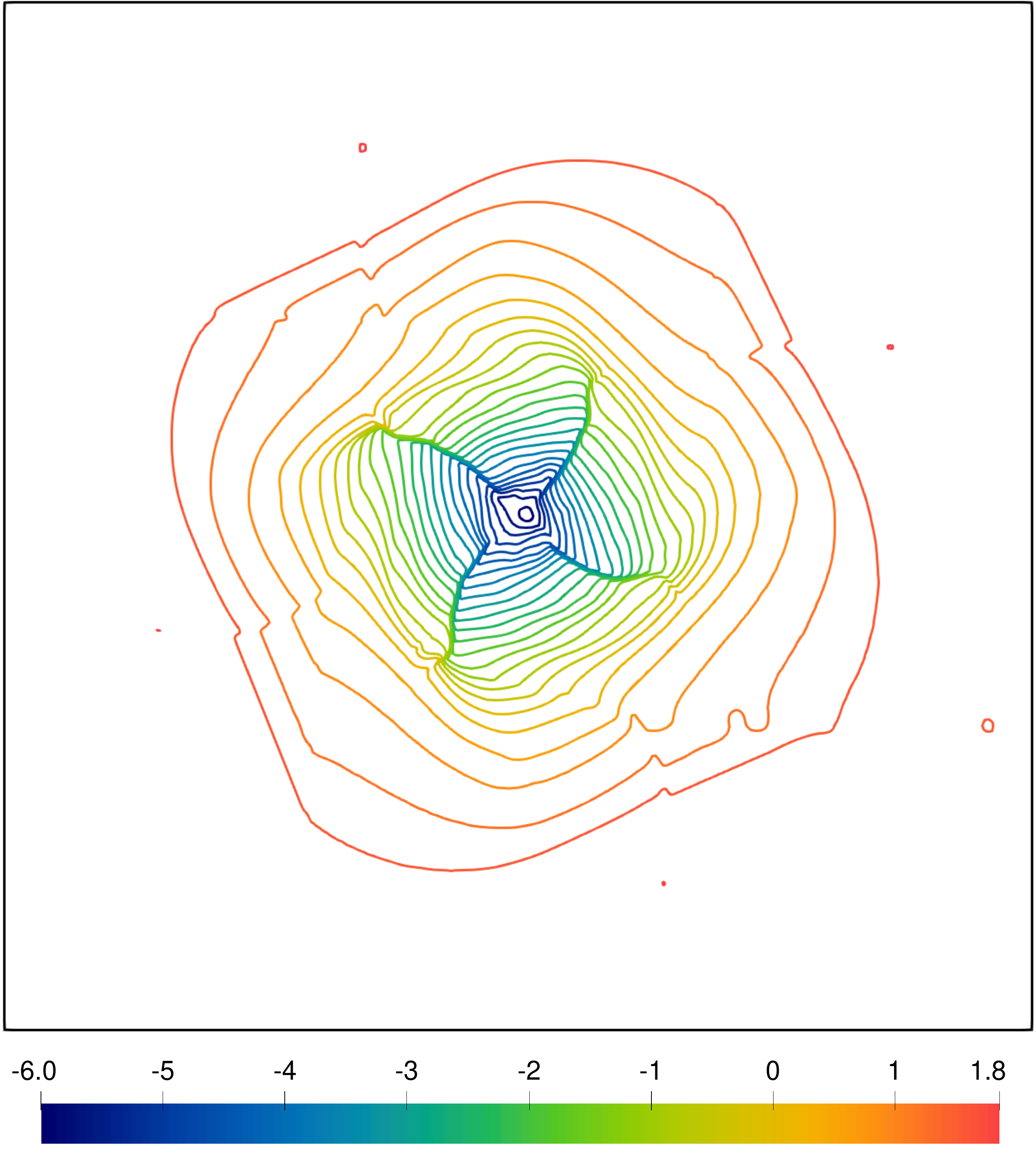}
         \caption{$\ln p$ with $25$ contour lines in $[-6.00, 1.85]$ using \rev{the scheme with degree $N=4$}.}
    \end{subfigure}
    \caption{Second Riemann problem in 2D: Plot of log of density and log of pressure using $400\times 400$ cells.}
    \label{2DRP2}
\end{figure}

\subsubsection{Second Riemann problem}
This problem is taken from~\cite{nunez2016xtroem}, where initially the four quadrants of the domain $[0,1]\times [0,1]$ are filled with a fluid having the following properties,
\[
(\rho, v_1, v_2, p) = \begin{cases}
        (0.5, 0.5, -0.5, 5) & \text{if}\ x > 0.5,\ y > 0.5\\
        (1, 0.5, 0.5, 5) & \text{if}\ x < 0.5,\ y>0.5\\
        (3, -0.5, 0.5, 5) & \text{if}\ x < 0.5,\ y < 0.5\\
        (1.5, -0.5, -0.5, 5) & \text{if}\ x > 0.5,\ y<0.5.\\
    \end{cases}
\]
We run the simulation till $t=0.4$ with $400\times 400$ cells with outflow boundary conditions at all the boundaries and show the result in Figure~\ref{2DRP2}. We observe that the schemes can capture the waves effectively; the \rev{scheme with degree $N=4$} captures the details more accurately than the \rev{scheme with degree $N=3$}, as expected.

\begin{figure}
    \centering
    \begin{subfigure}{0.45\textwidth}
         \includegraphics[width=\linewidth]{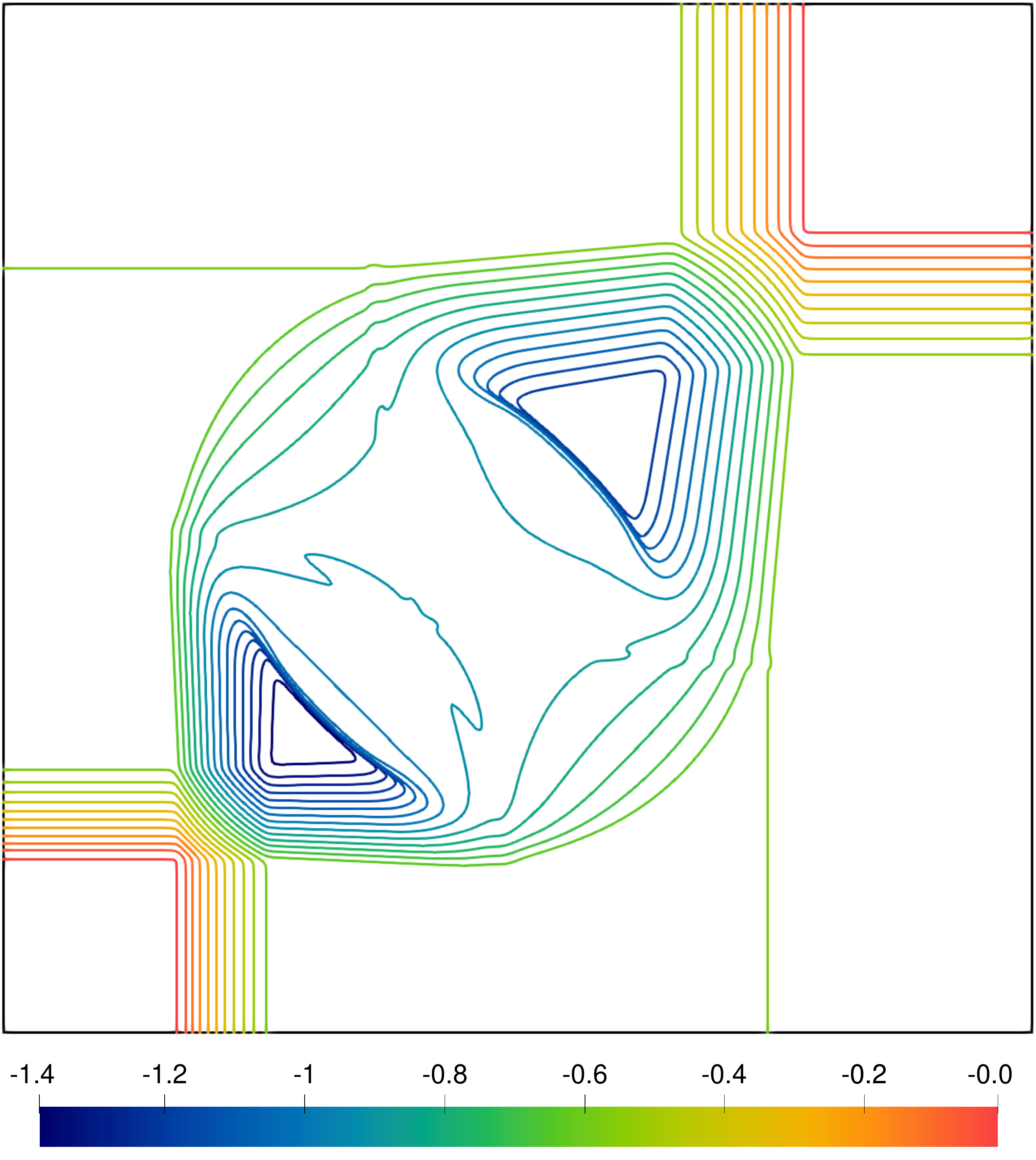}
         \caption{$\ln \rho$ with $25$ contour lines in $[-1.38, -0.01]$ using \rev{the scheme with degree $N=3$}.}
    \end{subfigure}
    \begin{subfigure}{0.45\textwidth}
         \includegraphics[width=\linewidth]{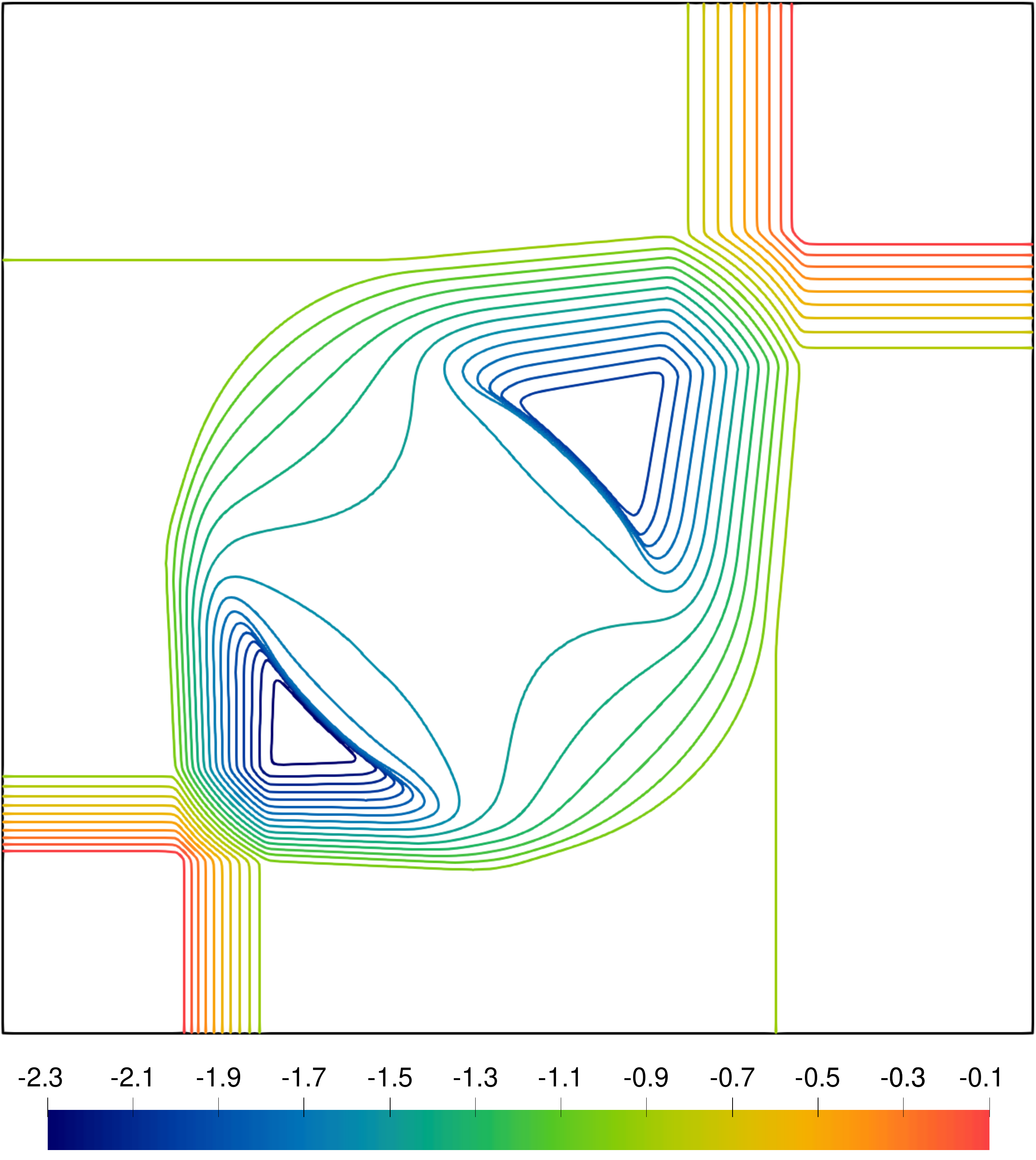}
         \caption{$\ln p$ with $25$ contour lines in $[-2.30, -0.01]$ using \rev{the scheme with degree $N=3$}.}
    \end{subfigure}
    \begin{subfigure}{0.45\textwidth}
         \includegraphics[width=\linewidth]{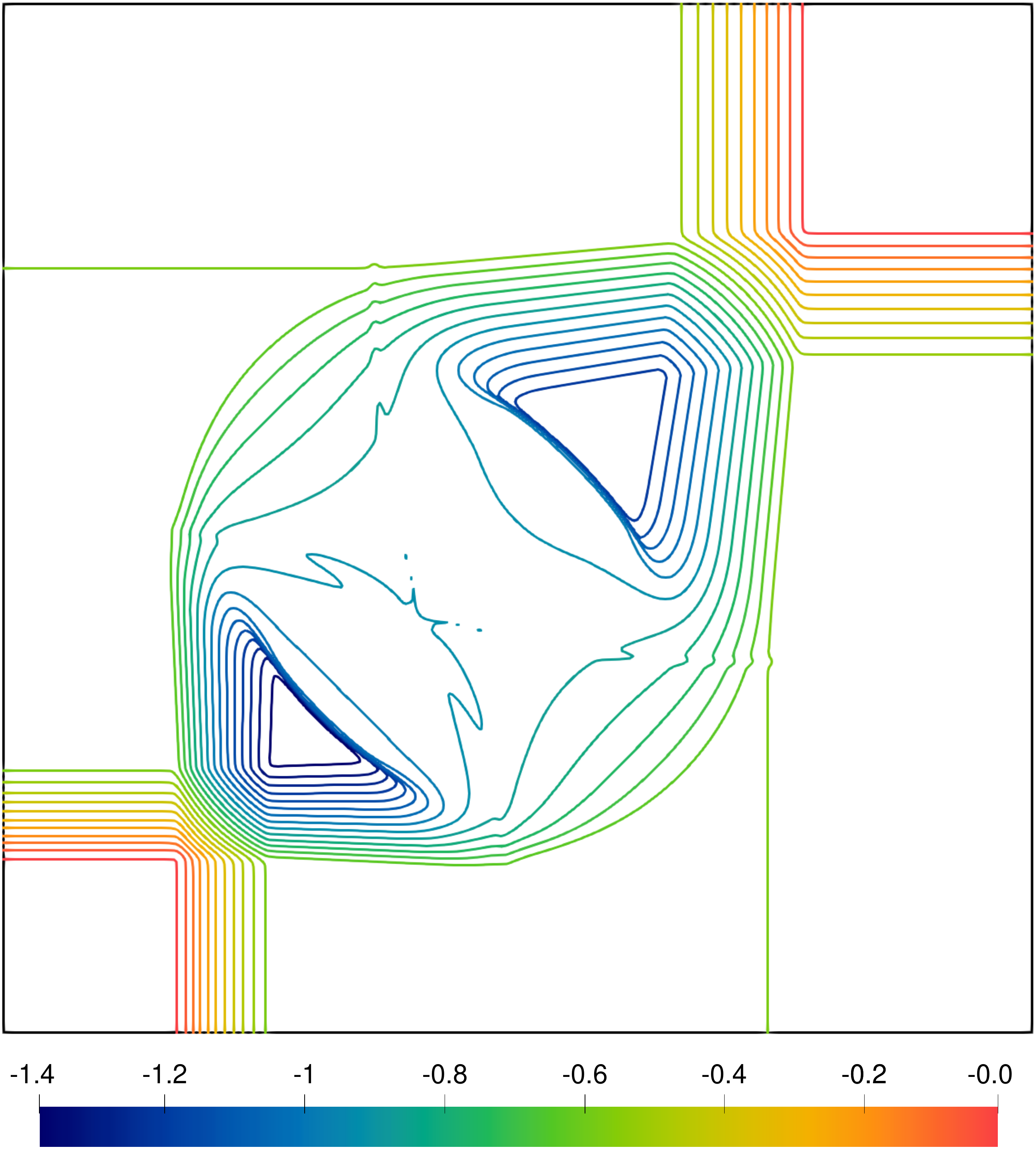}
         \caption{$\ln \rho$ with $25$ contour lines in $[-1.38, -0.01]$ using \rev{the scheme with degree $N=4$}.}
    \end{subfigure}
    \begin{subfigure}{0.45\textwidth}
         \includegraphics[width=\linewidth]{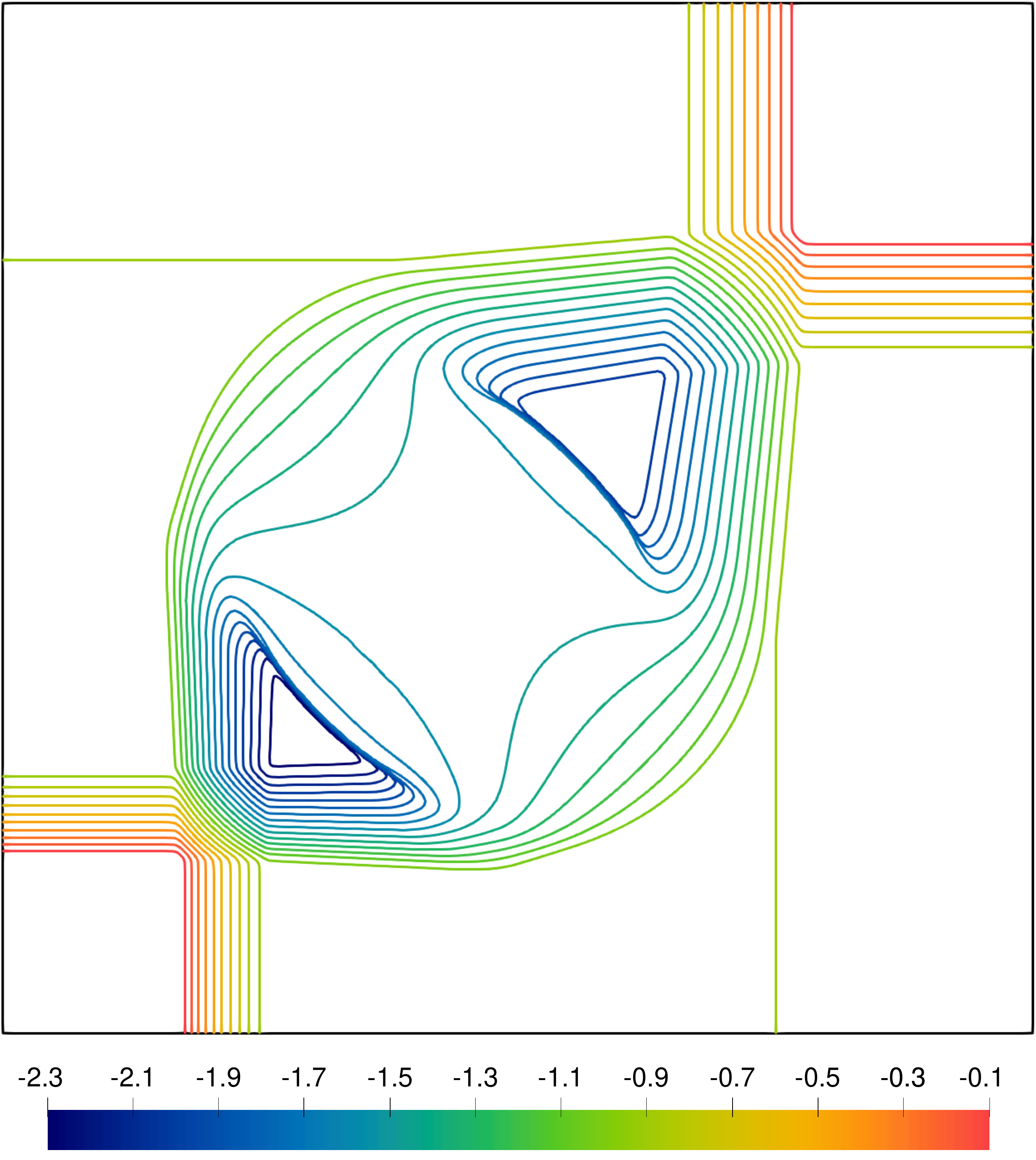}
         \caption{$\ln p$ with $25$ contour lines in $[-2.30, -0.01]$ using \rev{the scheme with degree $N=4$}.}
    \end{subfigure}
    \caption{Third Riemann problem in 2D: Plot of log of density and log of pressure using $400\times 400$ cells.}
    \label{2DRP3}
\end{figure}

\subsubsection{Third Riemann problem}
This Riemann problem is taken from~\cite{he2012adaptive}, and the initial condition is given by,
\[
(\rho, v_1, v_2, p) = \begin{cases}
        (1, 0, 0, 1) & \text{if}\ x > 0.5,\ y > 0.5\\
        (0.5771, -0.3529, 0, 0.4) & \text{if}\ x < 0.5,\ y>0.5\\
        (1, -0.3529, -0.3529, 1) & \text{if}\ x < 0.5,\ y < 0.5\\
        (0.5771, 0, -0.3529, 0.4) & \text{if}\ x > 0.5,\ y<0.5.\\
    \end{cases}
\]
The boundaries are taken as the outflow boundaries, and we run the simulation till $t=0.4$ with $400\times 400$ cells using \rev{the scheme with degrees $N=3,4$} and present the result in Figure~\ref{2DRP3}. We observe that the scheme can capture both the shock waves (symmetric to one another) properly. We also notice that for this test, there is not much difference in the performance of the \rev{scheme with $N=3$ or $N=4$}, which is a similar characteristic as noticed in~\cite{bhoriya2020entropy}.

\begin{figure}
    \centering
    \begin{subfigure}{0.45\textwidth}
         \includegraphics[width=\linewidth]{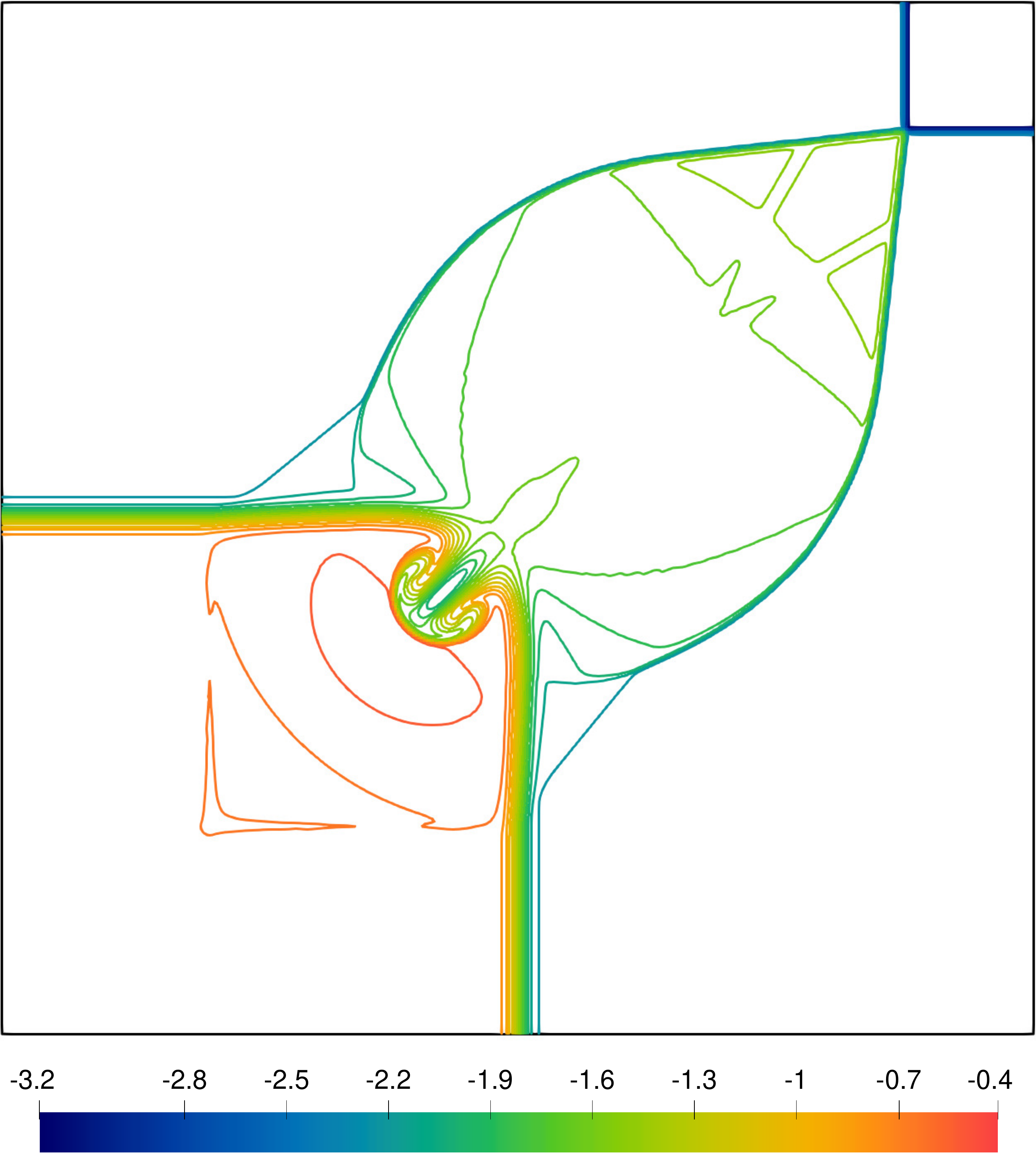}
         \caption{$\ln \rho$ with $25$ contour lines in $[-3.35, -0.41]$  using \rev{the scheme with degree $N=3$}.}
    \end{subfigure}
    \begin{subfigure}{0.45\textwidth}
         \includegraphics[width=\linewidth]{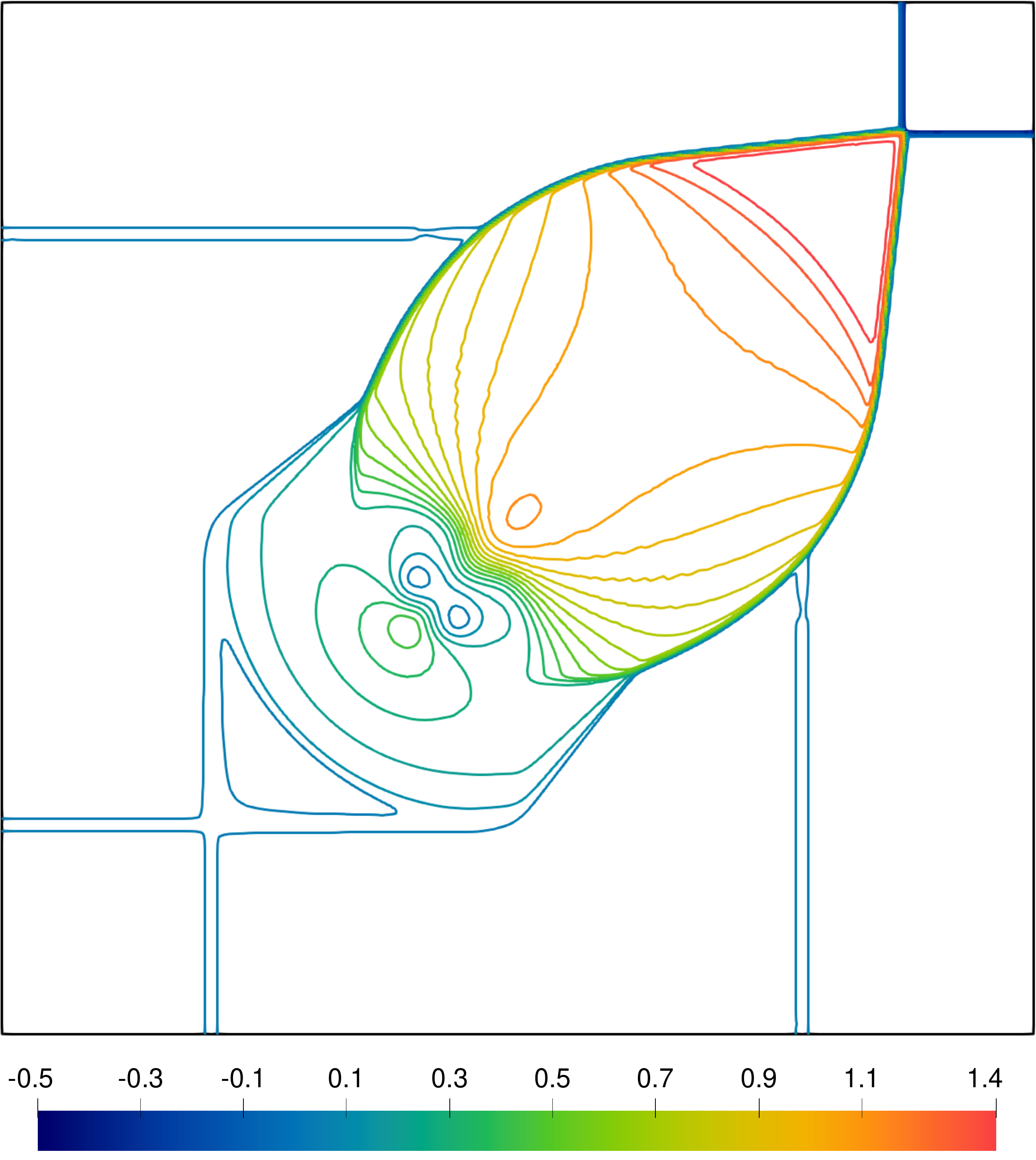}
         \caption{$\ln p$ with $25$ contour lines in $[-0.50, -1.36]$ using \rev{the scheme with degree $N=3$}.}
    \end{subfigure}
    \begin{subfigure}{0.45\textwidth}
         \includegraphics[width=\linewidth]{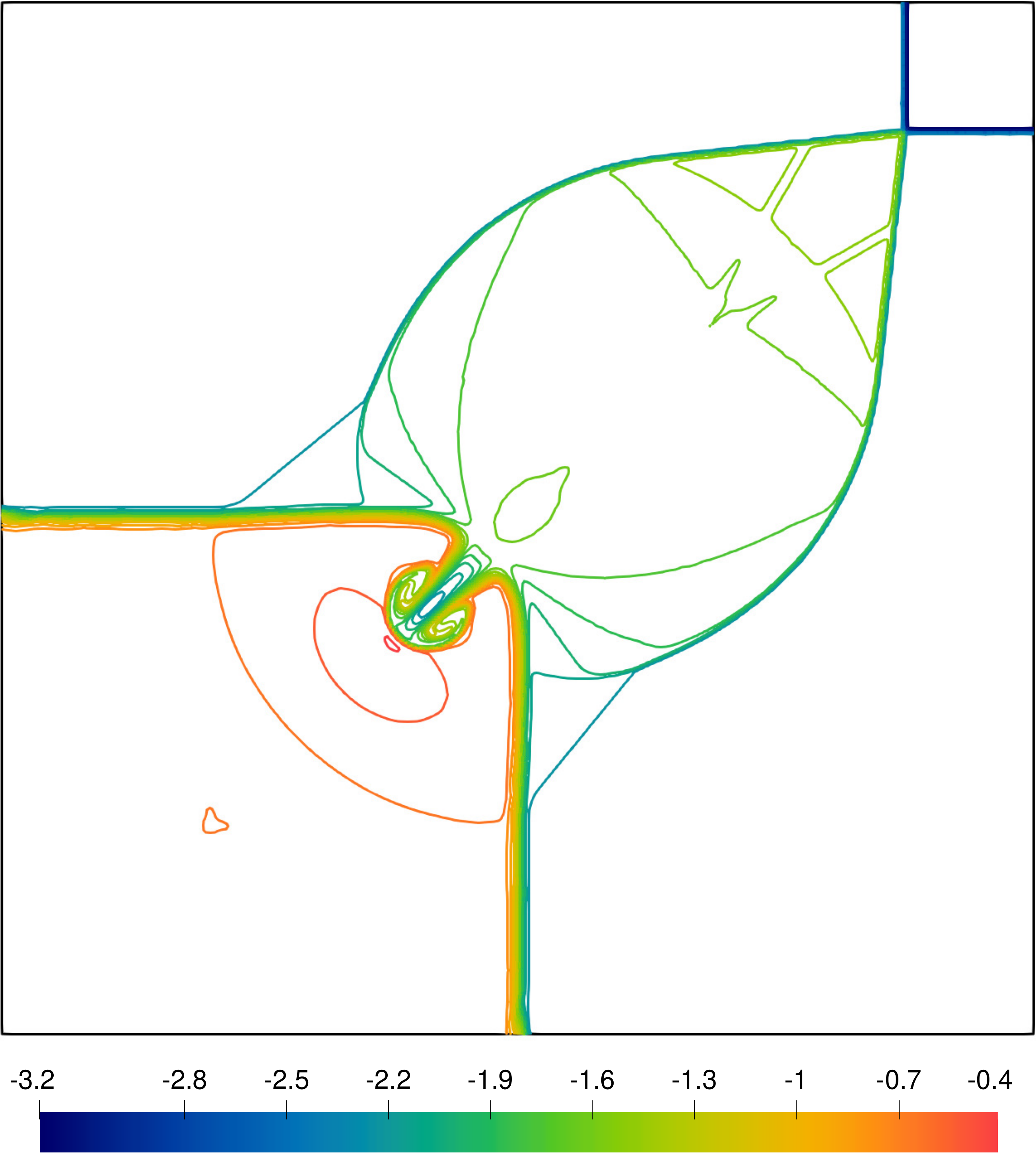}
         \caption{$\ln \rho$ with $25$ contour lines in $[-3.35, -0.41]$ using \rev{the scheme with degree $N=4$}.}
    \end{subfigure}
    \begin{subfigure}{0.45\textwidth}
         \includegraphics[width=\linewidth]{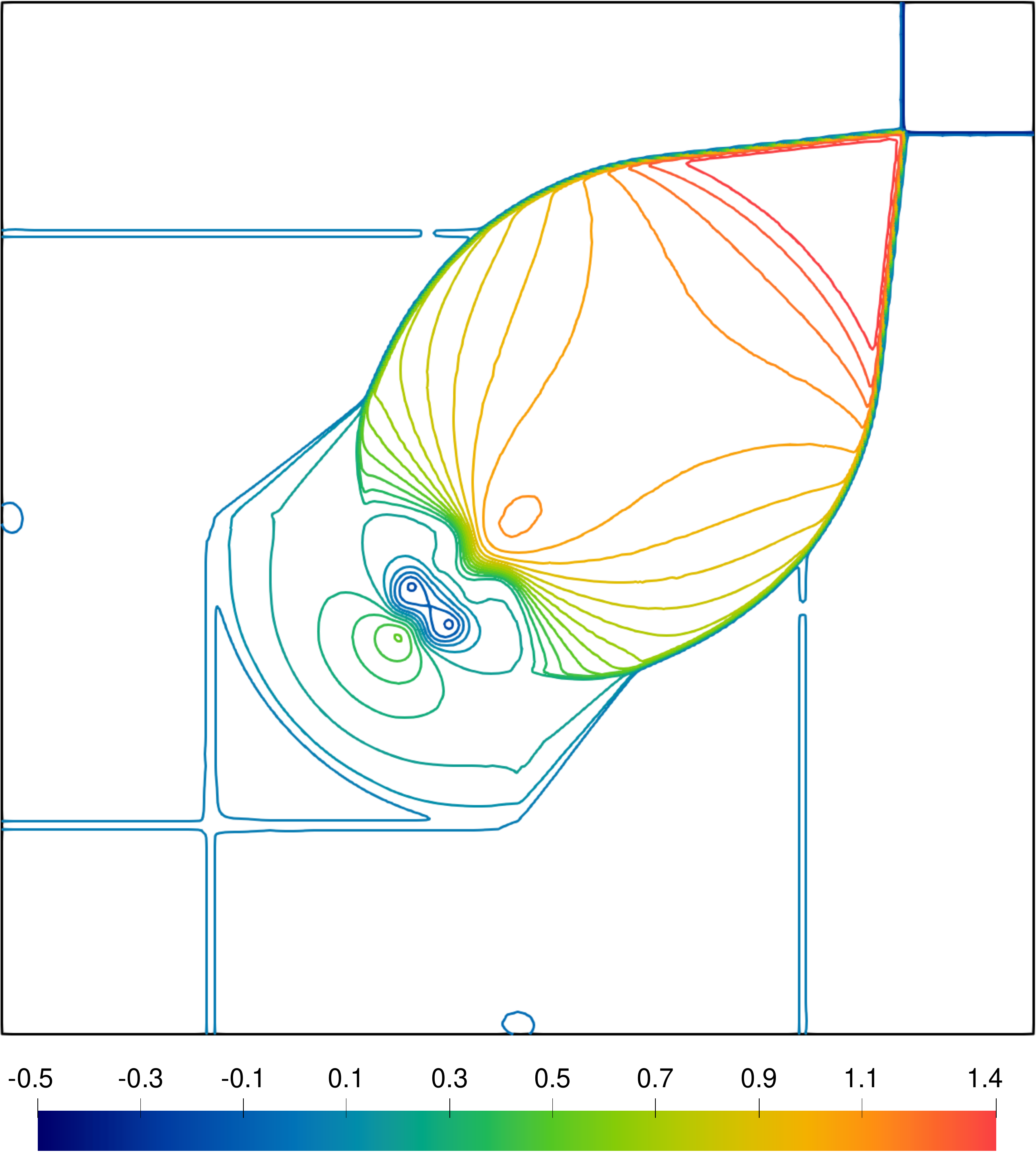}
         \caption{$\ln p$ with $25$ contour lines in $[-0.50, -1.36]$ using \rev{the scheme with degree $N=4$}.}
    \end{subfigure}
    \caption{Fourth Riemann problem in 2D: Plot of log of density and log of pressure using $400\times 400$ cells.}
    \label{2DRP4}
\end{figure}

\subsubsection{Fourth Riemann problem}
We consider one more test from~\cite{nunez2016xtroem}, whose solution has two contact discontinuities (near the left and the bottom boundaries). The initial state of the fluid is given by,
\begin{equation*}
    (\rho, v_1, v_2, p) = \begin{cases}
        (0.035145216124503, 0, 0, 0.162931056509027) & \text{if}\ x > 0.5,\ y > 0.5\\
        (0.1, 0.7, 0, 1) & \text{if}\ x < 0.5,\ y>0.5\\
        (0.5, 0, 0, 1) & \text{if}\ x < 0.5,\ y < 0.5\\
        (0.1, 0, 0.7, 1) & \text{if}\ x > 0.5,\ y<0.5.\\
    \end{cases}
\end{equation*}
We consider outflow boundary conditions at all the boundaries. By running the simulation till time $t=0.4$, we can see that a jet-like structure is formed in the lower-left quadrant as mentioned in~\cite{nunez2016xtroem}. We also observe that the scheme captures the structures in the solution effectively with both \rev{degrees $N=3,4$}. The simulation is run with $400\times 400$ cells, and the results are shown in Figure~\ref{2DRP4}.

We also present the result for this problem with the \rev{scheme with degree $N=4$} using the discontinuity indicator model of~\cite{hennemann2021provably} in Figure~\ref{2DRP4_gassner} and observe that it produces oscillatory results even in the central part, away from the sharp solution features, while the new discontinuity indicator model is able to eliminate these oscillations.

\begin{figure}[htbp]
    \centering
    \begin{subfigure}{0.45\textwidth}
         \includegraphics[width=\linewidth]{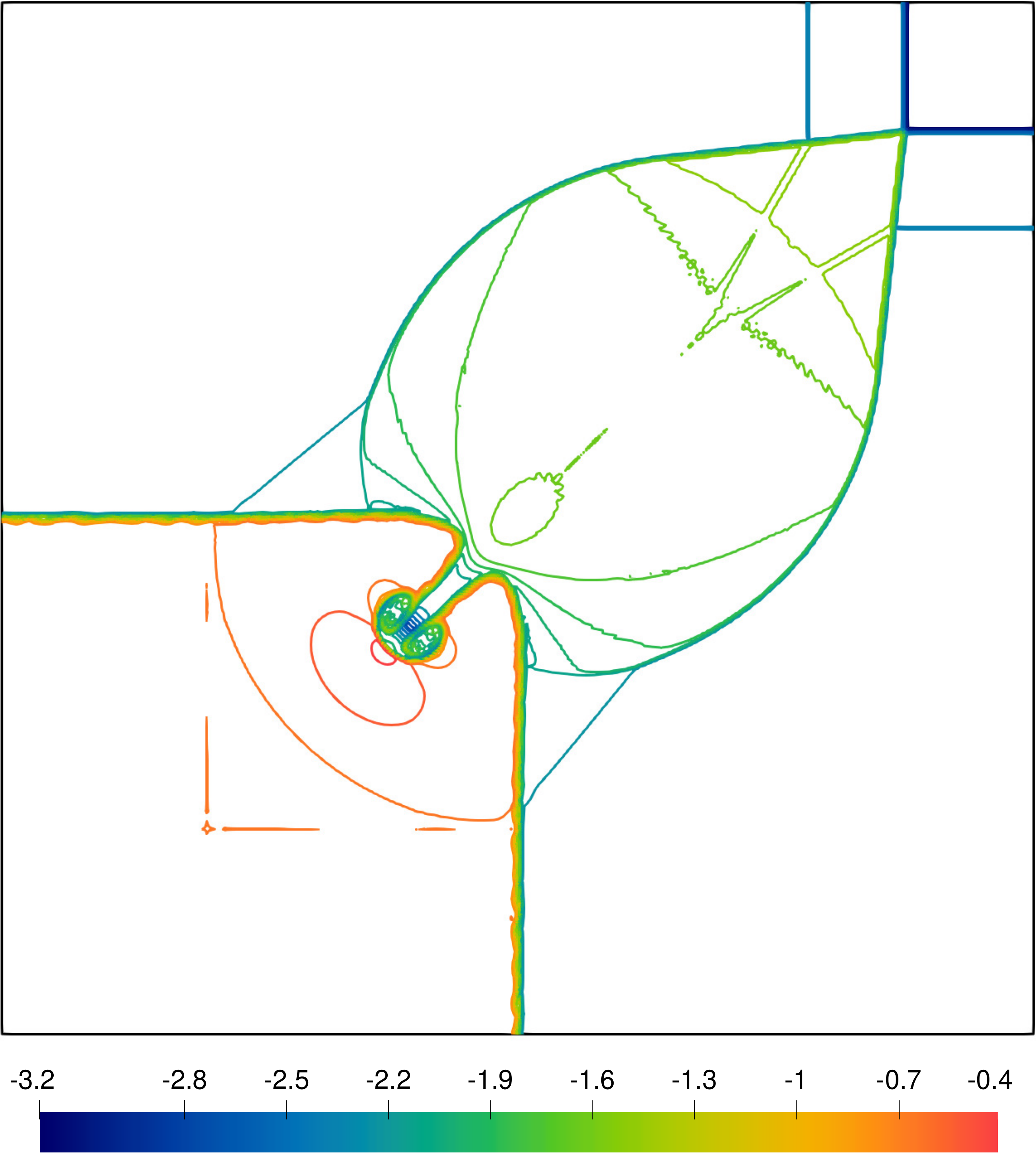}
         \caption{$\ln \rho$ with $25$ contour lines in $[-3.35, -0.41]$ using \rev{the scheme with degree $N=4$}.}
    \end{subfigure}
    \begin{subfigure}{0.45\textwidth}
         \includegraphics[width=\linewidth]{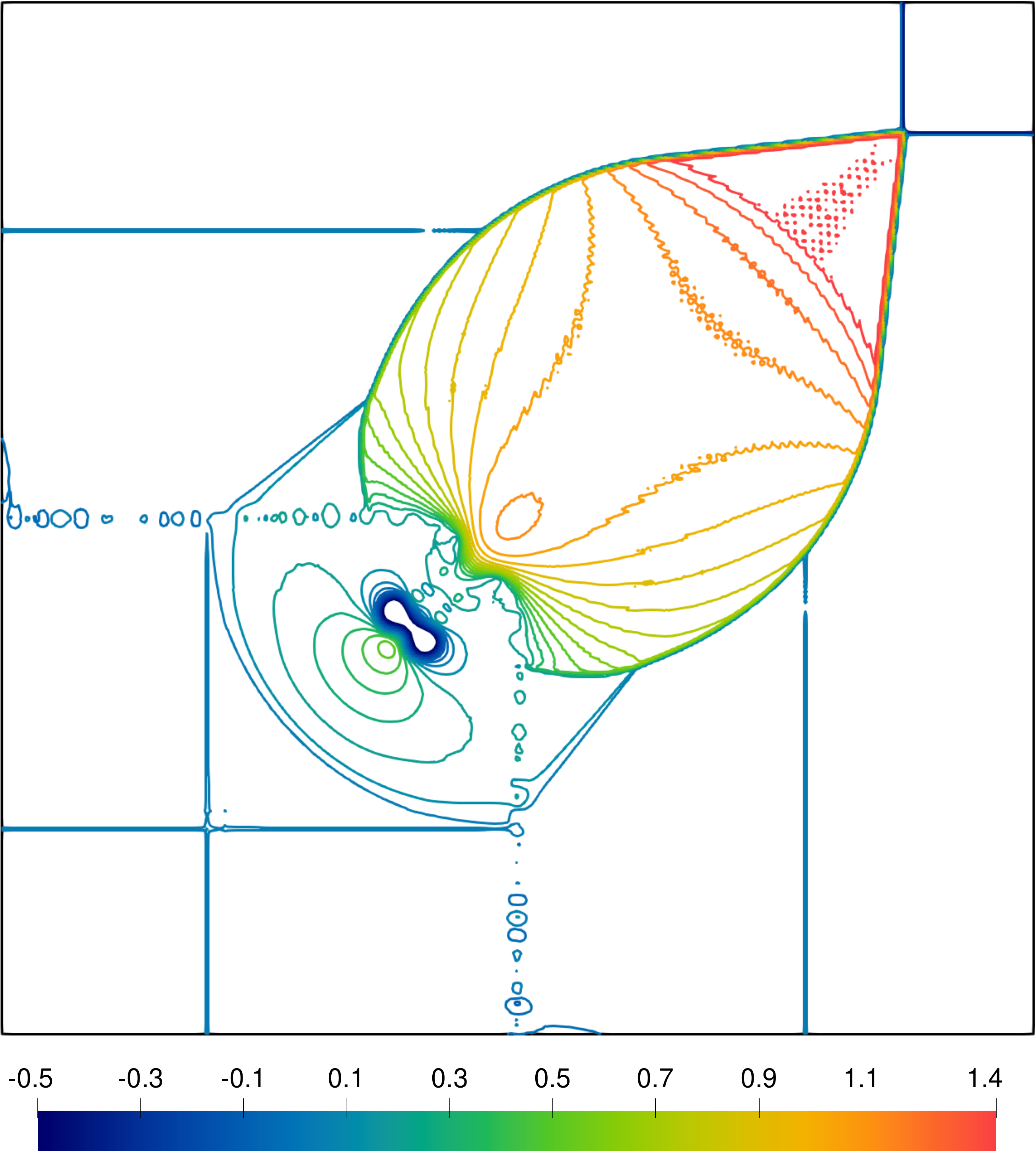}
         \caption{$\ln p$ with $25$ contour lines in $[-0.50, -1.36]$ using \rev{the scheme with degree $N=4$}.}
    \end{subfigure}
    \caption{Fourth Riemann problem in 2D: Plot of log of density and log of pressure using $400\times 400$ cells using indicator model from~\cite{hennemann2021provably}.}
    \label{2DRP4_gassner}
\end{figure}

\begin{figure}
    \centering
    \begin{subfigure}{0.45\textwidth}
         \includegraphics[width=\linewidth]{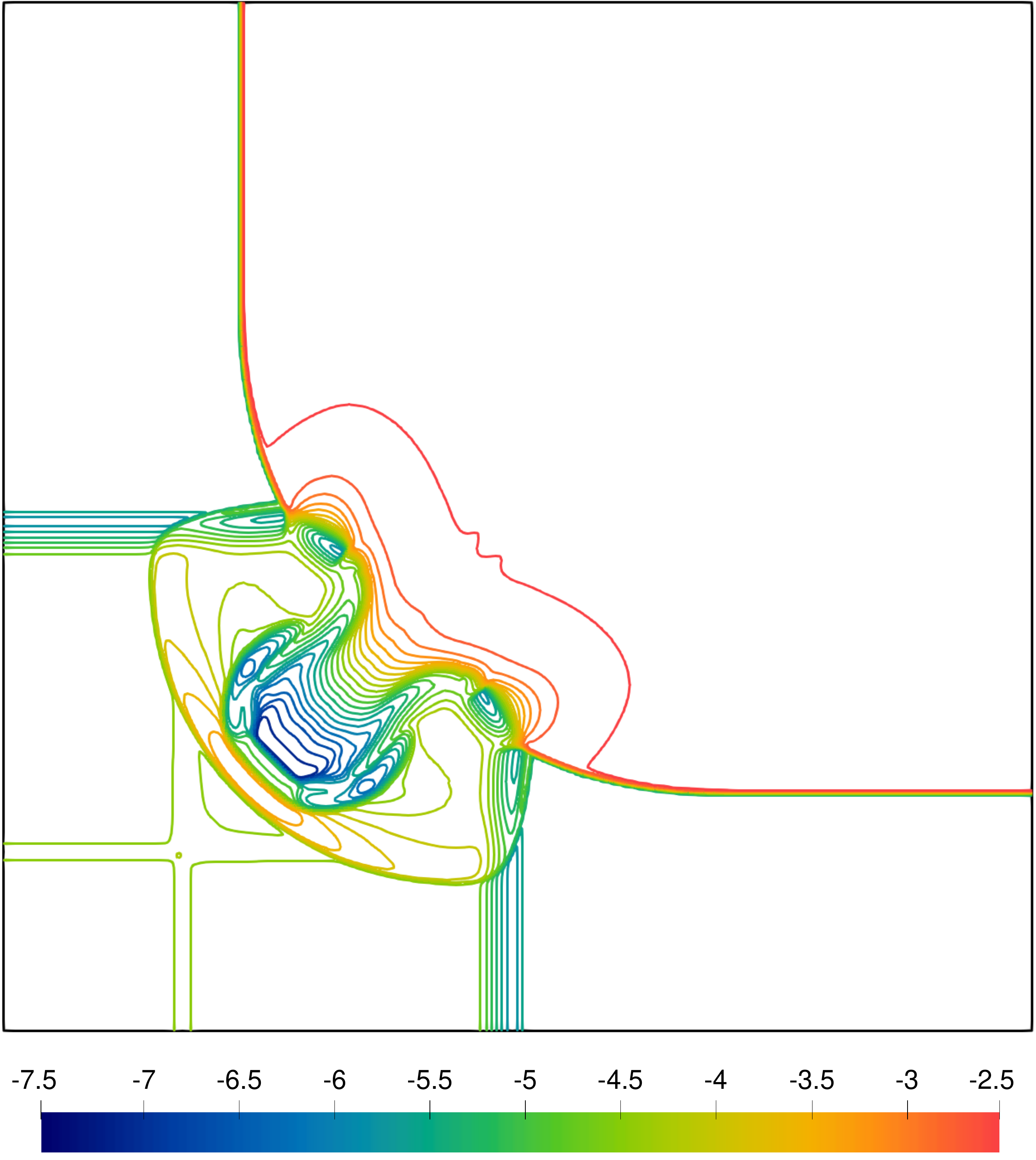}
         \caption{$\ln \rho$ with $25$ contour lines in $[-7.54, -2.30]$ using \rev{the scheme with degree $N=3$}.}
    \end{subfigure}
    \begin{subfigure}{0.45\textwidth}
         \includegraphics[width=\linewidth]{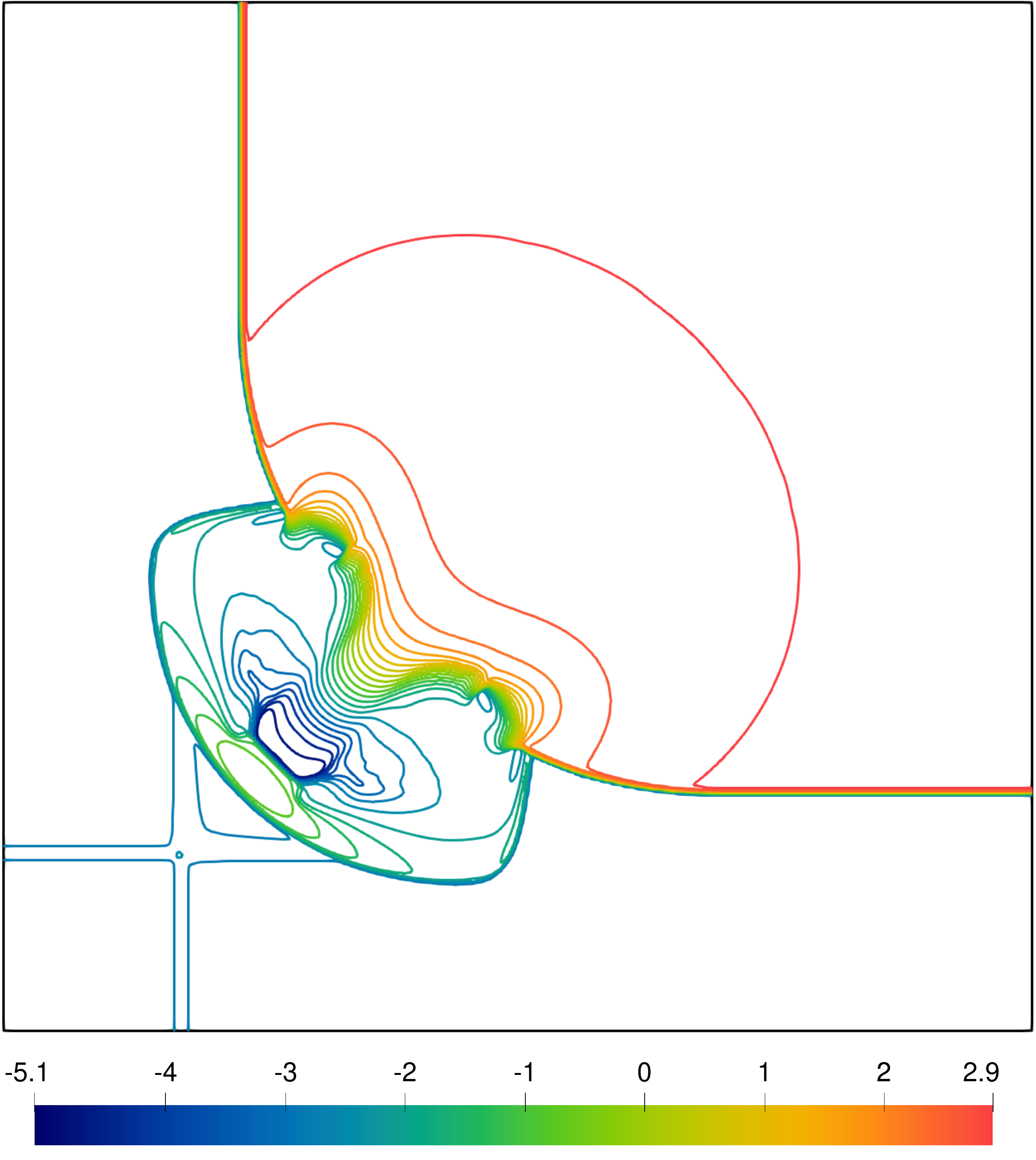}
         \caption{$\ln p$ with $25$ contour lines in $[-5.10, 2.90]$ using \rev{the scheme with degree $N=3$}.}
    \end{subfigure}
    \begin{subfigure}{0.45\textwidth}
         \includegraphics[width=\linewidth]{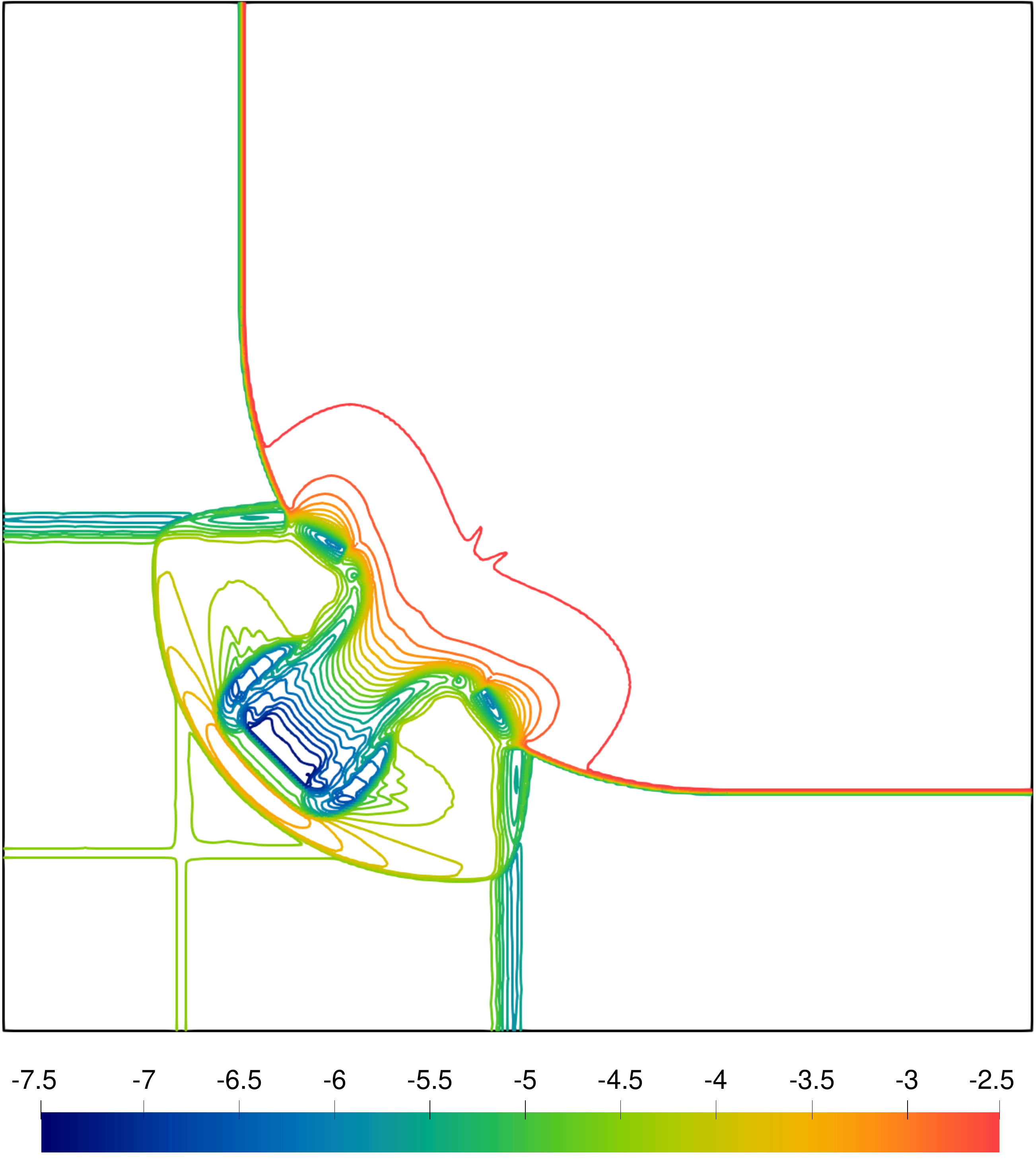}
         \caption{$\ln \rho$ with $25$ contour lines in $[-7.54, -2.30]$ using \rev{the scheme with degree $N=4$}.}
    \end{subfigure}
    \begin{subfigure}{0.45\textwidth}
         \includegraphics[width=\linewidth]{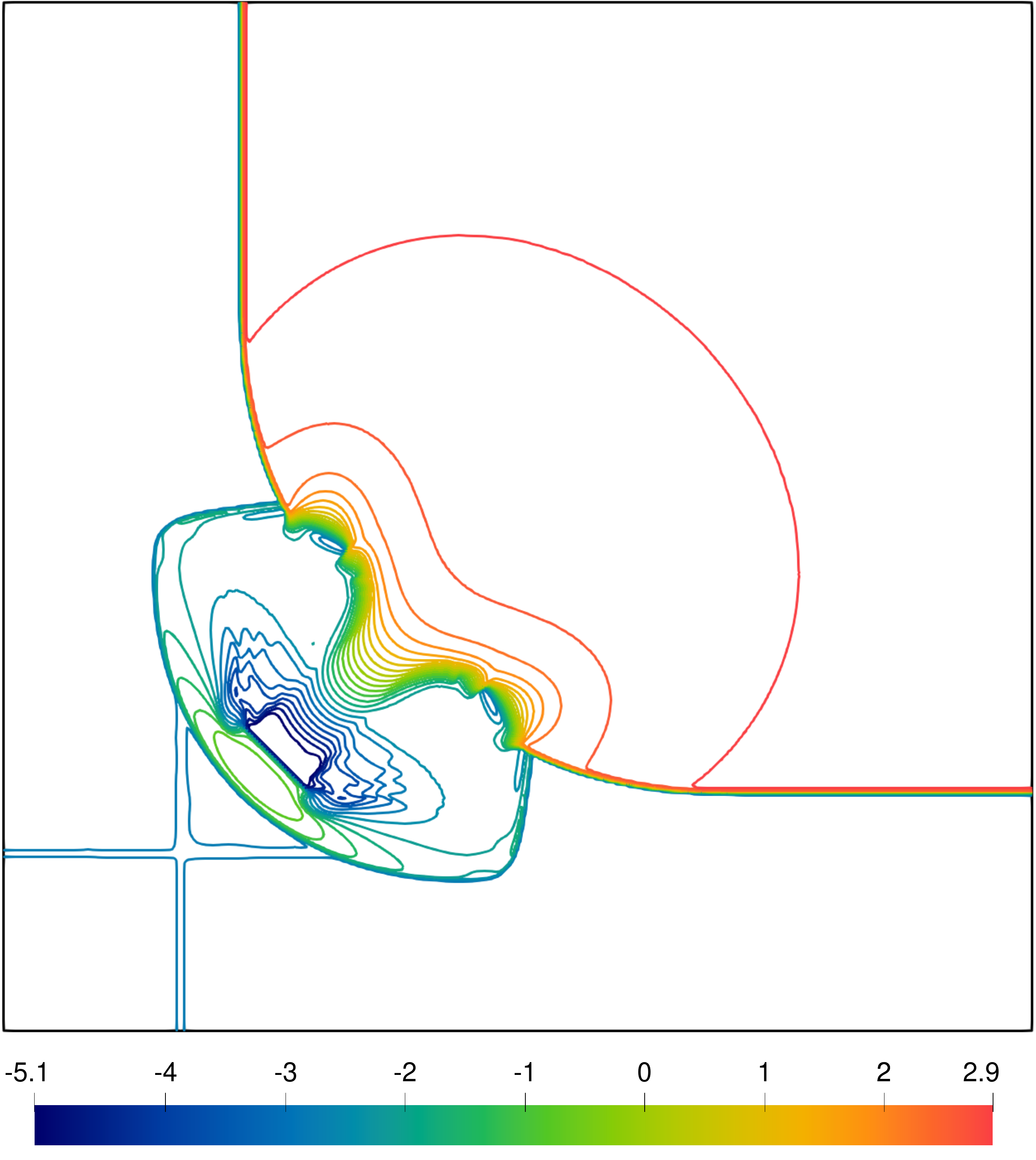}
         \caption{$\ln p$ with $25$ contour lines in $[-5.10, 2.90]$ using \rev{the scheme with degree $N=4$}.}
    \end{subfigure}
    \caption{Fifth Riemann problem in 2D: Plot of log of density and log of pressure using $400\times 400$ cells.}
    \label{2DRP5}
\end{figure}

\subsubsection{Fifth Riemann problem}
This test case is taken from~\cite{wu2016physical}, whose initial condition is given by,
\begin{equation*}
    (\rho, v_1, v_2, p) = \begin{cases}
        (0.1, 0, 0, 20) & \text{if}\ x > 0.5,\ y > 0.5\\
        (0.00414329639576, 0.9946418833556542, 0, 0.05) & \text{if}\ x < 0.5,\ y>0.5\\
        (0.01, 0, 0, 0.05) & \text{if}\ x < 0.5,\ y < 0.5\\
        (0.00414329639576, 0,0.9946418833556542, 0.05) & \text{if}\ x > 0.5,\ y<0.5.\\
    \end{cases}
\end{equation*}
We can see that in the upper-left and bottom-right quadrants, the density is very low with the fluid velocity very near to unity, making it a good test case to check the robustness of the scheme in terms of physical admissibility preservation. The simulation requires positivity correction, since otherwise the numerical solution violates the constraint $|v|<1$. We take the outflow boundary conditions at all the boundaries and run the simulation till time $t=0.4$ with $400\times 400$ cells and present the result in Figure~\ref{2DRP5}.
We observe that the \rev{scheme with degree $N=4$} can capture the curved shock structures more accurately than the \rev{scheme with degree $N=3$}.

\begin{figure}[!htbp]
    \centering
    \begin{subfigure}{0.8\textwidth}
         \includegraphics[width=\linewidth]{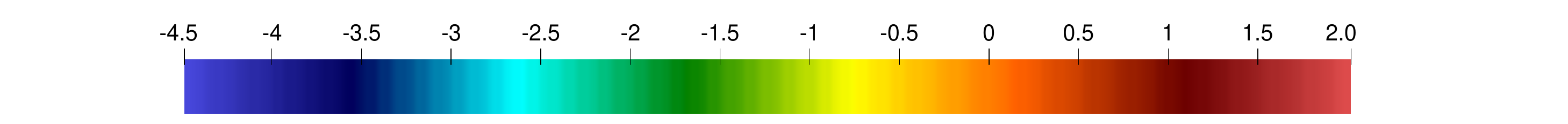}
    \end{subfigure}
    \begin{subfigure}{0.33\textwidth}
         \includegraphics[width=\linewidth]{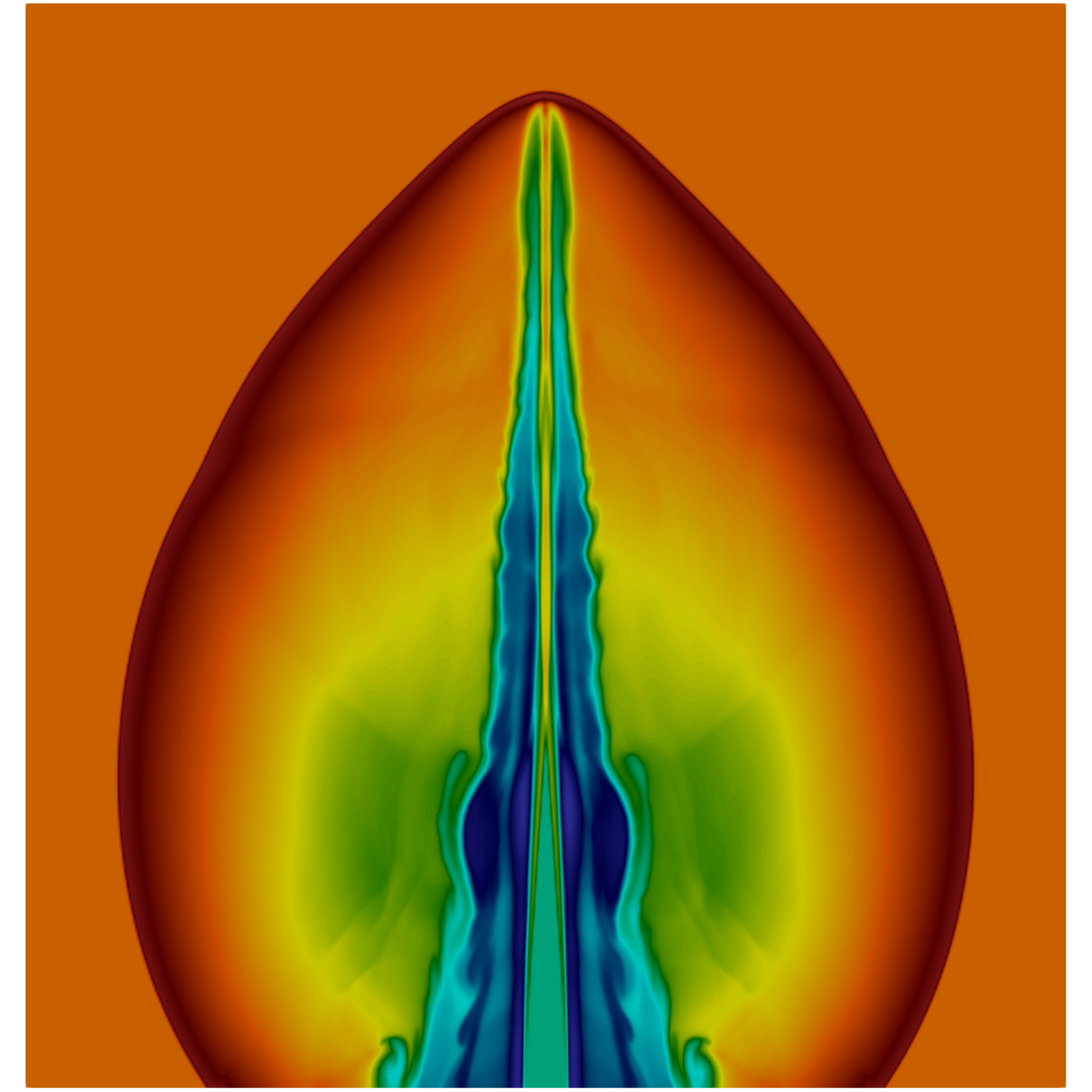}
         \caption{$\ln \rho$ for case~\ref{setup1}.}
    \end{subfigure}
    \begin{subfigure}{0.33\textwidth}
         \includegraphics[width=\linewidth]{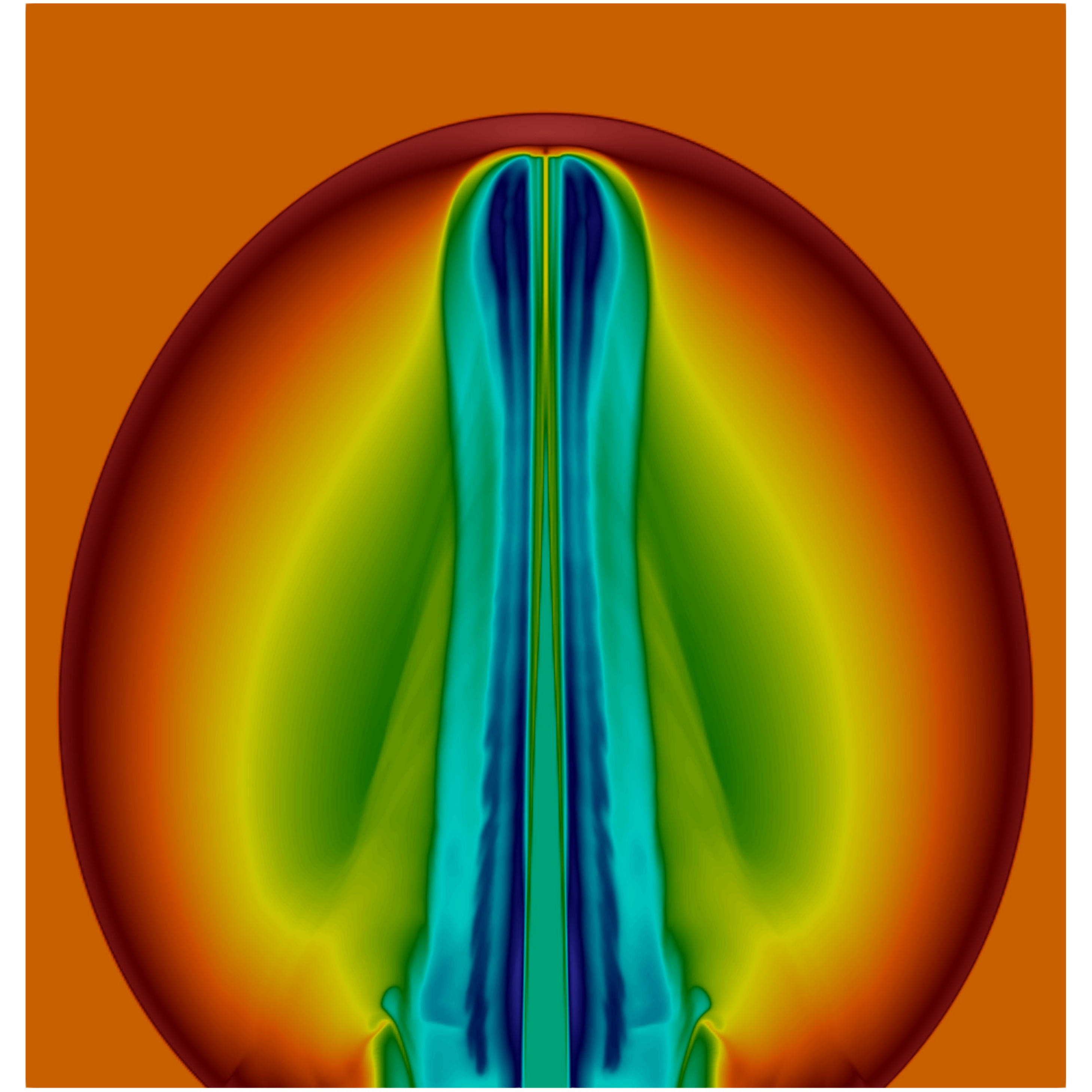}
         \caption{$\ln \rho$ for case~\ref{setup2}.}
    \end{subfigure}
    \begin{subfigure}{0.33\textwidth}
         \includegraphics[width=\linewidth]{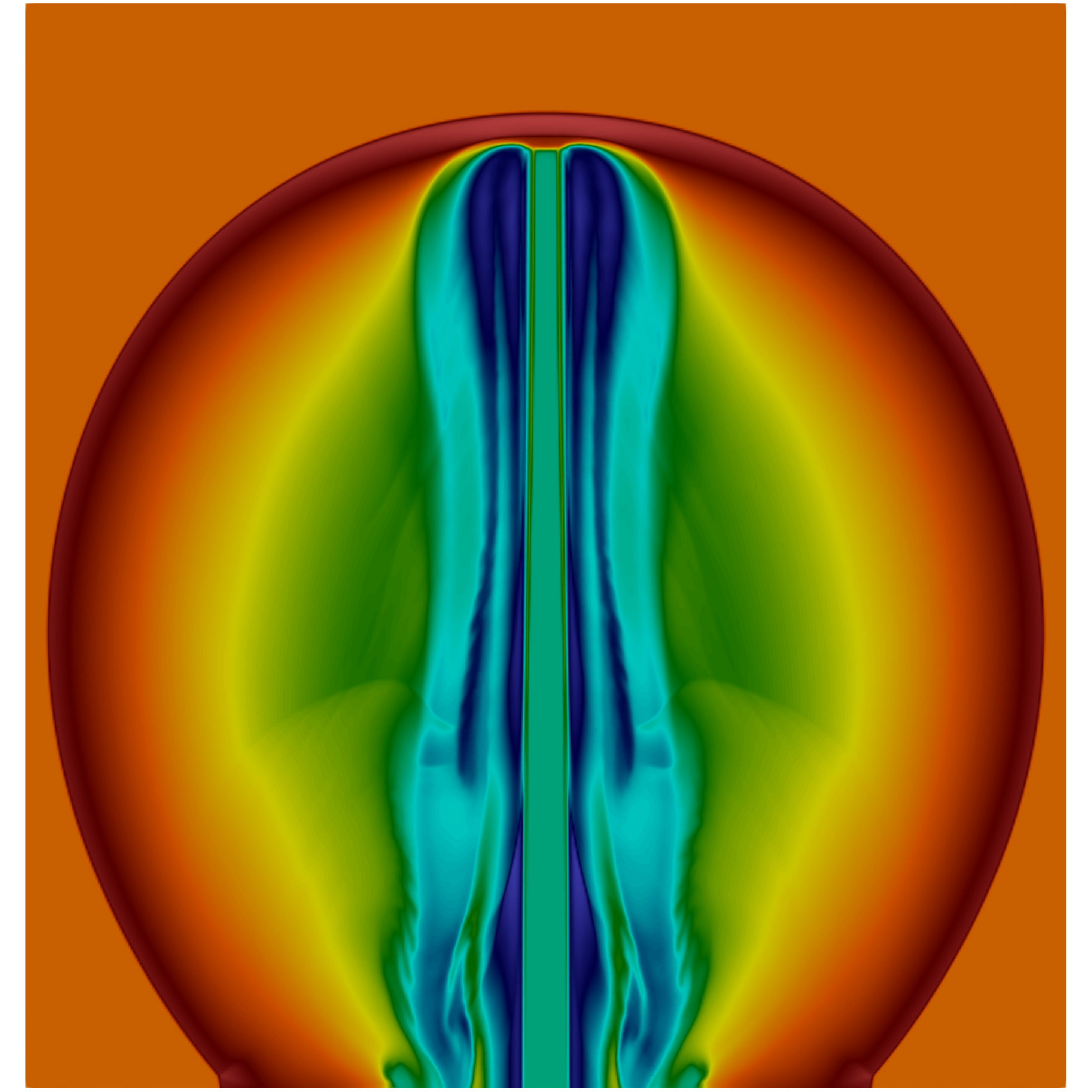}
         \caption{$\ln \rho$ for case~\ref{setup3}.}
    \end{subfigure}
    
    \caption{Relativistic jet problem: Plot of log of density using $480\times 500$ cells with \rev{the scheme with degree $N=4$}.}
    \label{reljet_density}
\end{figure}
\begin{figure}[!htbp]
    \centering
    \begin{subfigure}{0.8\textwidth}
         \includegraphics[width=\linewidth]{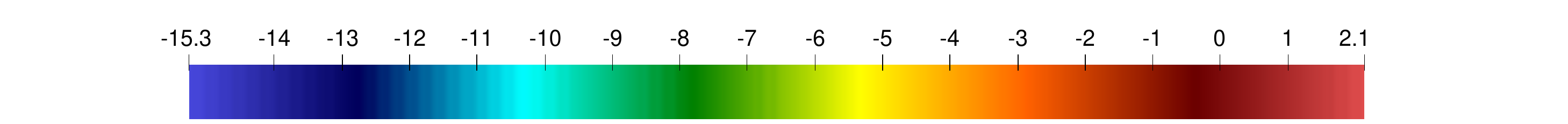}
    \end{subfigure}
    \begin{subfigure}{0.33\textwidth}
         \includegraphics[width=\linewidth]{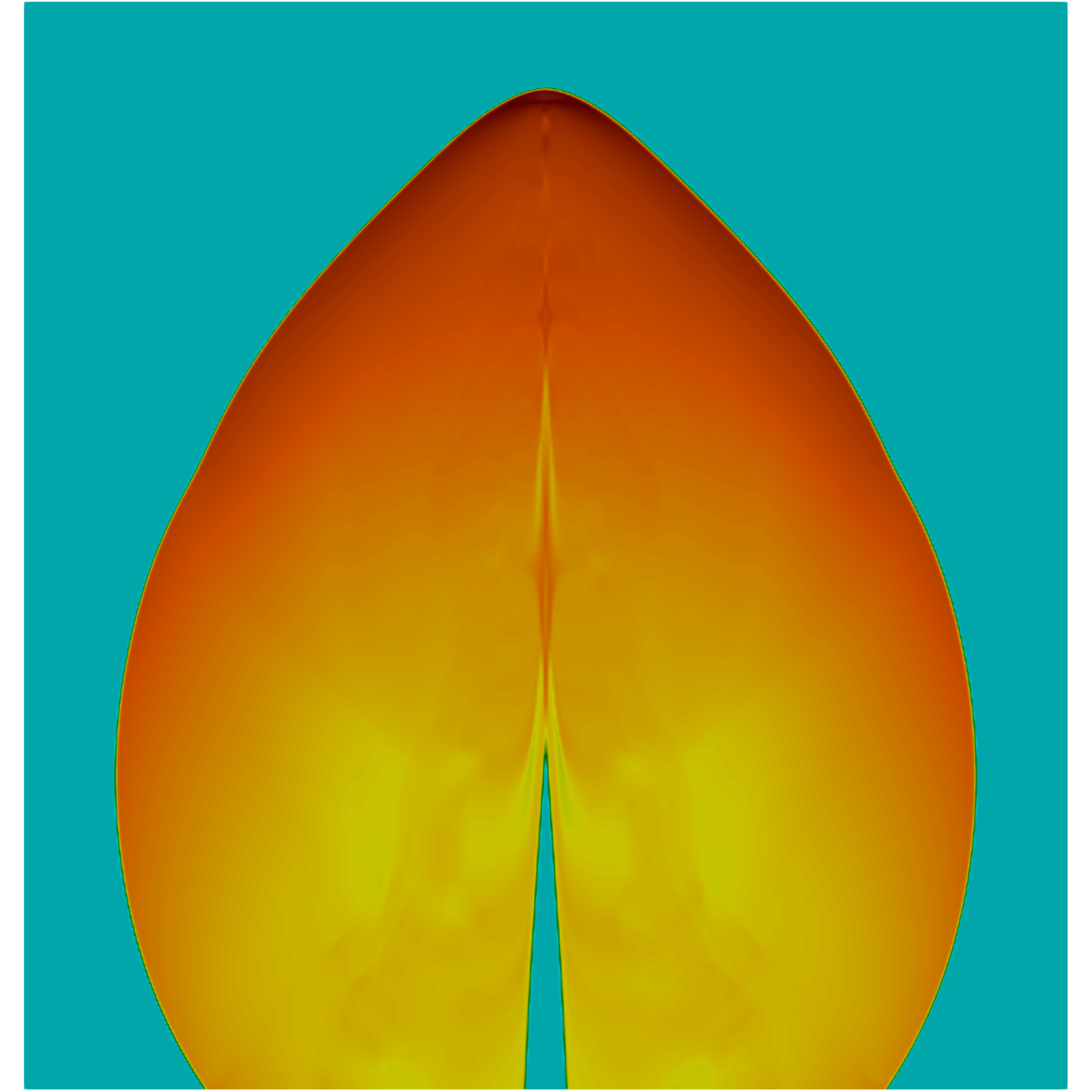}
         \caption{$\ln p$ for case~\ref{setup1}.}
    \end{subfigure}
    \begin{subfigure}{0.33\textwidth}
         \includegraphics[width=\linewidth]{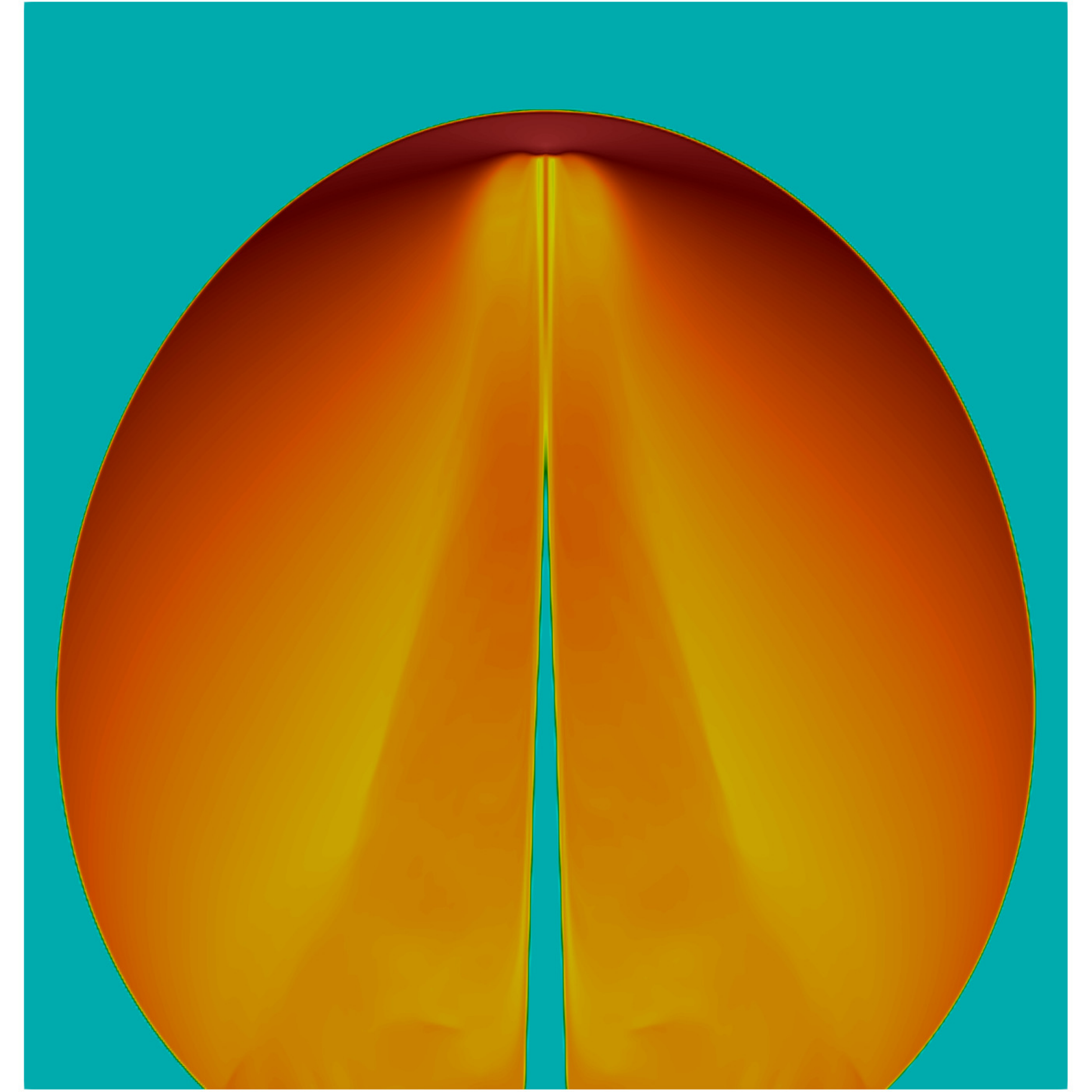}
         \caption{$\ln p$ for case~\ref{setup2}.}
    \end{subfigure}
    \begin{subfigure}{0.33\textwidth}
         \includegraphics[width=\linewidth]{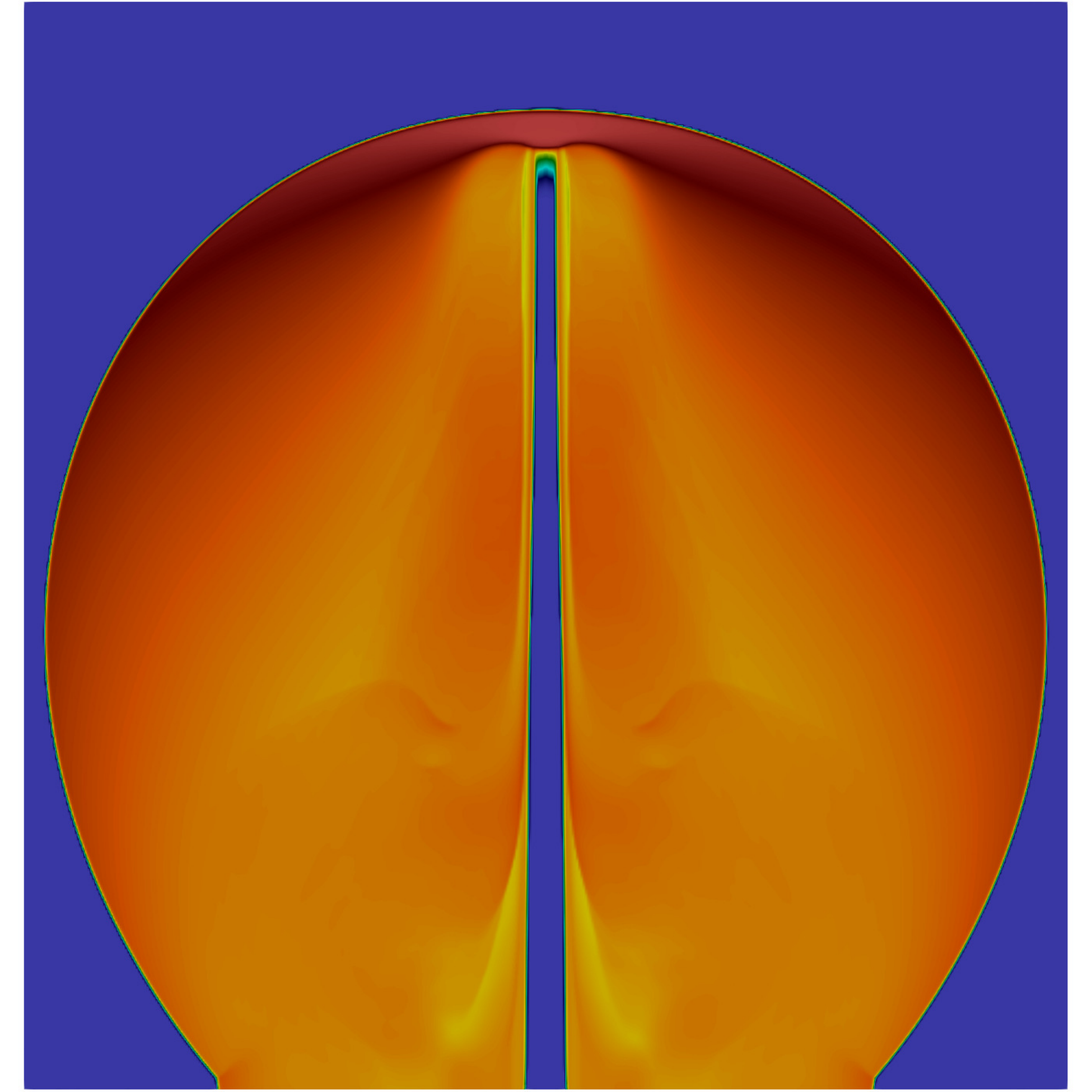}
         \caption{$\ln p$ for case~\ref{setup3}.}
    \end{subfigure}
    \caption{Relativistic jet problem: Plot of log of pressure using $480\times 500$ cells with \rev{the scheme with degree $N=4$}.}
    \label{reljet_pressure}
\end{figure}

\subsubsection{Relativistic jet problem}
This test case is taken from~\cite{wu2016physical}, which is relevant to model extra-galactic radio phenomena. This also plays a very important role in proving the robustness of the scheme because of the presence of a very high Lorentz factor and strong relativistic shock waves in its solution.

We consider three different experimental cases. But before that, let us fill the domain $[-12, 12]\times [0,25]$ with fluid of unit density. We consider the inflow boundary condition partially in the bottom such that $|x|<0.5$ according to the different cases explained below and outflow boundary conditions at all the other parts of the boundary. In all the test cases, we inject the same fluid but having density $0.1$ through $\{(x,y):|x|<0.5,y=0\}$ with different velocities.

\begin{enumerate}[label=(\roman*)]
    \item Injected velocity of the fluid is $0.99$, and pressure is the same as the ambient pressure with classical Mach number equal to $50$.\label{setup1}

    \item Injected velocity of the fluid is $0.999$, and pressure is the same as the ambient pressure with classical Mach number equal to $50$.\label{setup2}

    \item Injected velocity of the fluid is $0.9999$, and pressure is the same as the ambient pressure with classical Mach number equal to $500$.\label{setup3}
\end{enumerate}

We run the simulation till time $t=30, 25, 23$ for the first, second, and third cases respectively, with $480\times 500$ cells using the \rev{scheme with degree $N=4$}. The results are shown in Figure~\ref{reljet_density} and Figure~\ref{reljet_pressure}. We observe that the scheme can capture different wave structures of the solution effectively for all three cases having very high Lorentz factors of $7.09$, $22.37$, and $70.71$ respectively.

\begin{figure}[ht]
    \centering
    \begin{subfigure}{0.85\textwidth}
         \includegraphics[width=\linewidth]{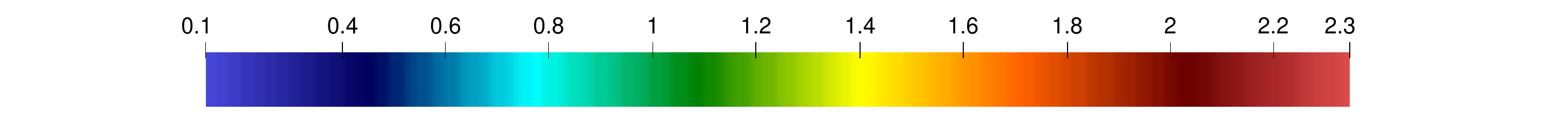}
    \end{subfigure}
    \begin{subfigure}{0.85\textwidth}
         \includegraphics[width=\linewidth]{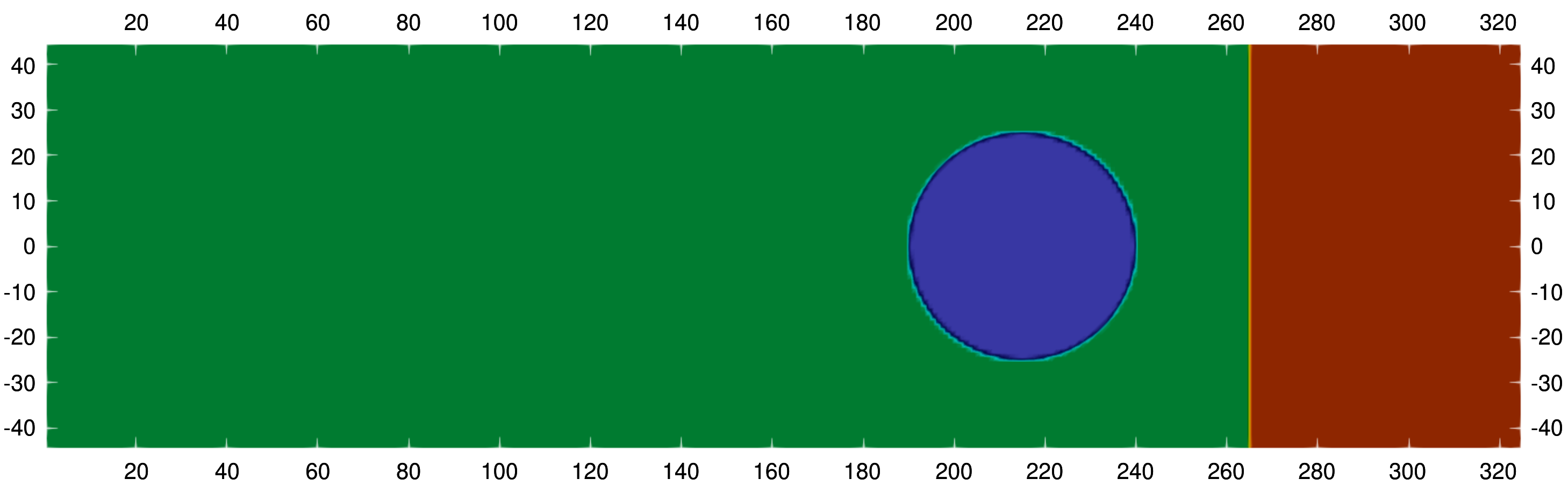}
         \caption{At time $t=0$.}
    \end{subfigure}
    \begin{subfigure}{0.85\textwidth}
         \includegraphics[width=\linewidth]{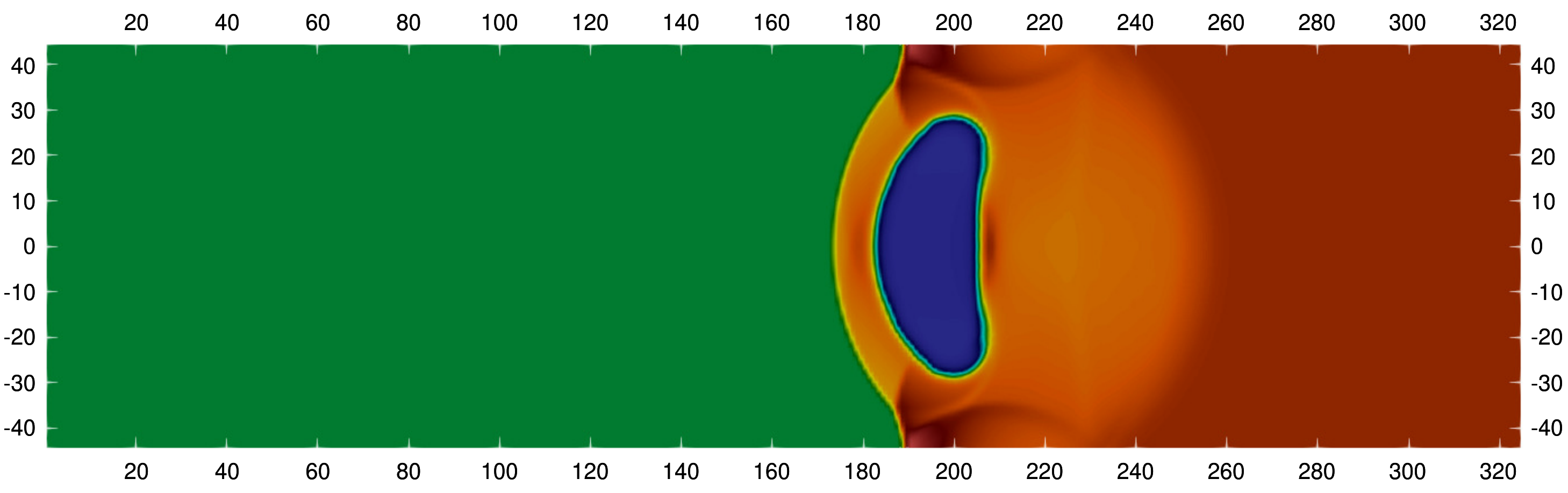}
         \caption{At time $t=180$.}
    \end{subfigure}
    \begin{subfigure}{0.85\textwidth}
         \includegraphics[width=\linewidth]{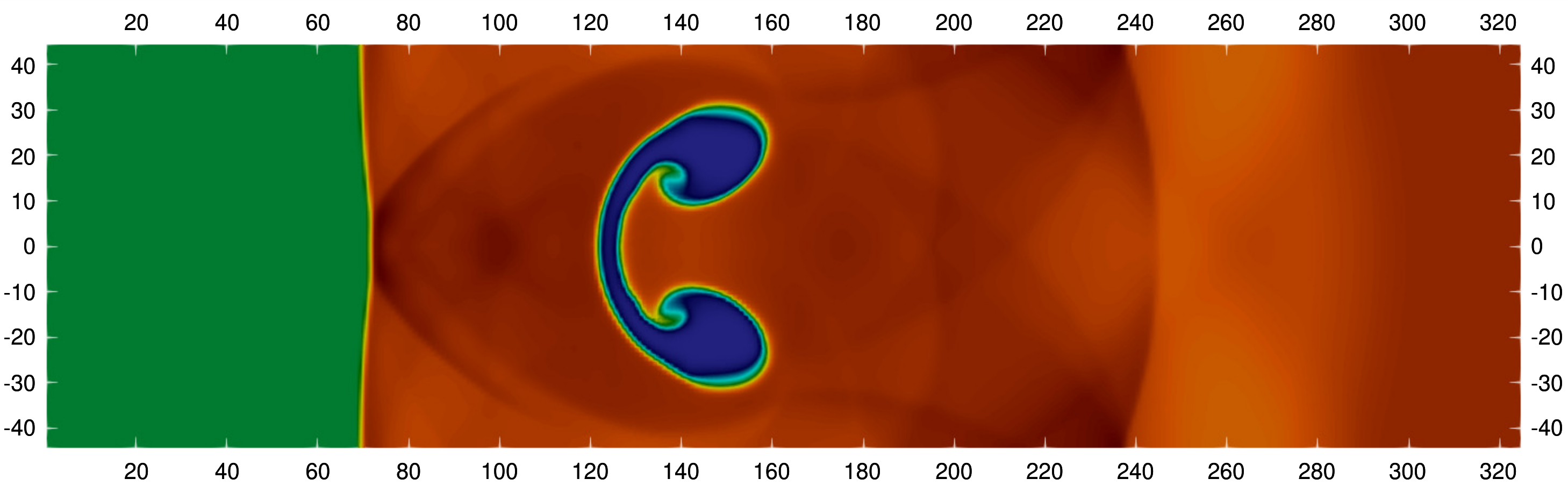}
         \caption{At time $t=450$.}
    \end{subfigure}
    \caption{Shock and bubble interaction: Plot of density using $325\times 90$ cells \rev{and using the scheme with degree $N=4$} for the case~\ref{setup1_bs}.}
    \label{shock_bubble1}
\end{figure}

 \begin{figure}[ht]
    \centering
    \begin{subfigure}{0.85\textwidth}
         \includegraphics[width=\linewidth]{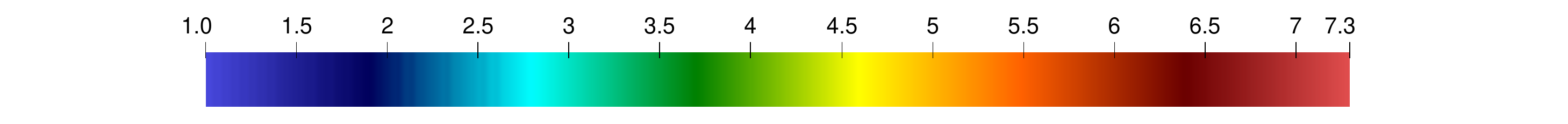}
    \end{subfigure}
    \begin{subfigure}{0.85\textwidth}
         \includegraphics[width=\linewidth]{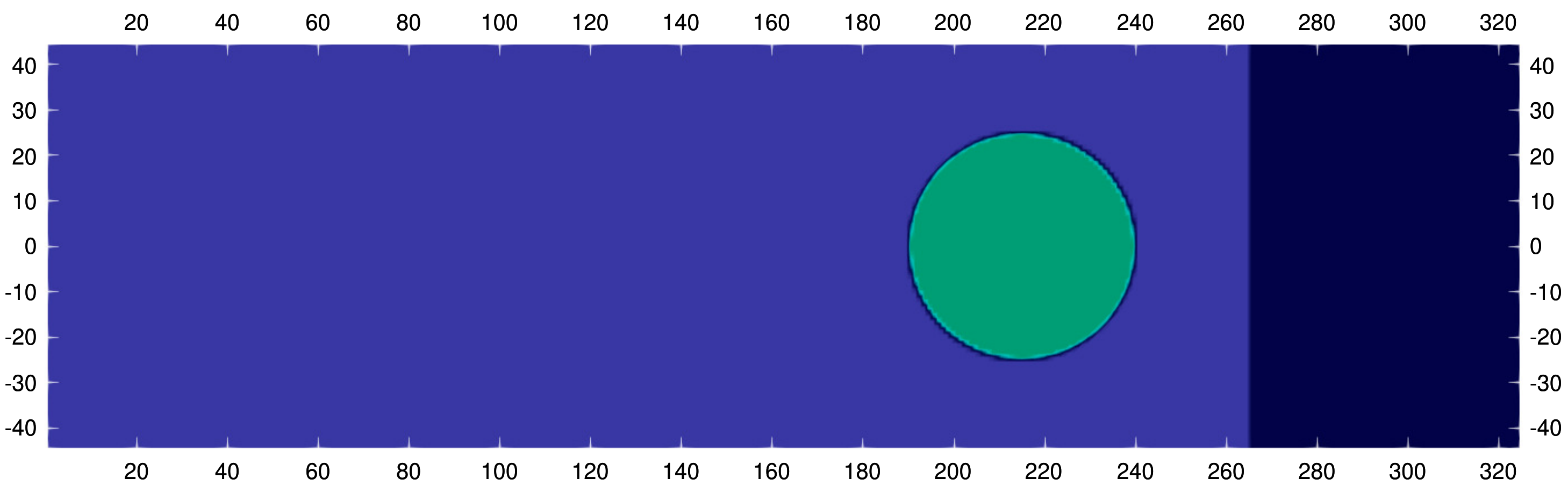}
         \caption{At time $t=0$.}
    \end{subfigure}
    \begin{subfigure}{0.85\textwidth}
         \includegraphics[width=\linewidth]{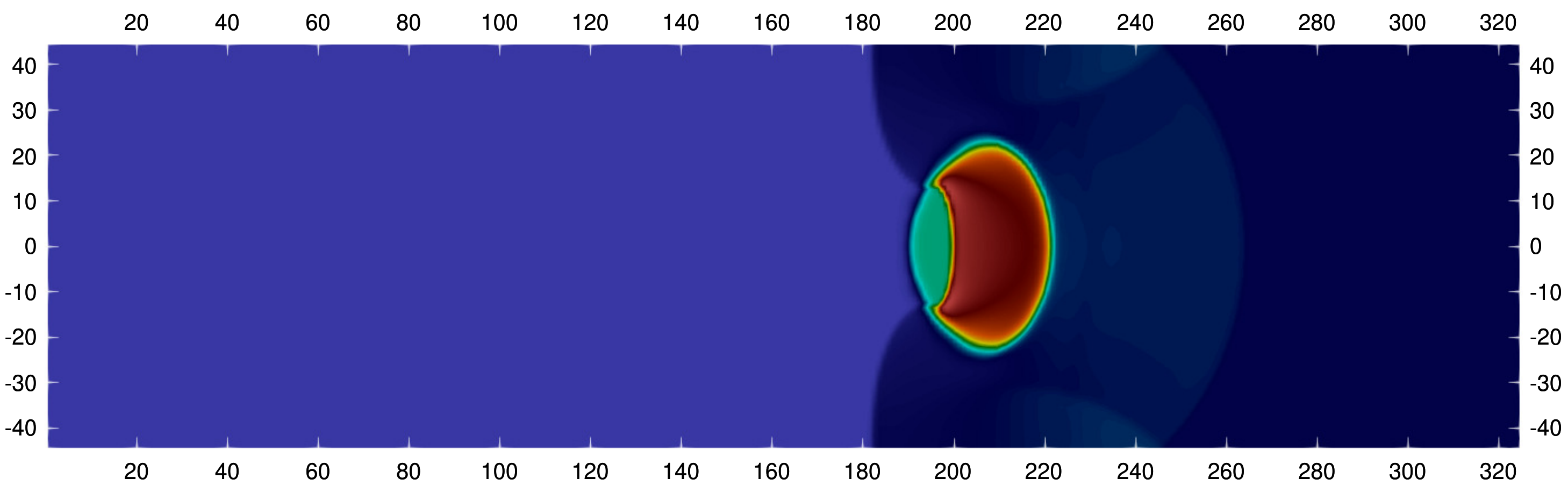}
         \caption{At time $t=200$.}
    \end{subfigure}
    \begin{subfigure}{0.85\textwidth}
         \includegraphics[width=\linewidth]{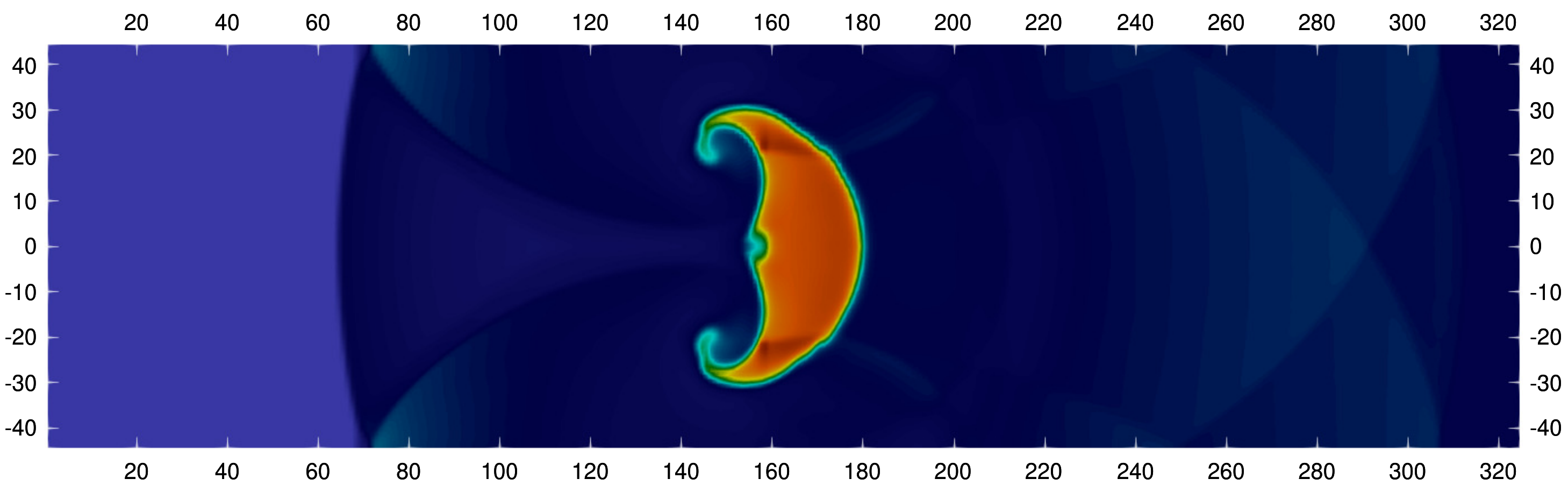}
         \caption{At time $t=500$.}
    \end{subfigure}
    
    \caption{Shock and bubble interaction: Plot of density using $325\times 90$ cells \rev{and using the scheme with degree $N=4$} for the case~\ref{setup2_bs}.}
    \label{shock_bubble2}
\end{figure}

\subsubsection{Shock and bubble interaction}
These types of problems are widely used in the literature for non-relativistic cases. We have taken this problem from~\cite{he2012adaptive}, where the authors have modified it for the relativistic case. This is a very interesting problem to check the performance of the scheme where a shock interacts with a bubble in the domain, forming various wave structures.

We consider the domain of computation to be $[0, 325]\times [-45, 45]$ with the reflective boundary conditions at the upper and bottom parts of the boundary and use constant left and right shock states at the left and right boundaries, respectively. Initially, the shock is placed at $x=265$ with the following initial condition,
\begin{equation*}
    (\rho, v_1, v_2, p) = \begin{cases}
        (1, 0, 0, 0.05) & \text{if}\ x < 265\\
        (1.865225080631180, -0.196781107378299, 0, 1.5) & \text{if}\ x > 265\\
    \end{cases}
\end{equation*}
 and a bubble of radius $25$ is placed inside the domain with center at $(215,0)$ in two different cases:
 \begin{enumerate}[label=(\roman*)]
     \item The bubble is filled with a lighter fluid having $\text{density} = 0.1358$ with $\text{pressure} = 0.05$.\label{setup1_bs}
     \item The bubble is filled with a heavier fluid having $\text{density} = 3.1538$ with $\text{pressure} = 0.05$.\label{setup2_bs}
 \end{enumerate}
 The simulation is run till $t=450$, $t=500$ for the cases~\ref{setup1_bs},~\ref{setup2_bs} respectively, and the density profiles are presented in Figure~\ref{shock_bubble1} and Figure~\ref{shock_bubble2} respectively at several time levels. We observe that the scheme can effectively capture all the waves in the solution during and after the interaction.

\begin{figure}[htbp]
    \centering
    \begin{subfigure}{0.45\textwidth}
         \includegraphics[width=\linewidth]{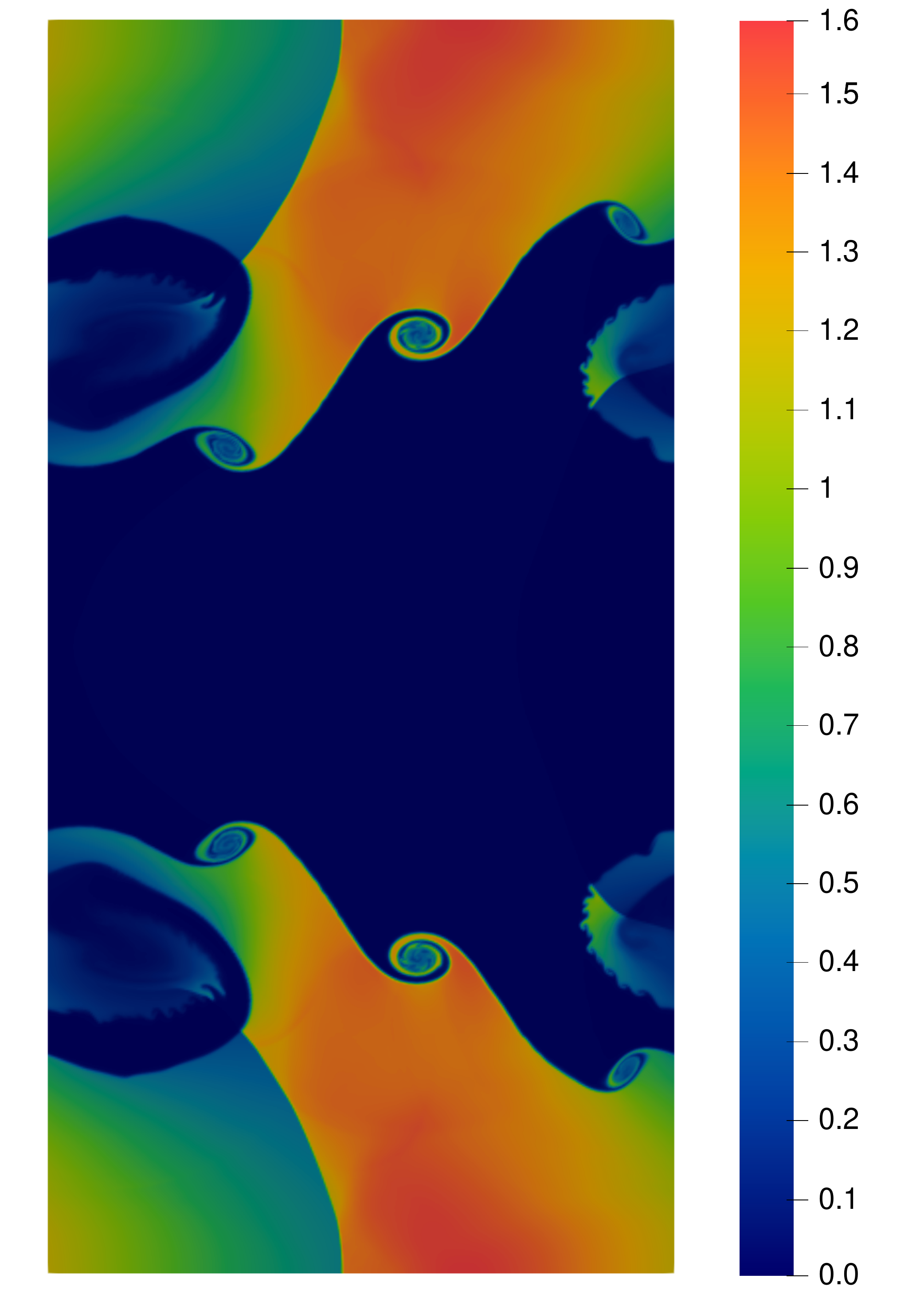}
         \caption{Using our indicator model.}
    \end{subfigure}
    \begin{subfigure}{0.45\textwidth}
         \includegraphics[width=\linewidth]{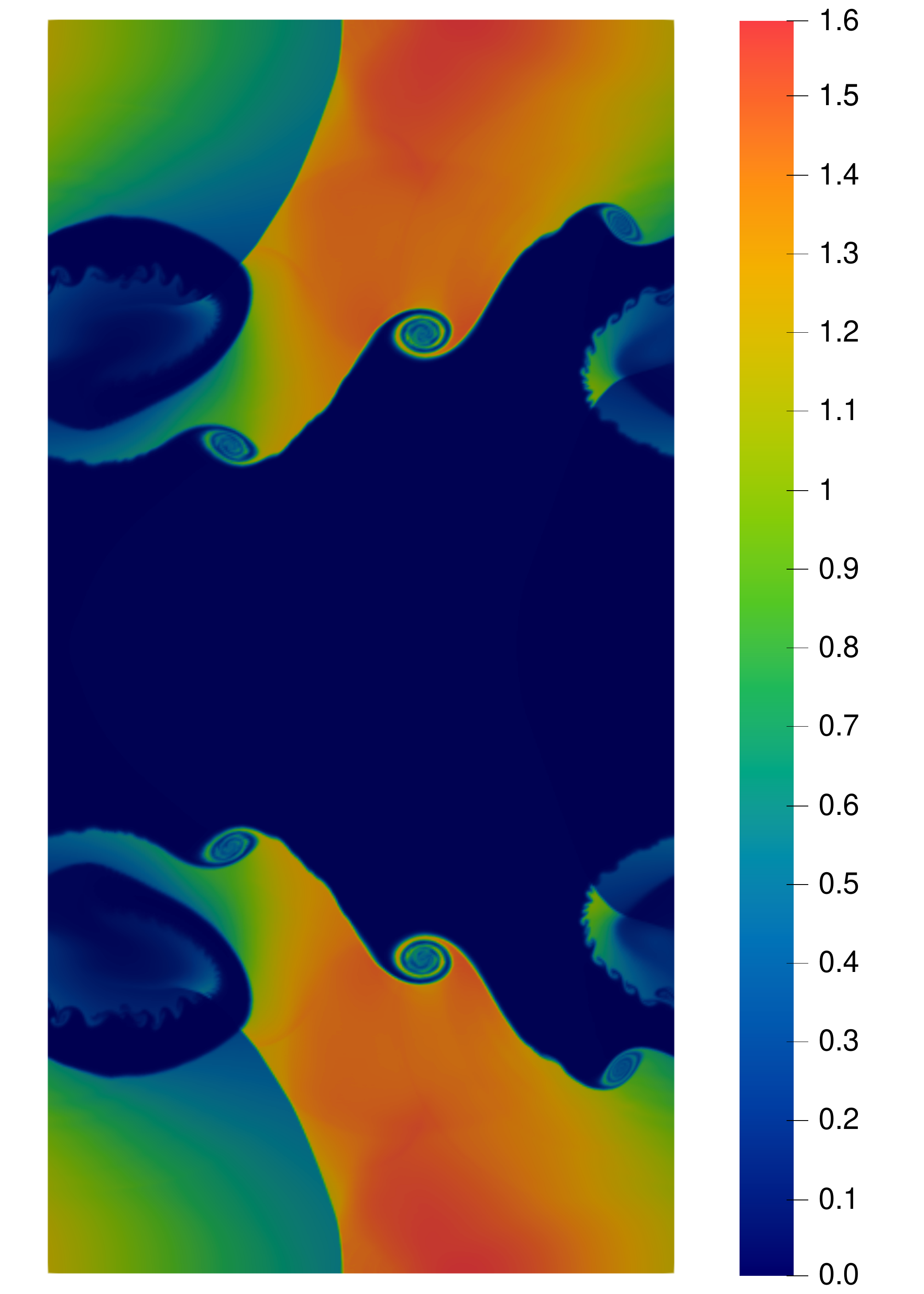}
         \caption{Using the indicator model from~\cite{hennemann2021provably}.}
    \end{subfigure}
    \caption{Kelvin-Helmholtz instability: Plot of density using $320\times 640$ cells.}
    \label{KH_instability}
\end{figure}
 \subsubsection{Kelvin–Helmholtz instability}
 Kelvin-Helmholtz instability is a classical benchmark problem for the RHD codes, where the small circles of instabilities get formed which are very hard to capture by a diffusive scheme. Following~\cite{beckwith2011second,radice2012thc,zanotti2015high}, we take the initial fluid density as,
 \[
 \rho = \begin{cases}
     0.505 + 0.495 \tanh\left(\frac{y-0.5}{a}\right)\qquad &\text{if}\ y>0\\
     0.505 - 0.495 \tanh\left(\frac{y+0.5}{a}\right)\qquad &\text{if}\ y\leq 0,
 \end{cases}
 \]
 where the characteristic size, $a=0.01$. The initial fluid velocity in $x$-direction is taken as,
 \[
 v_{\rev{1}} = \begin{cases}
     v_s \tanh\left(\frac{y-0.5}{a}\right)\qquad &\text{if}\ y>0\\
     -v_s \tanh\left(\frac{y+0.5}{a}\right)\qquad &\text{if}\ y\leq0,
 \end{cases}
 \]
 where $v_s=0.5$ is the velocity at the shear layer, and the initial velocity in $y$-direction is taken with a small perturbation to trigger the small instabilities,
  \[
 v_{\rev{2}} = \begin{cases}
     \eta_0 v_s \sin(2 \pi x)\exp\left(\frac{-(y-0.5)^2}{\sigma}\right)\qquad &\text{if}\ y>0\\
     -\eta_0 v_s \sin(2 \pi x)\exp\left(\frac{-(y+0.5)^2}{\sigma}\right)\qquad &\text{if}\ y\leq0,
 \end{cases}
 \]
 where $\eta_0 = 0.1$ and $\sigma = 0.1$ are the amplitude and length scale of the perturbation. Pressure is taken as unity in the whole domain.

 We consider the computational domain to be $[-0.5, 0.5]\times [-1.0, 1.0]$ with periodic boundaries. We run the simulation using the \rev{scheme with degree $N=4$} till time $t=3$ with $320\times 640$ cells and a smaller $\alpha'_\text{max} = 0.25$~\eqref{eq:amax.defn} as the problem contains only shear/contact discontinuities which are linear waves. The density profile at the final time is presented in Figure~\ref{KH_instability} using both the new and the old indicator models~\cite{hennemann2021provably}, and we observe that the scheme does not lag behind in capturing the secondary instability structures along the shear layer with the new indicator model compared to the old model which was more focused on capturing small scale structures rather than controlling small scale oscillations.

\section{Summary and conclusions}\label{sec: summary}
In this work, we have designed a Jacobian-free Lax-Wendroff flux reconstruction scheme for the RHD equations. In~\cite{BABBAR2022111423}, the authors have discussed this scheme and shown the results for the Euler equations. The non-linearity of the RHD equations can lead to the formation of shocks around which solutions can be oscillatory, which have been controlled by blending the scheme with a first-order finite volume scheme on sub-cells according to a discontinuity indicator model. The strength of the blending scheme, as has been numerically demonstrated in this work, is in its accuracy improvement in comparison with the TVB limiter which loses most of the sub-cell information. However, the indicator model of~\cite{hennemann2021provably} used in~\cite{babbar2024admissibility} for the Euler equations sometimes fails to capture the irregularities in the solution of the highly non-linear RHD equations and hence we have designed a discontinuity indicator model with the help of the Legendre orthogonal basis functions which can more reliably capture the local irregularities in the solution. This indicator model can also be used in other schemes to solve different conservation laws. For the admissibility part, we have first made the element means of the solution admissible, which is achieved by blending the high-order numerical flux with that of an admissibility preserving first-order finite volume scheme defined on sub-cells; and then using a scaling limiter from~\cite{zhang2010maximum} to make the solution admissible at all solution points. But this method needs the constraints in the admissibility region to be concave functions of the conserved variables, and unfortunately, this is not the case for the RHD equations. Hence, following~\cite{wu2015high}, we have taken an equivalent admissibility region having concave constraints that are expressible in the form of conservative variables; consequently, the computational costs for checking the admissibility of the solution were reduced to a large extent.  Lastly, we have shown the numerical results \rev{using the scheme with degrees $N=3$ and $N=4$} of the RHD equations with a wide class of different initial conditions and have found that the scheme can capture all the waves in the solution effectively without producing many oscillations even for the case of strong discontinuities and having very high Lorentz factor. The admissibility preserving limiter is found to be necessary and effective in problems where the solution is close to the boundaries of the admissible region.

\section*{Acknowledgements}
The work of Sujoy Basak is supported by the Prime Minister's Research Fellowship with PMRF ID 1403239.
The work of Arpit Babbar and Praveen Chandrashekar is supported by the Department of Atomic Energy,  Government of India, under project no.~12-R\&D-TFR-5.01-0520.

\revc{\section*{Data availability}
The code supporting the findings of this work will be made publicly available at~\cite{jcp2025} along with the data used.}

\appendix

\revc{
\section{Admissibility preservation property of first-order finite volume scheme}\label{sec: appendix}
From~\cite{ryu2006equation,bhoriya2020entropy}, the expressions of eigenvalues of the flux Jacobian of the RHD equations~\eqref{RHD_equation} in $x_1$ direction  are
\begin{align}
\begin{split}
    &\lambda_1 = \frac{v_1(1-s^2) - s \Gamma^{-1} \sqrt{(1-|\mb{v}|^2 s^2)-(1-s^2)v_1^2}}{1-|\mb{v}|^2 s^2},\quad  \lambda_2=\frac{v_1(1-s^2) + s \Gamma^{-1} \sqrt{(1-|\mb{v}|^2 s^2)-(1-s^2)v_1^2}}{1-|\mb{v}|^2 s^2},\\
    &\lambda_3=\lambda_4=v_1,
\end{split}
\end{align}
where $0<s<1$ is the speed of sound. One can easily verify that the corresponding spectral radius,
\begin{equation}\label{eq:spectral_radius}
\Lambda = \max\{|\lambda_1|, |\lambda_2|, |\lambda_3|, |\lambda_4|\}    =\frac{|v_1|(1-s^2) + s \Gamma^{-1} \sqrt{(1-|\mb{v}|^2 s^2)-(1-s^2)v_1^2}}{1-|\mb{v}|^2 s^2}.
\end{equation}
Let us now prove two inequalities regarding $\Lambda$, which will be used later. From equation~\eqref{eq:spectral_radius} we have,
\begin{align}
     \Lambda  &= \frac{|v_1|(1-s^2) + s \Gamma^{-1} \sqrt{1-v_1^2-s^2(|\mb{v}|^2-v_1^2)}}{1-|\mb{v}|^2 s^2}\nonumber\\
     &> \frac{|v_1|(1-s^2) + s \Gamma^{-1} \sqrt{1-v_1^2-(|\mb{v}|^2-v_1^2)}}{1-|\mb{v}|^2 s^2}, \quad \text{since}\ 0 < s < 1,\ \Gamma > 0,\ |\mb{v}| \ge |v_1|\nonumber\\
    &=\frac{|v_1|(1-s^2) + s(1-|\mb{v}|^2)}{1-|\mb{v}|^2 s^2}.\label{eq: Gamma}
\end{align}
Again by some algebraic manipulations in equation~\eqref{eq:spectral_radius}, we have,
\begin{align}
    \Lambda^2 (1-s^2|\mb{v}|^2) -2 \Lambda |v_1| (1-s^2) +|v_1|^2(1-s^2) &= \frac{s^2}{\Gamma^2}\nonumber\\
    (1-\Lambda^2)s^2 &= \Gamma^2(\Lambda-|v_1|)^2(1-s^2)\label{eq: Lambda2}
\end{align}
and since $0<s<1$ we have the second inequality for $\Lambda$,
\begin{equation}\label{eq: Lambda3}
    \Lambda < 1.
\end{equation}
Now, we have the following expression regarding the sound speed $s$ (see equation (9) in~\cite{ryu2006equation}) for the ideal equation of state~\eqref{eq: eos},
\begin{align}
    s^2 &= \frac{\gamma (p/\rho)(\gamma - 1)}{\gamma (p/\rho)+ \gamma - 1}\label{eq: sound_speed1}\\
    s^2 &= \frac{\gamma p}{\rho h} \label{eq: sound_speed2}
\end{align}
At a solution point $(\xi_i, \eta_j)$, $0<i,j<N$ in element $\Omega_{pq}$, the low-order scheme~\eqref{eq:loworder_scheme} with the Rusanov fluxes~\eqref{eq: rusanov_flux} as the numerical fluxes can be written as,
\begin{align}
&\mb{u}_{ij}^{n+1} = \mb{u}_{ij}^n - \frac{\Delta t}{w_i \Delta x_p}[\mb{f}^{\text{NF}}(\mb{u}_{ij}, \mb{u}_{i+1,j})- \mb{f}^{\text{NF}}(\mb{u}_{i-1,j}, \mb{u}_{ij})]\nonumber 
- \frac{\Delta t}{w_j \Delta y_q}[\mb{g}^{\text{NF}}(\mb{u}_{ij}, \mb{u}_{i,j+1})- \mb{g}^{\text{NF}}(\mb{u}_{i,j-1}, \mb{u}_{ij})]\nonumber\\
    &\mb{u}_{ij}^{n+1} = A \mb{u}_{ij}^n + B_1 \mb{u}_{\Lambda_{i-\frac{1}{2},j}}^{n,x+} + B_2 \mb{u}_{\Lambda_{i+\frac{1}{2},j}}^{n,x-} + C_1 \mb{u}_{\Lambda_{i,j-\frac{1}{2}}}^{n,y+} +C_2 \mb{u}_{\Lambda_{i,j+\frac{1}{2}}}^{n,y-}    \label{eq: Rusanov scheme}
\end{align}
where the coefficients are given by,
\begin{align}\label{eq: rusanov_coefficients}
\begin{split}
    A &= \left[1-\frac{\Delta t}{2 w_i \Delta x_p}(\Lambda_{i-\frac{1}{2},j} + \Lambda_{i+\frac{1}{2},j}) -\frac{\Delta t}{2 w_j \Delta y_q}(\Lambda_{i,j-\frac{1}{2}} + \Lambda_{i,j+\frac{1}{2}})\right],\\
    B_1 &= \frac{\Delta t \Lambda_{i-\frac{1}{2},j}}{2 w_i \Delta x_p}, \qquad B_2 =  \frac{\Delta t \Lambda_{i+\frac{1}{2},j}}{2 w_i \Delta x_p},\\
    C_1 &= \frac{\Delta t \Lambda_{i,j-\frac{1}{2}}}{2 w_j \Delta y_q}, \qquad C_2= \frac{\Delta t \Lambda_{i,j+\frac{1}{2}}}{2 w_j \Delta y_q}.
\end{split}
\end{align}
Here, 
\begin{equation}\label{eq: Lambda_max}
\Lambda_{i\pm\frac{1}{2},j} = \max\big\{r\big(f'(\mb{u}_{i\pm 1,j}^n)\big), r\big(f'(\mb{u}_{ij}^n)\big)\big\},\qquad  \Lambda_{i,j\pm\frac{1}{2}} = \max\big\{r\big(g'(\mb{u}_{i,j\pm 1}^n)\big), r\big(g'(\mb{u}_{ij}^n)\big)\big\}
\end{equation}
with $r(\cdot)$ denoting the spectral radius and
\begin{align}\label{eq:directional_evolutions_first}
\begin{split}
    &\mb{u}_{\Lambda_{i-\frac{1}{2},j}}^{n,x+} = \mb{u}^n_{i-1,j} + \frac{1}{\Lambda_{i-\frac{1}{2},j}}\mb{f}(\mb{u}^n_{i-1,j}), \qquad \mb{u}_{\Lambda_{i+\frac{1}{2},j}}^{n,x-} = \mb{u}^n_{i+1,j} - \frac{1}{\Lambda_{i+\frac{1}{2},j}}\mb{f}(\mb{u}^n_{i+1,j}),\\
    &\mb{u}_{\Lambda_{i,j-\frac{1}{2}}}^{n,y+} = \mb{u}^n_{i,j-1} + \frac{1}{\Lambda_{i,j-\frac{1}{2}}}\mb{g}(\mb{u}^n_{i,j-1}), \qquad \mb{u}_{\Lambda_{i,j+\frac{1}{2}}}^{n,y-} = \mb{u}^n_{i,j+1} - \frac{1}{\Lambda_{i,j+\frac{1}{2}}}\mb{g}(\mb{u}^n_{i,j+1}).
\end{split}
\end{align}
Note that we have suppressed the superscript $L$ and the element index $e$ for notational simplicity. Now, to prove that $\mb{u}_{ij}^{n+1}$ is admissible, it is sufficient to prove that the coefficients $A,B_1,B_2,C_1, C_2$ are positive and $\mb{u}_{\Lambda_{i-\frac{1}{2},j}}^{n,x+}, \mb{u}_{\Lambda_{i+\frac{1}{2},j}}^{n,x-}, \mb{u}_{\Lambda_{i,j-\frac{1}{2}}}^{n,y+}, \mb{u}_{\Lambda_{i,j+\frac{1}{2}}}^{n,y-}$ are admissible. Then the convex combination~\eqref{eq: Rusanov scheme} will be in the admissible set $\Uad'$, as it is a convex set (Lemma~\ref{equivalent_admissibility_region}). 
 
Now to prove the admissibility of the $x$-directional quantities in equation~\eqref{eq:directional_evolutions_first}, we proceed by proving the admissibility of the quantities,
\begin{align}\label{eq:directional_evolutions_second}
\begin{split}
    &\mb{u}_{\Lambda_{i-1,j}}^{n,x+} = \mb{u}^n_{i-1,j} + \frac{1}{\Lambda_{i-1,j}}\mb{f}(\mb{u}^n_{i-1,j}), \qquad \mb{u}_{\Lambda_{i+1,j}}^{n,x-} = \mb{u}^n_{i+1,j} - \frac{1}{\Lambda_{i+1,j}}\mb{f}(\mb{u}^n_{i+1,j})
\end{split}
\end{align}
where $\Lambda_{i\pm1, j} = r\big(f'(\mb{u}_{i\pm 1,j}^n)\big)$, assuming the admissibility of $\mb{u}^n_{i\pm1,j}$ and following the idea  in~\cite{wu2017design}, where the author has proved the admissibility of similar quantities for an upper bound of the spectral radius.

For notational simplicity, we will again ignore the time and spatial indices in equation~\eqref{eq:directional_evolutions_second} and argue for $\mb{u}_{\Lambda}^{x\pm}$. Let us denote the admissibility constraints $D,q$~\eqref{eq:uad.defn} for $\mb{u}_\Lambda^{x\pm}$ by,
\begin{align}\label{eq:ad.constraint.app}
    D^{x\pm}_\Lambda = \mb{u}^{x\pm}_\Lambda[1],\qquad q^{x\pm}_\Lambda = \mb{u}_\Lambda^{x\pm}[4] - \sqrt{(\mb{u}_\Lambda^{x\pm}[1])^2 + (\mb{u}_\Lambda^{x\pm}[2])^2 + (\mb{u}_\Lambda^{x\pm}[3])^2}.
\end{align}
For the first constraint,
\[
D_\Lambda^{x\pm} = D \pm \frac{1}{\Lambda} D v_1  = D\left[1 \pm \frac{v_1}{\Lambda}\right]
            > 0,    \quad \text{since}\ \Lambda>|v_1|>0,\ D>0.
\]
Now,
\begin{align*}
    \mb{u}_\Lambda^{x\pm}[4] &= E \pm \frac{1}{\Lambda}m_1\\
        &\geq E - \frac{1}{\Lambda}|m_1|,\quad \text{since}\ \Lambda>0 \\
        &= \rho h\Gamma^2\left(1- \frac{1}{\Lambda}|v_1|\right) - p\\
        &> \rho h\Gamma^2 \left[1 - \left(\frac{|v_1|(1-s^2) + s(1-|\mb{v}|^2)}{1-|\mb{v}|^2 s^2} \right)^{-1} |v_1|\right] -p,\quad\ \text{using~\eqref{eq: Gamma}}\\
        &= \rho h\Gamma^2 \left[\frac{s (1-|\mb{v}|^2) (1 - |v_1| s)}{|v_1| (1-s^2) + s (1-|\mb{v}|^2)} \right] -p\\
        &\geq \rho h\Gamma^2 \left[\frac{s (1-|\mb{v}|^2) (1 - |v_1| s)}{|v_1| (1-s^2) + s (1-v_1^2)} \right] -p,\quad\ \text{since}\ |v_1|<|\mb{v}|<1,\ 0<s<1\  \text{and}\ \rho, h >0\\
        &=\rho h\Gamma^2 \left[\frac{s(1-|\mb{v}|^2)}{|v_1|+s} \right] -p\\
        &\geq \rho h \left(\frac{s}{1+s}\right) - p,\quad \ \text{since}\ |v_1|<1\ \text{and}\ \rho, h, s >0\\
        &= \left(\frac{\gamma p}{s^2}\right) \left(\frac{s}{1+s}\right) -p,\quad\ \text{using~\eqref{eq: sound_speed2}}\\
        &> \frac{\gamma p}{\sqrt{\gamma -1} + \gamma - 1} - p,\quad\ \text{since}\ s<\sqrt{\gamma -1},\quad \text{from~\eqref{eq: sound_speed1}}\\
        &\geq 0,\quad\ \text{since}\ \gamma\in(1,2]\ \text{and}\ p>0.
\end{align*}
Hence showing $(\mb{u}_\Lambda^{x\pm}[1])^2 + (\mb{u}_\Lambda^{x\pm}[2])^2 + (\mb{u}_\Lambda^{x\pm}[3])^2 - (\mb{u}_\Lambda^{x\pm}[4])^2 < 0$ is sufficient to show $q_\Lambda^{x\pm} > 0$~\eqref{eq:ad.constraint.app}. From equation~(6) of~\cite{wu2016physical} we have,
\[
    h \geq \sqrt{1 + \frac{p^2}{\rho^2}} + \frac{p}{\rho}
\]
and after some algebraic manipulations we get,
\begin{align}\label{eq:weak_taub2}
    (\rho h - p)^2 \geq (\rho^2 + p^2)
\end{align}
which will be used later. Now,
\begin{align*}
    & (\mb{u}_\Lambda^{x\pm}[1])^2 + (\mb{u}_\Lambda^{x\pm}[2])^2 + (\mb{u}_\Lambda^{x\pm}[3])^2 - (\mb{u}_\Lambda^{x\pm}[4])^2\\
    &=\left(D\pm \frac{1}{\Lambda}D v_1\right)^2 + \left(m_1 \pm \frac{1}{\Lambda}(m_1v_1 +p)\right)^2 +\left(m_2 \pm \frac{1}{\Lambda}m_2v_1\right)^2 - \left(E\pm \frac{1}{\Lambda}m_1\right)^2\\
    &= \left(1\pm \frac{v_1}{\Lambda}\right)^2 \Gamma^2 \left(\rho^2 + p^2 - (\rho h - p)^2\right) + p^2 \left(\frac{1}{\Lambda^2}-1\right)\\
    &\leq \left(1- \frac{|v_1|}{\Lambda}\right)^2 \Gamma^2 \left(\rho^2 + p^2 - (\rho h - p)^2\right) + p^2 \left(\frac{1}{\Lambda^2}-1\right),\quad \text{by}~\eqref{eq:weak_taub2}\\
    &= \left(\frac{1}{\Lambda^2}-1\right) \left(\frac{s^2}{1-s^2}\right) \left(\rho^2 + p^2 - (\rho h - p)^2\right) + p^2 \left(\frac{1}{\Lambda^2}-1\right), \quad \text{using~\eqref{eq: Lambda2}}\\
    &= \left(\frac{1}{\Lambda^2}-1\right) \left(\frac{1}{1-s^2}\right) \left[s^2\left(\rho^2 - (\rho h - p)^2\right) + p^2 \right]\\
    &= \left(\frac{1}{\Lambda^2}-1\right) \left(\frac{1}{1-s^2}\right) \left(\frac{p^2}{\gamma - 1}\right)\left[\gamma - 1 -s^2 \left(\frac{1}{\gamma -1}+\frac{2\rho}{p}\right)\right],\quad \text{using}\ \eqref{eq: eos}\\
    &\leq \left(\frac{1}{\Lambda^2}-1\right) \left(\frac{1}{1-s^2}\right) \left(\frac{p^2}{\gamma - 1}\right)\left[1 - s^2\left( \frac{1}{\gamma -1}+ \frac{2\rho}{p}\right)\right],\quad\ \text{since}\ \gamma\in(1,2],\ s<1,\ \Lambda < 1\ \text{and using}\ \eqref{eq: Lambda3}\\
    &= \left(\frac{1}{\Lambda^2}-1\right) \left(\frac{1}{1-s^2}\right) \left(\frac{p^2}{\gamma - 1}\right)\left[\frac{1-2\gamma}{h}\right]\quad\text{using}~\eqref{eq: sound_speed2},\ \eqref{eq: eos}\\
    &< 0,\quad\ \text{since}\ \Lambda < 1,\ s<1,\ h>0\ \text{and}\ \gamma\in(1,2].
\end{align*}
Hence we have $q_\Lambda^{x\pm}>0$ and $D_\Lambda^{x\pm}>0$~\eqref{eq:ad.constraint.app} and so $\mb{u}_\Lambda^{x\pm}\in \Uad'$, that is, we have the admissibility of the $x$-directional quantities in equation~\eqref{eq:directional_evolutions_second}. Now, we will prove a lemma.
\begin{lemma}\label{lemma: big_lam}
    If $\mb{u}_\Lambda^{x\pm} \in \Uad'$ for some $\Lambda> 0$ then $\mb{u}_{\Lambda'}^{x\pm} \in \Uad'$ for all $\Lambda'>\Lambda$.
\end{lemma}
\begin{proof}
\begin{align*}
    \mb{u}_{\Lambda'}^{x\pm} &=\mb{u} \pm \frac{1}{\Lambda'} \mb{f}(\mb{u})\\
    &= \mb{u} \pm \frac{1}{\Lambda'} [\pm (\Lambda \mb{u}_\Lambda^{x\pm} - \Lambda \mb{u})]\\
    &= \mb{u} + \frac{1}{\Lambda'} (\Lambda \mb{u}_\Lambda^{x\pm} - \Lambda \mb{u})\\
    &= \left(1 - \frac{\Lambda}{\Lambda'}\right)\mb{u} + \frac{\Lambda}{\Lambda'}\mb{u}_\Lambda^{x\pm}
\end{align*}
Now using Lemma 3.3 of~\cite{wu2016physical}, we can see that $\mb{u}_{\Lambda'}^{x\pm} \in \Uad'$, since $0<\frac{\Lambda}{\Lambda'}<1$.
\end{proof}
Finally using Lemma~\ref{lemma: big_lam} and equation~\eqref{eq: Lambda_max},
\begin{align*}
    \mb{u}_{\Lambda_{i-\frac{1}{2},j}}^{n,x+}, \mb{u}_{\Lambda_{i+\frac{1}{2},j}}^{n,x-} \in \Uad'.
\end{align*}
That is, we have the admissibility of the $x$-directional quantities in equation~\eqref{eq:directional_evolutions_first}.

Similarly we can show for the admissibility of the $y$-directional quantities in equation~\eqref{eq:directional_evolutions_first}, that is
\[
\mb{u}_{\Lambda_{i,j-\frac{1}{2}}}^{n,y+}, \mb{u}_{\Lambda_{i,j+\frac{1}{2},j}}^{n,y-}\in \Uad'. 
\]
Now note that in equation~\eqref{eq: rusanov_coefficients}, $B_1$, $B_2$, $C_1$, and $C_2$ are always positive, since the quadrature weights and the spectral radii of the flux Jacobians are positive. Thus, the final theorem regarding the admissibility of the first-order scheme is as follows.
\begin{theorem}\label{theorem: admissility_rusanov_scheme}
    The numerical solution $\mb{u}_{ij}^{n+1}$ of the RHD equations~\eqref{RHD_equation}, at time $t_{n+1}$ and solution point $(\xi_i,\eta_j),\ 0<i,j<N$ using a first-order finite volume scheme with the Rusanov flux~\eqref{eq: rusanov_flux} belongs to the admissible region $\Uad'$ under the CFL type condition,
    \[
        \Delta t \left[\frac{\Lambda_{i-\frac{1}{2},j}+\Lambda_{i+\frac{1}{2},j}}{2 w_i \Delta x_p} + \frac{\Lambda_{i,j-\frac{1}{2}}+\Lambda_{i,j+\frac{1}{2}}}{2 w_j \Delta y_q}\right] < 1
    \]
    assuming $\mb{u}_{ij}^{n}, \mb{u}_{i\pm 1,j}^{n}, \mb{u}_{i,j\pm 1}^{n} \in \Uad'$. Here $w_i$'s are the corresponding quadrature weights.
\end{theorem}
\begin{remark}\label{remark: admissility_rusanov_scheme}
    In the low-order evolutions in equation~\eqref{eq: directional low-order evolutions}, the elements in the Rusanov flux are taken from the neighboring elements as well. Hence, considering the solutions at time $t_n$ are admissible in element $\Omega_{p,q}$ as well as in the neighboring elements, we can show in a similar way that the $x$-directional evolutions in equation~\eqref{eq: directional low-order evolutions} are admissible under a CFL type condition,
    \[
        \Delta t < \min\left\{\left(\frac{\Lambda_{p-\frac{1}{2},0}+\Lambda_{\frac{1}{2},0}^{pq}}{2 k_x w_0 \Delta x_p} \right)^{-1}, \left(\frac{\Lambda_{p-\frac{1}{2},0}+\Lambda_{N-\frac{1}{2},0}^{p-1,q}}{2 k_x w_N \Delta x_{p-1}} \right)^{-1} \right\}
    \]
    and the $y$-directional evolutions are admissible assuming,
    \[
        \Delta t < \min\left\{\left(\frac{\Lambda_{0,q-\frac{1}{2}}+\Lambda_{0,\frac{1}{2}}^{pq}}{2 k_y w_0 \Delta y_q} \right)^{-1}, \left(\frac{\Lambda_{0,q-\frac{1}{2}}+\Lambda_{0,N-\frac{1}{2}}^{p,q-1}}{2 k_y w_N \Delta y_{q-1}} \right)^{-1} \right\}
    \]
    where 
    \[
    \Lambda_{p-\frac{1}{2},0}= \max\left\{r\Big(\mb{f}'\big(\mb{u}_{N,0}^{p-1,q}\big)\Big), r\Big(\mb{f}'\big(\mb{u}_{0,0}^{p,q}\big)\Big)\right\}, \quad 
    \Lambda_{0,q-\frac{1}{2}}= \max\left\{r\Big(\mb{g}'\big(\mb{u}_{0,N}^{p,q-1}\big)\Big), r\Big(\mb{g}'\big(\mb{u}_{0,0}^{p,q}\big)\Big)\right\}
    \]
    and the superscripts $(\cdot)^{pq}$ denote the element indices.
\end{remark}
}

\bibliographystyle{abbrv}
\bibliography{reference}
\end{document}